\documentclass[11pt,a4paper,reqno]{amsart}
\usepackage[applemac]{inputenc}
\usepackage[T1]{fontenc}
\usepackage{amsmath}
\usepackage{amsthm}
\usepackage{amsfonts}
\usepackage{amssymb}
\usepackage{graphicx}
\usepackage{amsbsy}
\usepackage{mathrsfs}
\usepackage{url,color}
\usepackage{bbm}
\newtheorem{theorem}{Theorem}[section]
\newtheorem{lemma}[theorem]{Lemma}

\newtheorem{proposition}[theorem]{Proposition}

\theoremstyle{remark}
\newtheorem{remark}[theorem]{\it \bf{Remark}\/}

\numberwithin{equation}{section}
\catcode`@=11
\def\section{\@startsection{section}{1}%
  \z@{1.5\linespacing\@plus\linespacing}{.5\linespacing}%
  {\normalfont\bfseries\large\centering}}
\catcode`@=12
\newcommand{\be}{\begin{equation}}
\newcommand{\ee}{\end{equation}}
\newcommand{\bea}{\begin{eqnarray}}
\newcommand{\eea}{\end{eqnarray}}
\newcommand{\bee}{\begin{eqnarray*}}
\newcommand{\eee}{\end{eqnarray*}}

\def\pa{\partial}

\def\RR{\mathbb{R}}

\def\lab{\label}

\def\ba{\begin{array}}
\def\ea{\end{array}}

\def\non{\nonumber}
\def\Gammat{\widetilde{\Gamma}}

\def\Le{\mathcal L_{\text{ext}}}

\def\mut{\widetilde{\mu}}
\def\fref#1{{\rm (\ref{#1})}}

\catcode`@=11
\def\supess{\mathop{\operator@font Sup\,ess}}
\catcode`@=12
\def\Psit{\tilde{\Psi}}

\def\bt{\tilde{b}}

\def\RR{\mathbb{R}}

\def\e{\varepsilon}

\def\bar#1{{\overline #1}}
\def\fref#1{{\rm (\ref{#1})}}

\def\N{\mathcal N}

\def\R2+{\RR ^2_+}

\def\pa{\partial}

\def\lim{\mathop{\rm lim}}

\def\l{\lambda}

\def\log{{\rm log}}

\def\pa{\partial}

\def\Lamdba{\Lambda}

\def\pa{\partial}
\def\la{\langle}
\def\matchal{\mathcal}
\def\ra{\rangle}
\def\Mod{{\rm Mod}}
\def\NL{{\rm NL}}
\def\L{\mathcal L}

\begin{document}

\title[]{On strongly anisotropic type II blow up}
\author[C.Collot]{Charles Collot}
\address{Laboratoire J.A. Dieudonn\'e, Universit\'e de la C\^ote d'Azur, France}
\email{ccollot@unice.fr}
\author[F.Merle]{Frank Merle}
\address{LAGA, Universit\'e de Cergy Pontoise, France and IHES}
\email{merle@math.u-cergy.fr}
\author[P. Rapha\"el]{Pierre Rapha\"el}
\address{Laboratoire J.A. Dieudonn\'e, Universit\'e de la C\^ote d'Azur, France}
\email{praphael@unice.fr}

\begin{abstract} 
We consider the energy super critical $d+1$ dimensional semilinear heat equation $$\pa_tu=\Delta u+u^{p}, \ \ x\in \Bbb R^{d+1}, \ \ p\geq 3, \ d\geq 14.$$ A fundamental open problem on this canonical nonlinear model is to understand the possible blow up profiles appearing after renormalization of a singularity. We exhibit in this paper a new scenario corresponding to the first example of strongly anisotropic blow up bubble: the solution displays a completely different behaviour depending on the considered direction in space. A fundamental step of the analysis is to solve the {\it reconnection problem} in order to produce finite energy solutions which is the heart of the matter. The corresponding anistropic mechanism is expected to be of fundamental importance in other settings in particular in fluid mechanics. The proof relies on a new functional framework for the construction and stabilization of type II bubbles in the parabolic setting using energy estimates only, and allows us to exhibit new unexpected blow up speeds.
\end{abstract}

\maketitle

%%%%%%%%%%%%%%%%%%%%%%%%%%%%%%%%%%%%%%%%%%%%%%%%%%%%%%%%
%%%%%%%%%%%%%%%%%%%%%%%%%%%%%%%%%%%%%%%%%%%%%%%%%%%%%%%%

\section{Introduction}

%%%%%%%%%%%%%%%%%%%%%%%%%%%%%%%%%%%%%%%%%%%%%%%%%%%%%%%%
%%%%%%%%%%%%%%%%%%%%%%%%%%%%%%%%%%%%%%%%%%%%%%%%%%%%%%%%

%%%%%%%%%%%%%%%%%%%%%%%%%%%%%%%%%%%%%%%%%%%%%%%%%%%%%%%%

\subsection{On blow up profiles}

%%%%%%%%%%%%%%%%%%%%%%%%%%%%%%%%%%%%%%%%%%%%%%%%%%%%%%%%

We consider in this paper the nonlinear heat equation
\be
\label{heatnonlin}
\left\{\begin{array}{ll} \pa_tu=\Delta u+u|u|^{p-1},\\u_{|t=0}=u_0
\end{array}\right.  \ \ (t,x)\in\Bbb R^+\times \Bbb R^{n}, \ \ p>1, \ \ n\geq 1.
\ee 
Smooth well localized initial data $u_0$ yield unique local in time solutions \cite{BrCa,We} which dissipate the total energy of the flow $$\frac{d}{dt}E(u)\leq0, \ \ E(u)=\frac12\int_{\Bbb R^n}|\nabla u(t,x)|^2dx-\frac1{p+1}\int_{\Bbb R^n}|u(t,x)|^{p+1}dx.$$ We are interested in this paper on the singularity formation problem in the so called energy super critical range $$s_c=\frac n2-\frac 2{p-1}>1.$$ Since the pioneering works by Giga and Kohn \cite{Gi,Gi1,Gi2,Gi3, Gi4}, a fundamental problem is to understand which nonlinear structures emerge after renormalization of the flow for {\it finite energy} singularities. More precisely, let the self similar renormalization 
\be
\label{selfisimialrvariabels}
u(t,x)=\frac{1}{(T-t)^{\frac{1}{p-1}}}v(\tau,y), \ \ \tau=-\log (T-t), \ \ y=\frac{x}{\sqrt{T-t}}
\ee which maps \eqref{heatnonlin} onto the renomalized flow 
\be
\label{renormalizedflow}
\pa_\tau v=\Delta v-\frac12\Lambda v+v^p, \ \ \Lambda v=\frac{2}{p-1}v+y\cdot\nabla v,
\ee
then two classes of scenario have been understood so far.\\

\noindent{\em Type I blow up bubbles} correspond to global in time solutions to  \eqref{renormalizedflow} which to leading order are given by a smooth stationary self similar solution
\be\label{stationaryselfsim}
\Delta v-\frac 12\left(\frac{2}{p-1}v+y\cdot\nabla v\right)+v|v|^{p-1}=0.
\ee
This elliptic problem \eqref{stationaryselfsim} always admits the constant in space solution $
\kappa=\left(\frac{1}{p-1}\right)^{\frac 1{p-1}}$ which has been proved to generate a stable blow up dynamics \cite{Gi,Gi1,Gi2,Gi3, Gi4, BK, MZduke,MeZa2,MCRcritique} emerging from finite energy initial data for $s_c\le 1$. Other classes of radially symmetric profiles for $s_c>1$ with the boundary condition $\lim_{|y|\to +\infty} v(y)=0$ have been constructed using ODE techniques \cite{lepin,troy,buddone,buddselfsim} or a direct bifurcation argument \cite{bizon, CRS}, and they have been shown to be nonlinearily finite codimensionally stable within a suitable class of finite energy initial data, \cite{CRS}. All these solutions have the self similar blow up speed $$\|u(t)\|_{L^\infty}\sim \frac{1}{(T-t)^{\frac{1}{p-1}}}.$$

\noindent{\em Type  II blow-up bubbles}.  Let the Joseph-Lundgren exponent \be
\label{exponentpjl}
p_{JL}(n)=\left\{\begin{array}{ll} +\infty\ \ \mbox{for}\ \ n\leq 10,\\
1+\frac{4}{d-4-2\sqrt{d-1}}\ \ \mbox{for}\ \ n\geq 11,
\end{array}\right.
\ee
then for $p>p_{JL}$, new type II blow-up solutions $$\text{lim}_{t\to T} \ \|u(t)\|_{L^{\infty}}(T-t)^{\frac{1}{p-1}}=+\infty$$ appear in the radial setting as threshold dynamics at the boundary of the ODE blow up set, \cite{MaMe3}, and dynamical proofs were proposed in \cite{HV,Mizo}. Their construction has been revisited in the setting of dispersive Schr\"odinger and wave equations,  \cite{MRR,Co}, geometric models \cite{RaphRod, MRRinvent,RSc}, and for the non radial heat equation, \cite{Co2}. The structure of the singularity is deeply related to the {\it singular} self similar solution \be
\label{defphistar}\Phi^*(r)=\frac{c_\infty}{r^{\frac{2}{p-1}}}, \ \ c_\infty(p,n)=\left[\frac{2}{p-1}\left(n-2-\frac{2}{p-1}\right)\right]^{\frac{2}{p-1}}
\ee 
and its smooth regularization at the origin given by the soliton profile 
\be
\label{eqsoliton}
\left\{\begin{array}{ll}Q''+\frac{d-1}{r}Q'+Q^p=0,\\Q(0)=1,\ \ Q'(0)=0.
\end{array}\right.
\ee The blow up speed is given by a universal countable sequence $$ \|u(t)\|_{L^\infty}\sim \frac{1}{ (T-t)^{\frac 2{p-1}\frac{\ell}{\alpha(d)}}}, \ \ \ell\in \Bbb N, \ \ \ell>\frac \alpha 2$$ where $\alpha=\alpha(d,p)>2$  given by \eqref{defalpha} is one of the key phenomenological number of the super critical numerology, and these rates are known to be the only possibility for radially symmetric data, \cite{Mizo}.

%%%%%%%%%%%%%%%%%%%%%%%%%%%%%%%%%%%%%%%%%%%%%%%

\subsection{Anisotropic blow up bubbles}

%%%%%%%%%%%%%%%%%%%%%%%%%%%%%%%%%%%%%%%%%%%%%%%

Despite substantial efforts, the above known nonlinear structures all correspond to isotropic bubbles with a simple geometry of the blow up sets. In various classical nonlinear models, {\it anisotropy} and the possiblity of completly different behaviours in the various directions in space is expected to be an essential feature of the problem. This is typically the case for the anisotropic nonlinear Schr\"odinger equation $$i\pa_tu+\pa_{xx}u-\pa_{yy}u+u|u|^{p-1}=0, \ \ (x,y)\in \Bbb R^2$$ which arises from classical fluid mechanics models \cite{SS}, and for which the blow up problem is completely open. More generally, the question here is to understand how dramatically unbounded solutions to the {\it time dependent} renormalized flow \eqref{renormalizedflow} may correspond to finite energy singularities. This problem is the {\it reconnection problem} and is the heart of the forthcoming analysis. Let us stress that this is mostly an unexplored domain, and that even formal predictions are to our knowledge unclear.\\

We solve in this paper the anisotropic reconnection problem in a canonical situation. Consider a space dimension $$n=d+1,$$ then any radially symmetric solution $U(t,r)$, $r\in \mathbb R^d$, to the $d$ dimensional nonlinear heat equation provides a solution to the $d+1$ model with cylindrical symmetry, $x=(r,z)\in \mathbb R^d\times \mathbb R$, by considering its natural lift $u(t,r,z)=U(t,r)$. This solution corresponds to a {\em line} singularity along the additional $e_z$ axis, but has infinite energy for the $\Bbb R^{d+1}$ model. We claim that given $U(t,r)$ a type II blow up bubble for the $d$-dimensional radially symmetric problem, we may {\em solve the reconnection problem} for the additional $z$ direction and produce {\em finite energy} $d+1$ dimensional initial data which after renormalization are very elongated along the $e_z$ direction, and eventually reconnect $U(t,r)$ to a decreasing profile in the $z$ direction. The associated reconnection profile is {\it universal and rigid}. This eventually produces a  point and not a line singularity but with an elongated  pancake like profile in self similar variables and new blow up rates.

\begin{theorem}[Existence and finite codimensional stability of anistropic blow up bubbles]
\label{thmmain}
Let $\alpha=\alpha(d,p)$, $\Delta=\Delta(d,p)$ be the super critical numbers given by \eqref{defgamma}, \eqref{defalpha},  and assume: 
\be
\label{confitiondiemsnion}
d\geq 11,\ \  p\geq \max\{3,p_{JL}(d)\}, \ \  \sqrt{\Delta}>2.
\ee 
Pick $$\ell\in\Bbb N^*\ \ \mbox{with}\ \ \ell>\frac \alpha 2.$$ Then there exists a finite codimensional set of initial data $u_0\in \mathcal C^{\infty}_c(\Bbb R^{d+1},\Bbb R)$ with cylindrical symmetry such that the corresponding solution $u(t,x)$ to \eqref{heatnonlin} with $n=d+1$ blows up in finite time $0<T<+\infty$ with the following asymptotics. The solution admits on $[0,T)$ in self similar variables \eqref{selfisimialrvariabels} a decomposition $$ v(t,r,z)=\frac{1}{D(t,z)^{\frac{2}{p-1}}}Q\left(\frac{r}{D(t,z)}\right)+V(t,r,z)$$ where $Q$  denotes the $d$-dimensional smooth radially symmetric soliton profile \eqref{eqsoliton}, and with the following sharp description:\\
\noindent{\em 1. Computation of the reconnection}: there holds 
\be
\label{functionreconnction}
D(t,z)=\sqrt{b(t)}(1+a(t)P_{2\ell}(z))^{\frac 1\alpha}
\ee where $P_{2\ell}(z)$ is the $2\ell$-th one dimensional Legendre polynomial given by \eqref{hermitepolynomial}, and $(a,b)\in \mathcal C^1([0,T),\Bbb R^*_+)$ with the sharp asymptotics near blow up time:
\bea
\label{loideb}
&&b(t)=b^*(1+o_{t\to T}(1))(T-t)^{\frac{2\ell-\alpha}{\alpha}}, \ \ 0<b^*(u_0),\\
\label{loidea}
&&a(t)=a^*(1+o_{t\to T}(1)), \ \ 0<a^*(u_0)\ll1.
\eea
\noindent{\em 2. Soliton profile and blow up speed}: 
\be
\label{linftydecay}
\lim_{t\to T}(\sqrt{b})^{\frac{2}{p-1}}\|V(t,\cdot)\|_{L^{\infty}}=0
\ee
and 
\be
\label{cneoneon}
\|u(t,\cdot)\|_{L^{\infty}}=\frac{c(u_0)(1+o(1))}{(T-t)^{\frac 2{p-1}\frac{\ell}{\alpha}}}, \ \ c(u_0)>0.
\ee
\end{theorem}

\noindent{\em Comments on the result}.\\

\noindent{\em 1. New $d+1$-dimensional blow up rates.} The blow up speed \eqref{cneoneon} is new and unexpected, and shows that all the type II blow up rates of the $d$-dimensional problem are admissible blow up speeds for  the $d+1$ dimensional problem. This also shows that the classification of all type II blow up speeds for {\em radially symmetric data} obtained in \cite{Mizo2} using maximum principle like arguments {\em no longer holds} for non radially symmetric data. The method clearly designs an iteration process for this dimensional reduction procedure. The fact that the solutions described by Theorem \ref{thmmain} correspond to a smooth finite codimensional manifold of initial data can be addressed similarly as in \cite{Co}. The extension of this result to the energy critical case and the construction of new non radial type II blow up bubbles below the Joseph Lundgren exponent is work in progress.\\

\noindent{\em 2. Structure of the reconnection profile.} In the companion paper \cite{MRS1}, we address a similar result in the context of type I blow up with decreasing at infinity self similar profile. The analysis of type I blow up is simpler, and the reconnection profile is universal 
\be
\label{reconnectionautosim}
D(t,z)\sim \sqrt{1+b(t)z^2}, \ \ b(t)\sim \sqrt{\log(T-t)}
\ee which is reminiscent to the stability of the ODE type I blow up \cite{BK,MZduke}. The structure of the reconnection profile \eqref{functionreconnction} of type II blow up is more complicated and the associated moving free boundary $r=D(t,z)$ locating the region where the singularity is large has a non trivial geometrical description. This shapes follows from an {\em all order} algebraic cancellation, Lemma \ref{lemmacancellationa}, related to the fact that the parameter $a$ is nearly constant in time according to \eqref{loidea}, and this was very much unexpected.\\

\noindent{\em 3. Assumption \eqref{confitiondiemsnion}}. The optimal range of exponents for the existence of radially symmetric type II blow up in $\Bbb R^d$ is $p>p_{JL}(d)$ which is equivalent to $\Delta>0$. However $p_{JL}(d)\to 1$ as $d\to +\infty$ and this causes lack of differentiability of the nonlinearity. From direct check, the assumption \eqref{confitiondiemsnion} is automatically satisfied for $p\geq 3$, $d\geq 14$ and it will avoid additional technicalities when $\Delta$ is small.\\

The heart of the proof of Theorem \ref{thmmain} is the computation of the free boundary \eqref{functionreconnction} which is the zone where the soliton $Q$ profile dominates, and the development of a suitable functional setting to control the flow for solutions which after renormalization are nearly constant in $z$ for $|z|\leq z^*(t)$, $z^*(t)\to +\infty$ as $t\to T$. A similar issue occurs for the study of the type I ODE blow up where the blow up profile is given by the constant profile $\kappa$. The analysis developed in \cite{BK,MZduke} amounts to first the formal derivation of the reconnection profile, and then the control of suitable $L^{\infty}$ bounds for the perturbation. This last step is essential and sees the reconnection procedure, and is performed in \cite{BK,MZduke} using propagator estimates for the linearized flow close to $\kappa$ {\it which is explicit}, and in \cite{MZduke} Liouville type classification theorems to rule out some possible growth scenario.\\ 

This approach seems hardly applicable in the setting of type II blow up bubbles, and the analysis of the proof of Theorem \ref{thmmain} requires two main new inputs.\\

\noindent{1. New approach to type II blow up}. First we propose a {\it new functional framework} for the study of type II blow up bubbles which relies on a Lyapounov type bifurcation argument to study the flow near the soliton $Q$ seen as a {\em non compact} perturbation of the singular profile $\Phi^*$. Here we clearly adopt a point view which is already central in the pioneering breakthrough work \cite{HV}, but the main novelty is that the bifurcation argument allows us to construct the {\em exact time dependent} spectral basis needed for the analysis which avoids complicated matching procedures, and trivializes the computation of modulation equations and the derivation of the type II blow up speeds. Such an approach was implemented for the first time in \cite{HR} for the study of the Stefan melting problem which formally corresponds to a zero soliton case. The computation of the eigenvectors, Proposition \ref{spectraltheorem}, involves a universal sequence of functions, see \eqref{id Ti}, which is already central in the {\it tail computation} developped in \cite{MRR} and makes the link between the approaches developped in \cite{HV} and \cite{MRR}.\\ 

\noindent{2. $L^{\infty}$ bounds}. Once the flow is controlled in suitable weighted norms, it remains to close the nonlinear term using $L^{\infty}$ bounds. For the $d$-dimensional parabolic problem, the difficulty is at the origin $r=0$ and this can be done in various ways using for example an elementary maximum principle like argument \cite{BSeki}, or a direct brute force energy method. Treating the full cylindrical problem is more complicated. We first need to compute the free boundary $D(t,z)$, and here we rely on the spectacular algebra \eqref{formulapsi}. Once the reconnection is computed, we derive $L^{\infty}$ bounds from $W^{1,q}$ energy estimates for the linearized flow which are particularly efficient both at the origin and infinity in space. Here we use the parabolic structure again and the repulsive nature of the linearized operator close to $Q$  {\em when measured in suitable weighted norms}, section \ref{linftybound}.\\

Hence the proof designs a new route map for the  study type II blow up bubbles which uses in an optimal way the parabolic structure of the problem with respect to the pioneering works \cite{RSc,Co2}. The parabolic structure allows us to work with self adjoint operators and energy estimates which considerably simplify the analysis, but the essence of the argument which relies on energy estimates only may in principle be propagated to more complicated dispersive Schr\"odinger or wave like problems.

\subsection*{Acknowledgements}  C.C. and P.R. are supported by the ERC-2014-CoG 646650 SingWave. Part of this work was done while all authors were attending the IHES program 2016 on nonlinear waves, and they wish to thank IHES for its stimulating hospitality.

\subsection*{Notations} We let $$Y=(y,z)\in \Bbb R^{d+1}, \ \ r=|y|=\left(\sum_1^d y_i^2\right)^{\frac 12}$$ and define the japanese bracket: $$\la z\ra =\sqrt{1+z^2}.$$ We say that $u(Y)$ has even cylindrical symmetry if $$u(Y)=u(r,z)=u(r,-z).$$ We note the generator of scaling $$\Lambda=\Lambda_r+z\pa_z, \ \ \Lambda_r=\frac{2}{p-1}+r\pa_r.$$ We let the 
roots of the quadratic equation
\be
\label{defequgamma}
\gamma(d-2-\gamma)=pc_\infty^{p-1}.
\ee 
be
\be
\label{defgamma}
\gamma=\frac{d-2-\sqrt{\Delta}}{2}, \ \ \Delta=\left(\frac{d-2}{2}\right)^2-pc_\infty^{p-1}>0 \ \ \mbox{for}\ \ p>p_{JL}(d)
\ee
and 
\be
\label{defgammtwtp}
\gamma_2=d-2-\gamma>\gamma.
\ee
We let for an arbitrary constant $\epsilon>0$
\be
\label{defalpha}
\alpha=\gamma-\frac{2}{p-1}>2, \ \ g=\min\{\alpha,\sqrt{\Delta},2\}-\epsilon=2-\epsilon
\ee
because of the assumption \fref{confitiondiemsnion}, meaning that $g$ is arbitrarily close to $2$. We let $Q(r)$ be the unique radially symmetric solution to $$\left|\begin{array}{ll}Q''+\frac{d-1}{r}Q'+Q^p=0\\ Q(0)=1,\ \ Q'(0)=0
\end{array}\right.$$ and recall the following properties \cite{Li,KaSt}. $p>p_{JL}$ implies
\be
\label{positivitylamdbaq}
0<Q(r)<\Phi^*(r)=\frac{c_\infty}{r^{\frac{2}{p-1}}}, \ \ \Lambda Q(r)>0
\ee
where $\Phi^*,c_\infty$ are given by \eqref{defphistar}. At infinity, 
\be
\label{asyptotq}
Q(r)=\Phi^*(r)-\frac c{r^\gamma}+O\left(\frac 1{r^{\gamma+\delta }}\right), \ \ \text{as} \ r\rightarrow +\infty, \ c>0, \ \delta \geq g, 
\ee
which propagates for derivatives, where $0<\delta<\text{min}(\sqrt{\Delta},\alpha)$ is a constant arbitrarily close to $\text{min}(\sqrt{\Delta},\alpha)$, and there holds in particular the fundamental cancellation: 
\be
\label{id LambdaQ}
\Lambda Q(r)=\frac{c}{r^{\gamma}}+O\left(\frac{1}{r^{\gamma+\delta}}\right)=\frac{c}{r^{\gamma}}+O\left(\frac{1}{r^{\gamma+g}}\right) \ \ \text{as} \ r\rightarrow +\infty, \ c\neq 0.
\ee
Given $b>0$, we define 
\be
\label{defqbÙ}
Q_b(r)=\frac{1}{(\sqrt{b})^{\frac2{p-1}}}Q\left(\frac{r}{\sqrt{b}}\right),
\ee the linearized operators 
\be
\label{linearizedoperators}
\matchal L_b=-\Delta_Y+\frac 12\Lambda -pQ_b^{p-1}, \ \ H_b=-\Delta_r+\frac 12\Lambda_r -pQ_b^{p-1}
\ee
and the weights 
\be
\rho_Y(Y)=\frac{1}{2^d\pi^{d-\frac{1}{2}}}e^{-\frac{|Y|^2}{4}}, \ \ \rho_r(r)=e^{-\frac{r^2}{4}}, \ \ \rho_z=\frac{1}{2\sqrt{\pi}}e^{-\frac{z^2}{4}}
\ee 
so that for $f:\mathbb R^{d+1}\rightarrow \mathbb R$ with cylindrical symmetry:
$$
\int_{Y\in \mathbb R^{d+1}}f(Y)\rho_YdY=\int_{r=0}^{+\infty}\int_{z\in \mathbb R} f(r,z)r^{d-1}\rho_r\rho_zdrdz.
$$
We denote by $L^2(\rho_Y)$, $L^2(\rho_r)$ the associated spaces of square integrable functions with scalar products $\la \cdot \ra_{L^2_{\rho_Y}}$ and $\la \cdot \ra_{L^2_{\rho_r}}$. We define 
\be
\label{defH}
H=-\Delta_r-pQ^{p-1}
\ee
the linearized operator close to Q. We let 
\be
\label{hermitepolynomial}
P_m(z)=c_m\sum_{k=0}^{[\frac m2]}\frac{m!}{k!(m-2k)!}(-1)^kz^{m-2k}
\ee
 be the $m$-th one dimensional Hermite polynomial which solves 
 \be
 \label{equationhermite}-\pa_z^2 P_m+\frac 12z\pa_z P_m=\frac m2 P_m, \ \ m\in\Bbb N
 \ee
 and the normalization (implying $P_0=1$)
 \be
 \label{orthohonermite}
 (P_m,P_{m'})_{L^2_{\rho_z}}=\delta_{mm'}.
 \ee
All along the paper, we fix once and for all an integer $$\ell>\frac{\alpha}{2}$$
and the associated set of indices
\be \lab{def mathcal I}
\mathcal I :=\left\{ (j,k)\in \mathbb N^2, \ \ 0\leq j\leq \ell, \ \  \left|\begin{array}{ll} 1\leq k\leq \ell-1\ \ \mbox{for}\ \ j=0\\ 0\leq k\leq \ell-j\ \ \mbox{for}\ \ 1\leq j\le\ell \end{array}\right. \right\}
\ee
 A quantity $b^\delta$ denotes a small gain with a constant $\delta=\delta(\ell)>0$ universal and small enough, independent of any of the constants that will appear in a central bootstrap argument. In the whole paper, we use the notation $$\eta(a)=o_{a\to 0}(1).$$  For $x\in \mathbb R$ and $k\in \mathbb N$ we use the notation 
 \be
 \label{ncenceononce}
(x)_0=1, \ \ (x)_k=x\times (x+1)\times ... \times (x+k-1) \ \ \text{for} \ k\geq 1.
\ee

%%%%%%%%%%%%%%%%%%%%%%%%%%%%%%%%%%%%%%%%%%%%%%%%%%%%%%%%
%%%%%%%%%%%%%%%%%%%%%%%%%%%%%%%%%%%%%%%%%%%%%%%%%%%%%%%%

\section{A perturbative spectral theorem}

%%%%%%%%%%%%%%%%%%%%%%%%%%%%%%%%%%%%%%%%%%%%%%%%%%%%%%%%
%%%%%%%%%%%%%%%%%%%%%%%%%%%%%%%%%%%%%%%%%%%%%%%%%%%%%%%%

We start the analysis of the renormalized flow \eqref{renormalizedflow} by the diagonalization of the linearized operator closed to a concentrated soliton given by \eqref{linearizedoperators}. This approach, initiated in \cite{HR} for the Stefan problem, is the conceptual link between the pioneering approach Herrero and Velazquez \cite{HR} and the soliton approach of Merle, Rapha\"el, Rodnianski \cite{MRR}. This will produce an elementary functional framework to compute the blow up speed, and will allow us to control the full flow in the anisotropic geometry.

%%%%%%%%%%%%%%%%%%%%%%%%%%%%%%%%%%%%%%%%%%%%%%%%%%%%%%

\subsection{Reduction of the problem}

%%%%%%%%%%%%%%%%%%%%%%%%%%%%%%%%%%%%%%%%%%%%%%%%%%%%%%

We first recall that the spectrum of the linearized operator $H_\infty$ close to $\Phi^*$ is explicit.

\begin{proposition}[Diagonalization of $H_\infty$, \cite{HV}]
\label{propinfty}
Assume $p>p_{JL}$. There exists a domain $\mathcal D\subset H^1_{\rho_r}$ with $H^2_{\rho_r}\subset \mathcal D$ such that $H_\infty:\mathcal D \rightarrow L^2(\rho_r)$ given by $$H_\infty=-\Delta+\frac 12\Lambda -\frac{pc_\infty^{p-1}}{r^2}$$ is self adjoint with compact resolvant. The spectrum in the radial sector is given by $$\lambda_{i,\infty}=i-\frac\alpha 2, \ \ i\in \Bbb N$$ with eigenfunctions 
$$
\phi_{i,\infty}(r)=\frac{n!}{\left(\frac d2 -\gamma\right)_n} \frac{L_i^{(\frac d2-\gamma-1)}\left(\frac{r^2}{2}\right)}{r^{\gamma}}=\sum_{j=0}^i c_{i,j}C_jr^{2j-\gamma}
$$
where $L$ is the modified Laguerre polynomial, and $c_{i,j}$ and $C_j$ are defined by \fref{id cij} and \fref{eq:relation Cni}. Moreover, there holds the spectral gap estimate: $\forall u\in H^1_{\rho_r}$ with radial symmetry, \be
\label{asyptoticgap}
\la u,\phi_{i,\infty} \ra_{L^2_{\rho_r}}=0, \ \ 0\leq i\leq \ell \ \Rightarrow \ \ \la H_\infty u,u \ra_{L^2_{\rho_r}}\geq \l_{\ell+1,\infty}\| u\|_{L^2_{\rho_r}}^2.
\ee.
\end{proposition}

We claim that this allows us to diagonalize the operator $H_b$ given by \eqref{linearizedoperators} in the radial sector of $R^d$ in the perturbative regime $0<b\ll1$.

\begin{proposition}[Partial diagonalisation of $H_b$] 
\label{spectraltheorem}
Assume \eqref{confitiondiemsnion}. There exists a universal domain $\mathcal D\subset H^1_{\rho_r}$ with $H^2_{\rho_r}\subset \mathcal D$ such that for $b>0$, $H_b:\mathcal D \rightarrow L^2(\rho_r)$ is self adjoint with compact resolvant. Moreover, for all $\ell\in \mathbb N$, there exists $C(\ell),c(\ell)>0$ and $0<b^*(\ell)\ll1$ such that for all $0<b\leq b^*(\ell)$, the following holds:\\
\noindent{\em 1. Eigenvalue computation}: the $i$-th eigenvalue of $H_b$, $0\leq i \leq \ell$, is:
\be \lab{id lambdaib}
\lambda_{i,b}=i-\frac{\alpha}{2}+\tilde{\lambda}_{i,b}, \ \  \text{with} \ |\tilde{\lambda}_i|\leq C(\ell)b^{\frac{g}{2}},
\ee
and it is associated to the eigenfunction
\be
\label{calculeignefunction}
\phi_{i,b}=\sum_{j=0}^i c_{i,j} \sqrt{b}^{2j-\gamma}T_j\left( \frac{r}{\sqrt{b}}\right) +\tilde{\phi}_{i,b}, \ \  \text{with} \ \|\tilde{\phi}_{i,b}\|_{H^1_{\rho_r}} \leq C(\ell)b^{\frac g2} 
\ee
where the coefficients $c_{i,j}$ are defined by \fref{id cij} and $(T_i)_{i\in \mathbb N}$ is defined by \fref{id Ti}.\\
\noindent{\em 2. Estimates}: there holds
\be\label{estproximity}
\|\phi_{i,b}-\phi_{i,\infty}\|_{H^1_{\rho_r}}\leq C(\ell)b^{\frac g2}.
\ee
and the more precise pointwise bounds: for $k=0,1,2$,
\bea
\label{poitwisepsijk}
&&|\pa_r^k\phi_{i,b}|\lesssim \frac{\la r\ra^{2i+4}}{(\sqrt{b}+r)^{\gamma+k}}, \ \ |\pa_r^kb\pa_b\phi_{i,b}|\lesssim \frac{b^{\frac g2}}{(\sqrt b+r)^g}\frac{\la r\ra^{2i+4}}{(\sqrt{b}+r)^{\gamma+k}}.\\
\label{bd pointwise tildephi}
&&|\pa_r^k\tilde{\phi}_{i,b}|+|\pa_r^kb\pa_b\tilde{\phi}_{i,b}|\lesssim \frac{b^{\frac g2}(1+r)^{2i+4}}{(\sqrt{b}+r)^{\gamma+k}}.
\eea
Moreover:
\be \label{bound:partialb}
|b\partial_b \tilde{\lambda}_{i,b}|\leq C(\ell)b^{\frac g 2 }, \ \ \ \parallel b\partial_b \tilde{\phi}_{i,b}\parallel_{H^1_{\rho}}\leq C(\ell)b^{\frac g 2 }.
\ee
\noindent{\em 3. Spectal gap estimate}: there exists $c(\ell)>0$, $\forall u\in H^1_{\rho_r}$ with radial symmetry, 
\be
\label{spectralgaptotalone}
 (u,\phi_{i,b})_{L^2_{\rho_r}}=0, \ \ 0\leq i\leq \ell \ \Rightarrow \ \ ( H_bu, u)_{L^2_{\rho_r}} \geq \l_{\ell,b}\parallel u\parallel_{L^2_{\rho_r}}^2+c(\ell)\parallel u\parallel_{H^1_{\rho_r}}^2.
\ee
\end{proposition}

Since the potential term $Q_b(r)$ is independent of $z$, the diagonalization of the full linearized operator $\L_b$ for even in $z$ cylindrical functions directly follows from Proposition \ref{spectraltheorem} and an elementary tensorial claim.

\begin{proposition}[Partial diagonalization of $\L_b$]
\label{diaglb}
Assume \eqref{confitiondiemsnion}. There exists a universal domain $\mathcal D\subset H^1_{\rho_Y}$ with $H^2_{\rho_Y}\subset \mathcal D$ such that for all $b>0$, the operator $\L_b:\mathcal D \rightarrow L^2(\rho_Y)$ is self adjoint with compact resolvant. Moreover, for all $\ell\in \mathbb N$, there exists $C(\ell),c(\ell)>0$ and $b^*(\ell)>0$ such that for all $0<b<b^*(\ell)$ and $0\leq i \leq \ell$, the following holds. The function 
$$\psi_{i,k}(r,z)=\phi_{i,b}(r)P_{2k}(z), \ \ k\in\Bbb N$$ with $P_{2k}$ given by \eqref{hermitepolynomial} is an eigenfunction 
\be
\label{valrupropre}
\mathcal L_b\psi_{i,k}=\lambda_{i,k}\psi_{i,k}, \ \ \lambda_{i,k}=i+k-\frac\alpha 2+\tilde{\l}_{j,b},
\ee and for all $u\in H^1_{\rho_Y}$ with even cylindrical symmetry satisfying $$(u,\psi_{i,k})_{L^2_{\rho_Y}}=0, \ \ 0\leq i\leq \ell, \ \ 0\leq k\leq \ell-i,$$ there holds:
\be
\label{spectralgaptotal}
(\mathcal L_bu,u)_{L^2_{\rho_Y}}\geq \l_{\ell,0}\|u\|_{L^2_{\rho_Y}}^2+c(\ell)\|u\|_{H^1_{\rho_Y}}^2.
\ee
\end{proposition}

\begin{proof} Let $$\L_\infty=-\Delta_Y+\frac 12\Lambda -p(\Phi^*)^{p-1},$$ let $k\in \Bbb N$ and $P_{2k}$ be the $k$-th even Hermite polynomial \eqref{hermitepolynomial}. Let $$\psi_{i,k,\infty}(r,z)=\phi_{i,\infty}(r)P_{2k}(z),$$ then since $Q(r)$ is independent of z, we compute from Proposition \ref{propinfty} and \eqref{equationhermite}:
$$\L_{\infty}\psi_{i,k,\infty}=P_{2k}(z)H_b(\phi_{i,\infty})+\phi_{i,\infty}\left[-\pa_z^2+\frac 12z\pa_z\right]P_{2k}=\left(i+k-\frac \alpha 2\right)\psi_{i,k,\infty}.$$ Since from standard tensorial claim the family $(\psi_{i,k,\infty})_{i,k\ge 0}$ is total in $H^1_{\rho_Y}$ restricted to functions with even cylindrical symmetry, we obtain the spectral gap estimate: $$(u,\psi_{i,k,\infty})_{L^2_{\rho_Y}}=0, \ \ 0\leq i\leq \ell, \ \ 0\leq k\leq \ell-i,$$  with even cylindrical symmetry implies :
\be
\label{spectralgaptotalbis}
(\mathcal L_{\infty}u,u)_{L^2_{\rho_Y}}\geq \left(\ell+k+1-\frac \alpha 2\right)\|u\|_{L^2_{\rho_Y}}^2.
\ee
Similarly from Proposition \ref{spectraltheorem}, $$\L_{b}\psi_{i,k}=\left(\l_{i,b}+k\right)\psi_{i,k,\infty}$$ and the uniform closeness estimate \eqref{estproximity} injected into \eqref{spectralgaptotalbis} now implies \eqref{spectralgaptotal}.
\end{proof}

The rest of this section is devoted to a sketch of proof of Proposition \ref{spectraltheorem} using a direct brute force matching ODE approach, see for example \cite{Wei,CRS} for related approaches.

%%%%%%%%%%%%%%%%%%%%%%%%%%%%%%%%%%%%%%%%%%%%%%%%%

\subsection{Interior problem} 

%%%%%%%%%%%%%%%%%%%%%%%%%%%%%%%%%%%%%%%%%%%%%%%%%

We construct in this subsection eigenfunctions of $H_b$ in the zone $0\leq r \leq r_0\ll 1$. We zoom on the soliton core by introducing $$y=\frac{r}{\sqrt b}, \ \ v(r)=u(y), \ \  H_bv(r)=\frac 1b(H+\frac b2\Lambda )u(y)$$ with $H$ given by \eqref{defH}, and study the Schr\"odinger operator $H+\frac b2\Lambda$ in the zone $0\leq y \leq y_0$ for 
\be
\label{defyzero}
1\ll y_0:=\frac{r_0}{\sqrt b}\ll \frac{1}{\sqrt{b}}
\ee 
restricted to radially symmetric functions. Let us recall from Lemmas 2.3 and 2.10 in \cite{Co2}, see also \cite{MRR}, the description of the iterates of the kernel of $H$:

\begin{lemma}[Generators of the generalized kernel of $H$] \lab{lem:Ti} There exists a family of smooth radial functions $(T_i)_{i\in \mathbb N}$ satisfying the identities
\be \lab{id Ti}
T_0=c\Lambda Q, \ c\in\mathbb R, \ c\neq 0, \ \ H T_0=0, \ \ H T_{i+1}=-T_i
\ee
with the expansion 
\be \lab{id expansion Ti}
\left|\ba{l l}
T_i \underset{y\rightarrow 0}{=} \sum_{l=0}^q d_{i,l} y^{2i+2l}+O(y^{2+2q}), \ \ \text{for} \ \text{all} \ q\in \mathbb N, \\
T_i \underset{y\rightarrow +\infty}{=} C_i y^{-\gamma+2i}+O(y^{-\gamma+2i-g}) . \\
\ea\right.
\ee
where
\be \lab{eq:relation Cni}
C_i=\frac{2^{-2i}}{i!\left(\frac d2 -\gamma \right)_i},
\ee
and the additional cancellation at infinity:
\be \lab{eq:Thetani}
\Lambda T_i=(2i-\alpha )T_i+\Theta_i, \ \ \Theta_i=O(\la y\ra^{-\gamma+2i-g}).
\ee
\end{lemma}

These profiles provide the {\it non perturbative} leading order term of inner eigenfunctions which is the heart of the proof of Proposition \ref{spectraltheorem}. Let for $a\in \mathbb R$ the norm
$$
\| f \|_{X_{y_0}^a} =\underset{0\leq y \leq y_0}{\text{sup}} \sum_{i=0}^2\frac{|(\la y\ra\pa_y)^if|}{\la y\ra^a}.$$

\begin{lemma}[Inner eigenfunctions]
 \lab{pr:int}
Let $i\in \mathbb N$, $\tilde \lambda=O(1)$, $0<r_0\ll1$ small enough and $0<b<b ^*(r_0)$ small enough. Then there exists a smooth profile $\phi_{\text{int}}\in\mathcal C^\infty([0,y_0],\Bbb R)$ satisfying
\be \lab{id eigenfunction int}
\left(H+\frac{b}{2}\Lambda \right)\phi_{\text{int}}=b\left(i-\frac{\alpha}{2} +\tilde \lambda \right)\phi_{\text{int}},
\ee
with the decomposition:
\be \lab{id uint}
\phi_{\text{int}}= \sum_{j=0}^i c_{i,j}b^jT_j+\tilde \lambda \sum_{j=0}^i b^{j+1}(-c_{i,j}T_{j+1}+S_j) +bR_i
\ee
where the constants $(c_{i,j})_{0\leq j \leq i}$ are given by
\be \lab{id cij}
c_{i,j}=(-1)^j\frac{i !}{(i-j)!}, \ \ c_{i,j+1}=-c_{i,j}(i-j),
\ee
and the correction functions $R_i,(S_j)_{0\leq j \leq i}:[0,y_0]\rightarrow \mathbb R$ satisfy
\be \lab{bd Sj}
\| S_j \|_{X_{y_0}^{2j+2-\gamma}} \lesssim r_0^2, \ \ \|\partial_b S_j \|_{X_{y_0}^{2j+4-\gamma}}\lesssim 1, \ \ \|\partial_{\tilde \l} S_j\|_{X_{y_0}^{2j+2-\gamma}}\lesssim r_0^2.
\ee
\be \lab{bd Ri}
\|R_i\|_{X_{y_0}^{2-\gamma-g}}\lesssim 1, \ \ \|\partial_b R_i\|_{X_{y_0}^{4-\gamma-g}}\lesssim 1, \ \  \|\partial_{\tilde \l} R_i\|_{X_{y_0}^{4-\gamma-g}}\lesssim b.
\ee
\end{lemma}

\begin{proof}[Proof of Lemma \ref{pr:int}] Injecting \eqref{id uint} into \eqref{id eigenfunction int} and using \eqref{id Ti}, \eqref{eq:relation Cni}, \eqref{id cij}, we are left with solving
\bee
&&H S_j=b\left[ \left(i-\frac{\alpha}{2}+\tilde \l \right)(-c_{i,j}T_{j+1}+S_j)-\frac{1}{2}\Lambda (-c_{i,j}T_{j+1}+S_j)\right]\\
&&H R_i=b\left[ \left(i-\frac{\alpha}{2}+\tilde \l \right)R_i-\frac{1}{2}\Lambda R_i \right]+\sum_{j=0}^i \frac{c_{i,j}}{2}b^j\Theta_j.
\eee
For this, we use the basis of fundamental solutions to $Hu=0$ given by $$\Lambda Q, \ \ \tilde{\Gamma}=-\Lambda Q\int_1^{y}\frac{d\tau}{(\Lambda Q)^2\tau^{d-1}}d\tau$$ which yields the explicit inverse  
\be \lab{id resolvante}
H^{-1}f(y):= -\Lambda Q(y) \int_0^y \Gammat(s) f(s) s^{d-1}ds+\Gammat (y) \int_0^y \Lambda Q(s)f(s)s^{d-1}ds .
\ee
The behaviour of $\Lambda Q,\tilde{\Gamma}$ at both the origin and $+\infty$ can be computed from \eqref{asyptotq} and yield from direct check the continuity estimate: for $a>-\gamma$, 
\be \lab{bd H-1}
\| H^{-1}f \|_{X_{y_0}^a} \lesssim \underset{0\leq y \leq y_0}{\text{sup}} \la y\ra^{-a+2} |f(r)|.
\ee
This draws a route map for the construction of $S_j,R_j$ using the Banach fixed point theorem with the bounds \eqref{bd Sj}, \eqref{bd Ri}, and the details are left to the reader.
\end{proof}

%%%%%%%%%%%%%%%%%%%%%%%%%%%%%%%%%%%%%%%%%%%%%%%%%

\subsection{Exterior problem} 

%%%%%%%%%%%%%%%%%%%%%%%%%%%%%%%%%%%%%%%%%%%%%%%%%

We construct in this subsection eigenfunctions of $H_b$ in the zone $r\in [r_0,+\infty)$ where $0<r_0\ll 1$. Since $pQ_b^{p-1}\sim \Phi^*(r)=pc_{\infty}^{p-1}r^{-2}$ as $b\rightarrow 0$ from \fref{asyptotq}, the problem can be written to leading order as
\be \lab{id equation vp exterieure}
\Le u= \left(i-\frac{\alpha}{2} \right) u, \ \ \Le :=-\Delta +\frac{1}{2}\Lambda-\frac{pc_{\infty}^{p-1}}{r^2}
\ee
We first construct the outer homogeneous basis.

\begin{lemma}[Fundamental solutions of the exterior problem] \label{lem:eigenfunctions ext}
Let $i\in \mathbb N$. Then the solutions of \fref{id equation vp exterieure} are spanned by two functions $\psi_{1},\psi_{2}$ with 
\be
 \lab{id psi1}
\psi_1= \sum_{j=0}^i c_{i,j}C_jr^{2j-\gamma} = \left\{ \ba{l l} r^{-\gamma}+O(r^{-\gamma +2}) \ \ \text{as} \ r\rightarrow 0, \\ c_{i,i}C_ir^{2i-\gamma}+O(r^{2i-\gamma-2}) \ \ \text{as} \ r\rightarrow +\infty, \ea \right.
\ee
where $c_{i,j}$ and $C_j$ are defined by \fref{eq:relation Cni} and \fref{id cij}, and
\be 
\lab{id asymptotique infty psi2}
\psi_{2} =\left\{\begin{array}{ll}  -\frac{2}{c_{i,i}C_i}r^{-2i+\gamma-d}e^{\frac{r^2}{4}}\left[1+O(r^{-2})\right]  \ \ \text{as} \ \ r\rightarrow +\infty ,\\
 \frac{r^{-\gamma_2}}{d-2\gamma-2}+O(r^{-\gamma_2+2})  \ \ \text{as} \ \ r\rightarrow 0.
 \end{array}\right. 
 \ee
Moreover, there exists a unique solution to $[\Le -(i-\alpha/2)]\tilde \psi_1 =\psi_1$ satisfying
\bea
\lab{as tildepsi1 infty}
&&\tilde \psi_1 = r^{2i-\gamma}\left[2c_{i,i}C_i\text{log}(r)+O(1)\right] \ \ \text{as} \ r \rightarrow +\infty,\\
\lab{as tildepsi1 0}
&&\tilde \psi_1 =r^{-\gamma_2}\left[K+O(r^2)\right]\ \ \text{as} \ r \rightarrow 0, \ K\neq 0,
\eea

\end{lemma}

The proof of Lemma \ref{lem:eigenfunctions ext} follows either by reducing the problem to known Laguerre type polynomials, or proceeding to a brute force expansion, the details are left to the reader. For $a,a' \in \mathbb R$, $0<r_0\ll1$, we let the norms:
\bee
\| f \|_{X_{r_0}^{a,a'}} & = & \underset{r_0\leq r\leq 1}{\text{sup}} \ r^{-a}|f|+r^{-a+1}|\partial_rf|+r^{-a+2}|\partial_{rr}f|\\
&&+\underset{r\geq 1}{\text{sup}} \ r^{-a'}(|f|+|\partial_{rr}f|)+r^{-a'+1}|\partial_r f|.
\eee

\begin{lemma}[Eigenfunctions of the exterior problem] 
\label{lem:ext}
For any $i\in \mathbb N$, $0<r_0<1$, there exists $\tilde \lambda^*>0$ and $b^*>0$ such that for all $0<b<b^*$ and $|\tilde{\l}|\leq \tilde{\l}^*$, there exists a solution $ \phi_{\text{ext}}$ on $[r_0,+\infty)$ of
\be \lab{id phiext} 
H_b \phi_{\text{ext}}=\left(i-\frac{\alpha}{2}+\tilde \l\right)\phi_{\text{ext}}, \ \  \phi_{\text{ext}}= \psi_1+\tilde{\lambda}(\tilde \psi_1+\tilde R_1)+\tilde R_2
\ee
satisfying the estimates:
\be \lab{bd tildeR1}
\| \tilde{R}_1\|_{X_{r_0}^{-\gamma_2, 2i+2-\gamma}}\lesssim |\tilde{\l}|, \ \  \partial_b \tilde R_1=0, \ \ \| \partial_{\tilde{\l}}\tilde{R}_1\|_{X_{r_0}^{-\gamma_2, 2i+2-\gamma}}\lesssim 1,
\ee
and for $a=-\gamma_2-2-\alpha$, $a'=2i+2-\gamma$
\be \lab{bd tildeR2}
\| \tilde{R}_2\|_{X_{r_0}^{a,a'}}\lesssim b^{\frac{\alpha}{2}}, \ \ \|\partial_{\tilde{\l}} \tilde{R}_2\|_{X_{r_0}^{a,a'}}\lesssim b^{\frac{\alpha}{2}}, \ \ \ \ \|\partial_b \tilde{R}_2\|_{X_{r_0}^{a,a'}}\lesssim b^{\frac{\alpha}{2}-1}.
\ee
\end{lemma}

\begin{proof}[Proof of Lemma \ref{lem:ext}] We ensure \eqref{id phiext} by solving
\bee
&&\left( \Le -\left(i-\frac{\alpha}{2}\right) \right)\tilde{R}_1-\tilde \lambda \tilde \psi_1-\tilde{\l}\tilde{R}_1 =0,\\
&&\left( \Le -\left(i-\frac{\alpha}{2}\right) \right)\tilde R_2+\left(\frac{pc_{\infty}^{p-1}}{r^2}-pQ^{p-1}_b \right)(\psi_1+\tilde{\lambda}(\tilde \psi_1+\tilde{R}_1)+\tilde{R}_2)-\tilde \lambda \tilde{R}_2=0.
\eee
For this, we consider the explicit inverse:
$$\left[\Le-(i-\frac{\alpha}{2})\right]^{-1}f(r)=-\psi_1(r)\int_1^{r}f\psi_2s^{d-1}e^{-\frac{s^2}{4}}ds-\psi_2(r)\int_r^{+\infty}f\psi_1s^{d-1}e^{-\frac{s^2}{4}}ds$$ which satisfies from direct check the continuity estimate:  for $a\leq -\gamma_2$, $a\neq \gamma-d$ and $a'>2i-\gamma$, 
$$
\| \Le^{-1}(f)\|_{X_{r_0}^{a,a'}} \lesssim \underset{r_0\leq r \leq 1}{\text{sup}} \ r^{-a}|f|+\underset{ r\geq 1}{\text{sup}} \ r^{-a'}|f|.
$$
This designs again a route map for the construction of $\tilde{R}_1,\tilde{R}_2$ with the Banach fixed point argument which is left to the reader. One additional difficulty is the dependance in $\tilde{\l}$ which is in fact explicit to leading order.
\end{proof}

%%%%%%%%%%%%%%%%%%%%%%%%%%%%%%%%%%%%%%%%%%%%%%%%%%%%%

\subsection{Matching}

%%%%%%%%%%%%%%%%%%%%%%%%%%%%%%%%%%%%%%%%%%%%%%%%%%%%%

We now match the inner and outer solutions provided by Lemma \ref{pr:int} and Lemma \ref{lem:ext}. The following Lemma directly implies Proposition \ref{spectraltheorem}.

\begin{lemma}[Matching]
 \lab{pr:matching}
Fix $\ell \in \mathbb N$. There exists $0<r_0\ll1$ and $b^*(r_0)>0$ small enough, such that for $0<b\leq b^*$ and $j\in \mathbb N$, $0\leq i \leq \ell$, there exists a unique $\tilde \lambda_i=\tilde \l_i(b) \in \mathbb R$ with
\be \lab{bd lambdai}
|\tilde \l_i|\lesssim b^{\frac g2}, \ \ |\partial_b \tilde \l_i|\lesssim b^{\frac g2 -1}
\ee
such that $\phi_i$ defined by
\be \lab{id phi matching}
\phi_i (r)=\left\{ \begin{array}{ll} b^{-\frac{\gamma}{2}}\phi_{\text{int}} \left(\frac{r}{\sqrt b} \right)\ \ \text{if} \ 0\leq r \leq r_0, \\
\frac{b^{-\frac{\gamma}{2}}\phi_{\text{int}} \left(\frac{r_0}{\sqrt b} \right)}{\phi_{\text{ext}(r_0)}} \phi_{\text{ext}}(r) \ \ \text{if} \ r\geq r_0, \end{array} \right.
\ee
where $\phi_{\text{int}}$ and $\phi_{\text{ext}}$ are given by Lemma \ref{pr:int} and Lemma \ref{lem:ext}, is a smooth solution of $H_b \phi_i=\l_i\phi_i$, $\l_i=i-\alpha/2+\tilde \l_i$. Moreover, all the bounds in Proposition \ref{spectraltheorem} hold true.
\end{lemma} 

\begin{proof}[Proof of Lemma \ref{pr:matching}] We rewrite the matching condition \be
\lab{id cond r0} \left\{ \begin{array}{r c l}
 b^{-\frac{\gamma}{2}}\phi_{\text{int}} \left(\frac{r_0}{\sqrt b} \right)=\frac{b^{-\frac{\gamma}{2}}\phi_{\text{int}} \left(\frac{r_0}{\sqrt b} \right)}{\phi_{\text{ext}(r_0)}} \phi_{\text{ext}}(r_0) \\
 b^{-\frac{\gamma}{2}-\frac 12}\phi_{\text{int}}' \left(\frac{r_0}{\sqrt b} \right)=\frac{b^{-\frac{\gamma}{2}}\phi_{\text{int}} \left(\frac{r_0}{\sqrt b} \right)}{\phi_{\text{ext}(r_0)}} \phi'_{\text{ext}}(r_0) \\
\end{array}
\right.  \Longleftrightarrow   \Phi[r_0](b,\tilde \l)=0
\ee
with
$$\Phi [r_0](b,\tilde \l):=\frac{b^{-\frac 12}\phi_{\text{int}}' (y_0)}{\phi_{\text{int}} (y_0)}-\frac{\phi'_{\text{ext}}(r_0) }{\phi_{\text{ext}(r_0)}}
$$
A tedius sequence of computations using the correction bounds of Lemma \ref{pr:int} and Lemma \ref{lem:ext} yields the expansion:
\bee
 &&\Phi[r_0](b,\tilde \l)=\tilde \l (\gamma_2-\gamma)K r_0^{\gamma-\gamma_2-1}[1+O(|\tilde{\l}|)]+O(b^{\frac{g}{2}}) \\
 &&\partial_b \Phi[r_0](b,\tilde \l) = O(b^{-1+\frac g 2})+O(|\tilde \l|b^{-1}r_0^3) -O(b^{\frac{\alpha}{2}-1}) =O(b^{-1+\frac g 2})+O(|\tilde \l|b^{-1}r_0^3)\\
 &&\partial_{\tilde \l} \Phi[r_0](b,\tilde \l)=(\gamma_2-\gamma)K r_0^{\gamma-\gamma_2-1}(1+O(|\tilde{\l}|)).
\eee
Therefore, since $(\gamma_2-\gamma)K\neq 0$, using the intermediate value theorem, for $r_0$ small enough, then for $0<b<b^*$ small enough, there exists at least one $\tilde \lambda_i=\tilde \l_i(b)=O(b^{\frac g2})$ such that the matching condition \fref{id cond r0} is satisfied. The uniqueness of the corresponding eigenvalue is now a simple consequence of the spectral gap \eqref{asyptoticgap} and the smallness \eqref{calculeignefunction}. 
\end{proof}

%%%%%%%%%%%%%%%%%%%%%%%%%%%%%%%%%%%%%%%%%%%%%%%%%%%%%%%%
%%%%%%%%%%%%%%%%%%%%%%%%%%%%%%%%%%%%%%%%%%%%%%%%%%%%%%%%

\section{Setting up the bootstrap argument}
\label{sectionbootstrap}

%%%%%%%%%%%%%%%%%%%%%%%%%%%%%%%%%%%%%%%%%%%%%%%%%%%%%%%%
%%%%%%%%%%%%%%%%%%%%%%%%%%%%%%%%%%%%%%%%%%%%%%%%%%%%%%%%

This section is devoted to the set up of the bootstrap argument at the heart of the proof of Theorem \ref{thmmain}. The initial datum is constructed by following the flow in a suitable regime, and we avoid the instability directions of type II blow up and the additional transversal directions using a nowadays standard outgoing flux argument. In this section, we prepare the analysis and show in particular how the spectral Proposition \ref{diaglb} provides an elementary setting to compute the type II blow up speeds \eqref{loideb} and control the flow in the associated $H^1_{\rho_Y}$ weighted norms.

%%%%%%%%%%%%%%%%%%%%%%%%%%%%%%%%%%%%%%%%%%%%%%%%%%%%%%%%

\subsection{The reconnection function}

%%%%%%%%%%%%%%%%%%%%%%%%%%%%%%%%%%%%%%%%%%%%%%%%%%%%%%%%

In order to build a solution to \eqref{heatnonlin}, we pick $0<T\ll1$ and consider the self similar change of variables
$$u(t,x)=\frac{1}{(T-t)^{\frac{1}{p-1}}}U(\tau,Y), \ \ Y=\frac{x}{\sqrt{T-t}}, \ \ \tau=-\log(T-t),$$ which maps \eqref{heatnonlin} onto the self similar equation 
\be
\label{selfsimilarequation}
\pa_{\tau}U-\Delta U+\frac 12\Lambda U -U|U|^{p-1}=0
\ee
on a global time interval 
\be
\label{renormalizedtime}
\tau\in[\tau_0,+\infty), \ \ \tau_0=-\log T\gg1.
\ee
In order to produce a suitable approximate solution to \eqref{selfsimilarequation}, we let $P_{2\ell}(z)$ be the even Legendre polynomial \eqref{hermitepolynomial} so that 
\be
\label{limiteuig}
\lim_{|z|\to+\infty}P_{2\ell}(z)=+\infty.
\ee Given $0<a<a^*$ small enough, this implies $1+aP_{2\ell}(z)\geq \frac 12$ for all $z\in \Bbb R$ and we may therefore consider 
\be
\label{defmunu}
\mu=\mu(a,z)=(1+\nu(z))^{\frac{1}{\alpha}}, \ \ \nu(z)=aP_{2\ell}(z).
\ee 
Given another parameter $0<b<b^*$ small enough, we introduce the reconnection function $$D=D(a,b,z)=\sqrt{b}\mu=\sqrt{b}(1+aP_{2\ell}(z))^{\frac{1}{\alpha}}$$ and the even cylindrical blow up profile: 
\be
\label{defphia}
\Phi_{a,b}(r,z)=\frac{1}{D^{\frac{2}{p-1}}}Q\left(\frac{r}{D}\right)=\frac{1}{\mu(z)^{\frac 2{p-1}}}Q_b\left(\frac{r}{\mu(z)}\right)
\ee
where we recall the notation \eqref{defqbÙ}. We claim that $\Phi_{a,b}$ is an approximate solution to \eqref{selfsimilarequation} in the following sense:

\begin{lemma}[All order cancellation in $a$]
\label{lemmacancellationa}
Assume that $(a,b)\in \mathcal C^1([\tau_0,\tau^*],(0,a^*]\times(0,b^*])$, then
\be
\label{eqautionprofile}
\pa_\tau\Phi_{a,b}-\Delta\Phi_{a,b}+\frac 12\Lambda\Phi_{a,b}-\Phi_{a,b}|\Phi_{a,b}|^{p-1}=\Psi_1
\ee
with
\be
\label{formulapsi}
\Psi_1  = \left[\frac12 \left(-\frac{b_\tau}{b}+1\right)-\frac{\ell}{\alpha}-\frac1{2\alpha}\frac{\pa_\tau\nu}{1+\nu}\right]\Lambda_r\Phi_{a,b} +\frac{\ell}{\alpha}\Lambda Q_b(r)+\tilde{\Psi}_1
\ee
and
\be
\label{psitildeone}
\Psit_1= \frac{\ell}\alpha \left[\frac{1}{\mu^{\gamma}}\Lambda Q_b\left(\frac{r}{\mu}\right)-\Lambda Q_b(r)\right]-\frac{1}{\alpha^2}\left(\frac{\pa_z\nu}{1+\nu}\right)^2\frac{1}{\mu^{\frac{2}{p-1}}}(\Lambda^2Q_b+\alpha\Lambda Q_b)\left(\frac{r}{\mu}\right).
\ee
\end{lemma}

\begin{remark} We claim that \eqref{formulapsi}, \eqref{psitildeone} correspond to an all order cancellation in the parameter $a$. Indeed, replace $Q$ by its far away asymptotic expansion \eqref{asyptotq}, then from \eqref{defalpha}, both terms in $\Psit_1$ vanish. The first term in \eqref{formulapsi} will vanish for the laws $$\frac12\left(-\frac{b_\tau}{b}+1\right)-\frac{\ell}{\alpha}=0, \ \ a_\tau=0$$ which yield \eqref{loideb}, \eqref{loidea}, and hence to leading order $$\Psi_1\sim \frac{\ell}{\alpha}\Lambda Q_b(r),$$ see \eqref{degenpsitildeun} and Appendix \ref{appendnonlin} for quantitative statements. This  external force will be adjusted {\em exactly} using an excitation of the eigenmode $\psi_{\ell,0}$.
\end{remark}

\begin{proof} This is a brute force computation. First:
\bee
\pa_\tau \Phi_{a,b}+\frac 12\Lambda_r \Phi_{a,b}&=&\left[\frac12\left(-\frac{b_\tau}{b}+1\right)-\frac{\pa_\tau\mu}{\mu}\right]\Lambda_r\Phi_{a,b}\\
& = & \left[\frac12\left(-\frac{b_\tau}{b}+1\right)-\frac{\ell}{\alpha}-\frac{\pa_\tau\mu}{\mu}\right]\Lambda_r\Phi_{a,b}+\frac{\ell}{\alpha}\Lambda_r\Phi_{a,b}.
\eee
We now let  $$\mu^\alpha=\tilde{\mu}=1+aP_{2\ell}=1+\nu$$ and compute:
\bee
&& \frac{\ell}{\alpha}\Lambda_r\Phi_{a,b}  -\pa_z^2\Phi_{a,b}+\frac 12z\pa_z\Phi_{a,b}\\
& = & \left[\frac{\ell}{\alpha}-\frac{1}{2}z\frac{\pa_z\mu}{\mu}\right]\frac{1}{\mu^{\frac{2}{p-1}}}\Lambda Q_b\left(\frac{r}{\mu}\right)-\frac{1}{\mu^{\frac{2}{p-1}}}\left[\left(\frac{\pa_z\mu}{\mu}\right)^2\Lambda^2Q_b-\pa_{zz}(\log \mu)\Lamdba Q_b\right]\left(\frac{r}{\mu}\right)\\
& = & \left[\frac{\ell}{\alpha}-\frac{1}{2}z\frac{\pa_z\mu}{\mu}+\alpha\left(\frac{\pa_z\mu}{\mu}\right)^2+\pa_{zz}\log \mu\right]\frac{1}{\mu^{\frac{2}{p-1}}}\Lamdba Q_b\left(\frac{r}{\mu}\right)-\frac{1}{\mu^{\frac{2}{p-1}}}\left(\frac{\pa_z\mu}{\mu}\right)^2(\Lambda^2Q_b+\alpha\Lambda Q_b)\left(\frac{r}{\mu}\right)\\
& = & \frac{1}{\alpha \mut}\left[\ell\mut-\frac12z\pa_z\mut+\pa_{zz}\mut\right]\frac{1}{\mu^{\frac{2}{p-1}}}\Lamdba Q_b\left(\frac{r}{\mu}\right)-\frac{1}{\alpha^2}\left(\frac{\pa_z\mut}{\mut}\right)^2\frac{1}{\mu^{\frac{2}{p-1}}}(\Lambda^2Q_b+\alpha\Lambda Q_b)\left(\frac{r}{\mu}\right)\\
& = & \frac{\ell}{\alpha}\frac{1}{1+\nu}\Lambda_r\Phi_{a,b}-\frac{1}{\alpha^2}\left(\frac{\pa_z\nu}{1+\nu}\right)^2\frac{1}{\mu^{\frac{2}{p-1}}}(\Lambda^2Q_b+\alpha\Lambda Q_b)\left(\frac{r}{\mu}\right)
\eee
where we used in the last step \eqref{equationhermite} which ensures: $$\ell\mut-\frac12z\pa_z\mut+\pa_{zz}\mut=\ell-a\left(-\pa_z^2+\frac 12 z\pa_z-\ell\right)P_{2\ell}=\ell.$$ We now observe from \eqref{defalpha}: $$ \frac{\ell}{\alpha}\frac{1}{1+\nu}\Lambda_r\Phi_{a,b}=\frac{\ell}{\alpha}\frac{1}{\mu^\gamma}\Lambda Q_b\left(\frac{r}{\mu}\right)$$ and, since $\Delta_rQ+|Q|^{p-1}Q=0$, \eqref{formulapsi} and \eqref{psitildeone} follow.
\end{proof}

%%%%%%%%%%%%%%%%%%%%%%%%%%%%%%%%%%%%%%%%%%%%%%%%%%%%%%%%

\subsection{Geometrical decomposition of the flow}

%%%%%%%%%%%%%%%%%%%%%%%%%%%%%%%%%%%%%%%%%%%%%%%%%%%%%%%%

We now describe our set of initial data and the associated bootstrap bounds required to control their time evolution through the renormalized flow \eqref{selfsimilarequation}.\\

Let $(\psi_{j,k})_{0\le j\leq \ell, 0\le k\le \ell-j}$ be the eigenmodes of Proposition \ref{diaglb}, then given parameters $$0<b<b^*, \ \ 0<a<a^*$$ small enough, we first claim the non degeneracy of the scalar products generated by the profile $\Phi_{a,b}$ given by \eqref{defphia} with the eigenbasis.

\begin{lemma}[Computation of the scalar products]
There holds for some universal constants $I\neq 0$ and $\delta>0$ for $0\le j\le \ell$, $0\le k\le \ell-j$:
\bea
\label{nondegenone}
&&\left(b\pa_b\Phi_{a,b},\psi_{j,k}\right)_{L^2_{\rho_Y}}=-\frac I2(\sqrt{b})^{\alpha}\left[\delta_{j0}\delta_{k0}+O(b^\delta+a)\right]\\
\label{nondegenonebis}&&\left(\pa_a\Phi_{a,b},\psi_{j,k}\right)_{L^2_{\rho_Y}}=-\frac I\alpha(\sqrt{b})^{\alpha}\left[\delta_{j0}\delta_{k\ell}+O(b^\delta+a)\right].
\eea
\end{lemma}

\begin{proof} The proof follows from the asymptotics \eqref{asyptotq} and the relation \eqref{defgamma} to provide integrability at the origin.\\

\noindent{\bf step 1} Norm computation. Let $\zeta\in \Bbb R$, $\mu\geq \frac 12$, $k\in \{0,1\}$, $(i,j)\in \Bbb N^2$, we claim:
\bea
\label{controlenorme}
&& \left\|\frac{1}{(\sqrt{b})^{\alpha}}\pa_r^k\left(\frac{1}{\mu^{\frac{2}{p-1}}}\Lambda^ iQ_b+\zeta\frac{1}{\mu^{\frac{2}{p-1}}}\Lambda^jQ_b\right)\left(\frac{r}{\mu}\right)\right\|_{L^2_{\rho_r}}^2\\
\nonumber &=& c^2(\gamma)_k^2\mu^{2\left(\frac d2-\frac{2}{p-1}-k\right)}\left[1+O(\mu^2-1)\right]\left[(-\alpha)^i+\zeta(-\alpha)^j)+O(b^\delta)\right]^2\int_0^{+\infty}\frac{1}{r^{2\gamma+2k}}\rho_rr^{d-1}dr.
\eea
for some universal constant $c>0$ and $\delta>0$ smal enough. Indeed, recall the asymptotics \eqref{asyptotq} which implies for $c_k=c(-1)^k(\gamma)_k\neq 0$
\be\label{asymptoticexpansion}
\pa_y^k\Lambda^iQ(y)=c_k\frac{(-\alpha)^i}{y^{\gamma+k}}+O\left(\frac{1}{y^{\gamma+\delta+k}}\right)\ \ \mbox{for}\ \ |y|\gg1, \ \ k\in \{0,1\}.\ee Let $A=\frac{1}{\sqrt{b}^\delta}$ for $\delta$ universal small enough, then for $k\in \{0,1\}$ and changing variables:
\bea
\label{degeneracyfundmantal}
\nonumber&&\frac{1}{\mu^{2\left(\frac d2-\frac{2}{p-1}-k\right)}}\left\|\frac{1}{(\sqrt{b})^{\alpha}}\pa^k_r\left[\frac{1}{\mu^{\frac{2}{p-1}}}(\Lambda^i Q_b+\zeta\Lambda^jQ_b)\left(\frac{r}{\mu}\right)\right]\right\|_{L^2_{\rho_r}}^2\\
\nonumber &=&\int_0^{+\infty} \left|\frac{1}{(\sqrt{b})^{\alpha}}\pa^k_r(\Lambda^ i Q_b+\zeta\Lambda^jQ_b)\left(r\right)\right|e^{-\frac{\mu^2r^2}{4}}r^{d-1}dr\\
\nonumber &=& \int_{r\leq A\sqrt{b}}\left|\frac{1}{(\sqrt{b})^{\alpha+\frac{2}{p-1}+k}}\left[\pa^k_y(\Lambda^ i Q+\zeta\Lambda^jQ)\right]\left(\frac{r}{\sqrt{b}}\right)\right|^2e^{-\frac{\mu^2r^2}{4}}r^{d-1}dr\\
\nonumber&+& \int_{r\geq A\sqrt{b}}\left|\frac{c_k}{(\sqrt{b})^{\alpha+\frac{2}{p-1}+k}}\frac{(\sqrt{b})^{\gamma+k}}{r^{\gamma+k}}\left((-\alpha)^i+\zeta(-\alpha)^j+O\left(\frac{1}{A^\delta}\right)\right)\right|^2e^{-\frac{\mu^2r^2}{4}}r^{d-1}dr\\
\nonumber& = & O\left(\frac{(A\sqrt{b})^{d}}{(\sqrt{b})^{2\gamma+2k}}\right)+\left[c_k^2((-\alpha)^i+\zeta(-\alpha)^j))^2+O\left(\frac{1}{A^\delta}\right)\right]\int_{r\geq A\sqrt b}\frac{1}{r^{2\gamma+2k}}e^{-\frac{\mu^2r^2}{4}}r^{d-1}dr\\
& = &  O\left(b^\delta\right)+c_k^2\left[(-\alpha)^i+\zeta(-\alpha)^j)+O\left(b^\delta\right)\right]^2\int_0^{+\infty}\frac{1}{r^{2\gamma+2k}}e^{-\frac{\mu^2r^2}{4}}r^{d-1}dr.
\eea
where we used \be
\label{estfumdamental}
d-1-(2\gamma+2)=-1+\sqrt{\Delta}>-1.
\ee
We now write $$\left|e^{-\frac{\mu^2r^2}{4}}-e^{-\frac{r^2}{4}}\right|=\left|\int_1^{\mu}\mu'\frac{r^2}{2}e^{-\frac{(\mu')^2r^2}{4}}d\mu'\right|\lesssim |\mu^2-1|e^{-\frac{r^2}{16}}$$ from $\mu\geq \frac 12$ and \eqref{controlenorme} follows.\\

\noindent{\bf step 2} $a$ correction. We now estimate the $a$ correction and claim:
\be
\label{controlerreur}
\|\Lambda_r\Phi_{a,b}-\Lambda Q_b(r)\|_{H^1_{\rho_Y}}\lesssim  |a|\sqrt{b}^\alpha.
\ee
Indeed,
$$\frac{d}{d\mu}\Lambda_r\Phi_{a,b}=-\frac{1}{\mu^{1+\frac2{p-1}}}\Lambda^2Q_b\left(\frac{r}{\mu}\right).
$$ The global bound 
\be
\label{roughboundlambdaq}
|\pa_i\Lambda^j Q|\lesssim \frac{1}{1+|y|^{\gamma+i}}, \ \ i=0,1
\ee
implies the global control $$|\pa_\mu \pa_r^i\Lambda_r\Phi_{a,b}|\lesssim \frac{1}{\mu}\frac{1}{(\mu\sqrt{b})^i}\frac{1}{(\mu\sqrt{b})^{\frac{2}{p-1}}}\frac{1}{1+\left(\frac{r}{\mu\sqrt{b}}\right)^{\gamma+i}}\lesssim \frac{\mu^{\alpha-1}(\sqrt{b})^\alpha}{r^{\gamma+i}}$$ which together with the Taylor Lagrange formula yields:
$$\pa_r^i\Lambda_r\Phi_{a,b}=\pa_r^i\Lambda Q_b(r)+O\left(\int_1^\mu (\sqrt{b})^\alpha\frac{(\mu')^{\alpha-1}}{r^{\gamma+i}}d\mu'\right)=  \pa_r^i\Lambda Q_b(r)+O\left((\sqrt{b})^\alpha\frac{|\mu^\alpha-1|}{r^{\gamma+i}}\right)
$$
and hence, since $d-2\gamma-2=\sqrt{\Delta}>0$ and $\mu^{\alpha}-1=aP_{2\ell}$, $$\|\pa_r^i\left(\Lambda_r\Phi_{a,b}-\Lambda Q_b(r)\right)\|_{L^2_{\rho_Y}}\lesssim |a|\sqrt{b}^\alpha\|\frac{P_{2\ell}}{r^{\gamma+i}}\|_{L^2_{\rho_Y}}\lesssim |a|\sqrt{b}^\alpha, \ \ i=0,1$$
 and \eqref{controlerreur} is proved.\\

\noindent{\bf step 3} Proof of \eqref{nondegenone}, \eqref{nondegenonebis}. We compute:
$$\pa_b\Phi_{a,b}=-\frac{\pa_b D}{D}\Lambda_r\Phi_{a,b}=-\frac{1}{2b}\Lambda_r\Phi_{a,b}.$$ Now from \eqref{calculeignefunction}: 
\be
\label{einoenoneoeon}
\phi_{0,b}=\frac{1}{(\sqrt{b})^\alpha}\Lambda Q_b+\tilde{\phi}_{0,b}, \ \ \|\tilde{\phi}_{0,b}\|_{H^1_{\rho_r}}\lesssim b^{\delta}.
\ee
Moreover, \eqref{controlenorme} with $\mu=1$ yields the expansion:
\be
\label{definitioni}
\|\frac{1}{\sqrt{b}^\alpha}\Lambda Q_b\|_{L^2_{\rho_r}}^2=I+O(b^\delta), \ \ I:=c^2\int_0^{+\infty}\frac{r^{d-1}}{r^{2\gamma}}e^{-\frac{r^2}{4}}r^{d-1}dr
\ee
and hence using the $L^2_{\rho_z}$ orthonormality of Hermite polynomials and the $L^2_{\rho_r}$ orthogonality of the $\phi_{i,b}$ and the fact that $P_0=1$ :
 \be
 \label{scalrapronienoevni}
 \left(\frac{1}{(\sqrt{b})^\alpha}\Lambda Q_b,\psi_{j,k}\right)_{L^2_{\rho_Y}}= \left(\frac{1}{(\sqrt{b})^\alpha}\Lambda Q_b,\phi_{j,b}P_{2k}\right)_{L^2_{\rho_Y}}=I\delta_{j0}\delta_{k0}+O(b^\delta).
 \ee 
 We conclude using \eqref{controlerreur}: 
 \bee
 \left(\pa_b\Phi_{a,b},\psi_{j,k}\right)_{L^2_{\rho_Y}}&=&-\frac{1}{2b}(\Lambda_r\Phi_{a,b},\phi_{j,b}P_{2k})_{L^2_{\rho_Y}}=-\frac{1}{2b}\left[(\Lambda Q_b(r),\phi_{j,b}P_{2k})_{L^2_{\rho_Y}}+O(a(\sqrt{b})^\alpha)\right]\\
 &=& -\frac{(\sqrt{b})^\alpha}{2b}\left[I\delta_{j0}\delta_{k0}+O(a+b^\delta)\right]
 \eee
 and \eqref{nondegenone} is proved. similarly, $$\pa_a\Phi_{a,b}=-\frac{\pa_a D}{D}\Lambda_r\Phi_{a,b}=-\frac{1}{\alpha}\frac{P_{2\ell}}{1+aP_{2\ell}}\Lambda_r\Phi_{a,b}$$ and hence using the exponential localization of the $\rho_zdz$ measure and the $L^2_{\rho_z}$ orthogonality of Hermite polynomials:
 \bee
 &&\left(\pa_a\Phi_{a,b},\psi_{j,k}\right)_{L^2_{\rho_Y}}=-\frac{1}{\alpha}\left(\frac{P_{2\ell}}{1+aP_{2\ell}}\Lambda_r\Phi_{a,b},\psi_{j,k}\right)_{L^2_{\rho_Y}}\\
 &=&-\frac{1}{\alpha}\left[\left(\frac{P_{2\ell}}{1+aP_{2\ell}}\Lambda Q_b(r),\psi_{j,k}\right)_{L^2_{\rho_Y}}+O(a(\sqrt{b})^\alpha)\right]\\
 &=&-\frac{(\sqrt{b})^{\alpha}}{\alpha}\left[\left(\frac{P_{2\ell}}{1+aP_{2\ell}}\phi_{0,b},\phi_{j,b}P_{2k}\right)_{L^2_{\rho_Y}}+O(a+b^\delta)\right]= -\frac{(\sqrt{b})^{\alpha}}{\alpha}\left[I\delta_{j0}\delta_{k\ell}+O(a+b^\delta)\right]
 \eee
 and \eqref{nondegenonebis} is proved.
\end{proof}

A standard application of the implicit function theorem using \eqref{nondegenone}, \eqref{nondegenonebis} and the smallness \eqref{bound:partialb} now ensures the following. Given a constant $K>0$ , there exist $0<a^*(K),b^*(K),c^*(K)\ll1$ and a universal constant such that for all $$|a_1|<a^*(K), \ \ 0<b_1<b^*(K)$$ the following holds: recall the definition of the set of indices \fref{def mathcal I} and let parameters
$${\bf b}_1=(b_{j,k})_1, \ \ (j,k)\in \mathcal I$$ with $|{\bf b}_1|<K(\sqrt{b_1})^\alpha$, 
 and define the finite dimensional projection
$$
  \psi_{{\bf b}_1,b_1}=\sum_{j=1}^\ell (b_{j,0})_1\psi_{j,0,b_1}-\left(\sum_{j=1}^\ell (b_{j,0})_1\right)\psi_{0,0,b_1}+\sum_{k=1}^{\ell-1}(b_{0,k})_1\psi_{0,k,b_1}+\sum_{j=1}^{\ell-1}\sum_{k=1}^{\ell-j}(b_{j,k})_1\psi_{j,k,b_1}.
 $$
Let $\chi$ be a cut-off function, $\chi(Y)=1$ for $|Y|\leq 1$ and $\chi(Y)=0$ for $|Y|\geq 2$ and $\kappa>0$ is universal. Then any function of the form
\be
\label{estinittaldata}
U=\Phi_{a_1,b_1}+ \chi(b_1^{\kappa}Y) \psi_{{\bf b}_1,b_1}+\e_1
\ee
with $\|\e_1\|_{L^2_{\rho_Y}}\leq c^*(K)(\sqrt{b_1})^\alpha$ can be uniquely reparametrized as 
\be \lab{decomp U}
U=\Phi_{a,b}+\psi_{{\bf b},b}+\e,
\ee
where $\e$ satisfies the orthogonality conditions:
\be
\label{orthoe}
(\e,\psi_{j,k,b})_{L^2_{\rho_Y}}=0, \ \ 0\leq j\leq \ell, \ \ 0\leq k\leq \ell-j.
\ee
Moreover, $$\left|\frac{b}{b_1}-1\right|+|a_1-a|<\sqrt{c^*(K)}$$ and $$|{\bf b}-{\bf b}_1|+\|\e\|_{L^2_{\rho_Y}}<\sqrt{c^*(K)}(\sqrt{b_1})^\alpha.$$ 
We therefore pick an initial data of the form \eqref{estinittaldata} ie equivalently $$U(\tau_0)=\Phi_{a,b}(\tau_0)+\e(\tau_0)+\psi(\tau_0)$$ with $\e(\tau_0)$ satisfying \eqref{orthoe} and 
\be
\label{initialsmallnessboot}
\|\e(\tau_0)\|_{H^1_{\rho_Y}}<c(\sqrt{b(\tau_0)})^{\alpha}, \ \ |b_{j,k}(\tau_0)|<K(\sqrt{b(\tau_0)})^{\alpha}
\ee
and where we note to ease notations 
\be
\label{defpsifinal}
\psi=\psi_{{\bf b},b}=\sum_{j=1}^\ell b_{j,0}\psi_{j,0}-\left(\sum_{j=1}^\ell b_{j,0}\right)\psi_{0,0}+\sum_{k=1}^{\ell-1}b_{0,k}\psi_{0,k}+\sum_{j=1}^{\ell-1}\sum_{k=1}^{\ell-j}b_{j,k}\psi_{j,k}.
\ee  Then the corresponding solution \eqref{selfsimilarequation} admits a unique time dependent decomposition 
\be
\label{deompe}
U=\Phi_{a,b}+V, \ \ V=\psi+\e
\ee 
with $\e(\tau,\cdot)$ satisfying \eqref{orthoe} on any time interval $[0,\tau^*]$ such that:
$$
\forall \tau\in[0,\tau^*], \ \ 0<b(\tau)<b^*(K), \ \ |a(\tau)|<a^*(K)
$$
and 
$$\|\e(\tau)\|_{H^1_{\rho_Y}}<c(\sqrt{b(\tau)})^{\alpha}, \ \ |b_{j,k}(\tau)|<K(\sqrt{b(\tau)})^{\alpha}$$
For initial data satisfying additional regularity assumptions as the one that we will now consider, the regularity $(a,b,b_{j,k})\in \mathcal C^1([\tau_0,\tau^*],\Bbb R)$ is a standard consequence of the smoothness of the flow for \eqref{heatnonlin}. We refer to \cite{HR,CMR,MRR} for related statements.

 %%%%%%%%%%%%%%%%%%%%%%%%%%%%%%%%%%%%%%%%%%%%%%%%%%%%%%%%%%%
%%%%%%%%%%%%%%%%%%%%%%%%%%%%%%%%%%%%%%%%%%%%%%%%%%%%%%%%%%%

\subsection{Bootstrap}

%%%%%%%%%%%%%%%%%%%%%%%%%%%%%%%%%%%%%%%%%%%%%%%%%%%%%%%%%%%
%%%%%%%%%%%%%%%%%%%%%%%%%%%%%%%%%%%%%%%%%%%%%%%%%%%%%%%%%%%

We now specify more carefully a set of bootstrap estimates in the variables of the decomposition \eqref{deompe}. In particular the exponentially weighted norms {\em do not provide enough information} to control the full nonlinear flow. Let us define
\be
\label{renormalizedbvariables}
b_{j,k}= b_{\ell,0}\tilde{b}_{j,k}, \ \ (\sqrt{b})^\alpha=-\alpha b_{\ell,0}(1+\tilde{b}).
\ee 
We pick two small enough universal constants $0<\tilde{\eta}\ll \eta\ll1$, a large enough integer $q\gg1$, a large enough universal constant $K\gg1$, and assume the following bounds on the initial data:
\begin{itemize}
\item normalization of the dominant mode:
\be
\label{initmode}
b_{\ell,0}(\tau_0)=-e^{-(\ell-\frac\alpha 2)\tau_0};
\ee
\item $a_0$ is very large compared to $b(\tau_0)$ and nonnegative:
\be
\label{condiotnazero}
0<\frac{1}{|\log b(\tau_0)|^{\frac 1K}}<a(\tau_0)<a^*\ll1;
\ee
\item initial exit condition:
\be
\label{exitcondition}
|\tilde{b}(\tau_0)|^2+\sum_{(j,k)\neq(\ell,0), j+k\leq \ell} |\tilde{b}_{j,k}(\tau_0)|^2< (\sqrt{b_0})^{2\eta};
\ee
\item initial smallness of the weighted norm:
\be
\label{smallnorminit}
\|\e(\tau_0)\|_{L^2_{\rho_Y}}< (\sqrt{b_0})^{\alpha+2\eta},
\ee
\item inital smallness of the global $W^{1,2q+2}$ like norm:
\be
\label{smallnorminitsobolev}
\mathcal N(\tau_0)< (\sqrt{b_0})^{2\tilde{\eta}}
\ee
where $\mathcal N$ is given \eqref{defnorme} and is constructed to control in particular 
\be
\label{controllinfty}
\|\phi V\|_{L^\infty}\lesssim \mathcal N(V).
\ee
where $\phi$ is a smooth function satisfying 
 \be
 \label{defphi}
\phi(r)=\left|\begin{array}{ll} 
r^{\frac{2}{p-1}}\ \ \mbox{for}\ \ 0\leq r\leq 1\\
1 \ \ \mbox{for}\  \ r\geq 2.\end{array}\right.
 \ee

\end{itemize}

The fact that our set of initial data is non empty and contains smooth profiles with finite energy follows from an elementary localization claim which is left to the reader, see for example \cite{CRS} for related computations. We then define $\tau_0\leq \tau^*\leq +\infty$ as the supremum of times such that the following bounds hold on $[\tau_0,\tau^*)$:
\begin{itemize}
\item control of the dominant modes:
\be
\label{initmodeboot}
-2<b_{\ell,0}e^{(\ell-\frac\alpha 2)\tau}<-\frac 12;
\ee
\item control of $a$:
\be
\label{controlea}
\frac{a(\tau_0)}{2}<a<2a(\tau_0),
\ee
\item exit condition:
\be
\label{exitconditionboot}
|\tilde b|^2+\sum_{(j,k)\in \mathcal I}|\tilde b_{j,k}|^2<(\sqrt b)^{2\eta}
\ee
\item smallness of the weighted norm:
\be
\label{smallnorminitboot}
\|\e\|_{L^2_{\rho_Y}}<(\sqrt{b})^{\alpha+\eta},
\ee
\item smallness of the global $W^{1,2q+2}$ like norm:
\be
\label{smallnorminitsobolevboot}
\mathcal N(V)< (\sqrt{b})^{\tilde{\eta}}
\ee
\end{itemize}

The heart of our analysis is the following bootstrap proposition:

\begin{proposition}[Bootstrap]
\label{propbootstrap}
There exists $K,q,\tau^*_0\gg 1$ and $0<\tilde{\eta}\ll \eta \ll 1$ such that the following holds. Let $ \tau_0'\geq \tau_0^*$, $(b_{\ell,0}',\tau_0',a_0',\e_0')$ satisfy \fref{initmode}, \fref{condiotnazero}, \fref{orthoe}, \fref{smallnorminit}, \fref{smallnorminitsobolev}. Then there exists $(\tilde{b}_0', \tilde{b}_{j,k,0}')$ satisfying \fref{exitconditionboot} such that the solution to \eqref{selfsimilarequation} with data at $\tau_0$ given by  \eqref{estinittaldata}, \fref{decomp U} satisfies 
$$
\tau^*=+\infty.
$$
\end{proposition}

Following a now classical path, we prove Proposition \ref{propbootstrap} by contradiction and assume that $\tau^*<+\infty$ for all $(\tilde{b}_0', \tilde{b}_{j,k,0})'$ satisfying \fref{exitconditionboot}. We then study the flow on $[\tau_0,\tau^*)$ and prove that the finite dimensional exit condition \eqref{exitcondition} is necessarily saturated at $\tau^*$. The control of the modulation equations will ensure that this condition corresponds to a strictly outgoing finite dimensional vector field, hence yielding a contradiction to Brouwer's theorem.\\

\noindent The application of Brouwer's fixed point theorem requires the study of all small enough initial perturbations along the instable modes. These latter being unbounded at infinity, we use a localisation and take initial data of the form \eqref{estinittaldata}. We however now always use the decomposition \fref{decomp U} in the analysis. We claim that passing from one decomposition to another preserves the bootstrap bounds and do not give the technical but straightforward proof here.\\

From now on, we therefore consider a time interval $[\tau_0,\tau^*)$ on which \eqref{initmodeboot}, \eqref{controlea}, \eqref{exitconditionboot}, \eqref{smallnorminitboot}, \eqref{smallnorminitsobolevboot} hold.\\

\noindent We observe that the bootstrap regime implies the following bounds.

\begin{lemma}[Direct bootstrap estimates]

There exists universal constants $c=c(\ell)>0$ and $C>0$ such that on $[\tau_0,\tau^*]$:
\be \lab{bd bjk bootstrap}
C^{-1} b^{\frac \alpha 2}\leq |b_{\ell,0}| \leq C b^{\frac \alpha 2} \ \ \text{and} \ \  |b_{j,k}|\lesssim b^{\frac \alpha 2+\eta} \ \text{for} \ (j,k)\in \mathcal I, \ (j,k)\neq (\ell,0),
\ee
\be \lab{bd psi bootstrap}
|\psi|\lesssim \frac{b^{\eta}}{(\sqrt b +r)^{\frac{2}{p-1}}} \la r\ra^c\la z \ra^c, \ \ |\psi|\lesssim  \frac{b^{\frac{\alpha}{2}}}{(\sqrt b +r)^{\gamma}} \la r\ra^c\la z \ra^c
\ee
\be \lab{bd V Linftyphi bootstrap}
|V|\lesssim \frac{(\sqrt{b})^{\tilde{\eta}}}{r^{\frac{2}{p-1}}} \ \ \text{for} \ r\leq 2, \ \ |V|\lesssim (\sqrt{b})^{\tilde{\eta}}\ \ \text{for} \ r\geq 2.
\ee

\end{lemma}

\begin{proof}

\fref{bd bjk bootstrap} is a direct consequence of \fref{renormalizedbvariables} and \fref{exitcondition}. From \fref{defpsifinal}
$$
\psi=b_{\ell,0}(\psi_{\ell,0}-\psi_{0,0})+ \sum_{j=1}^{\ell-1} b_{j,0}\psi_{j,0}-\left(\sum_{j=1}^{\ell-1} b_{j,0}\right)\psi_{0,0}+\sum_{k=1}^{\ell-1}b_{0,k}\psi_{0,k}+\sum_{j=1}^{\ell-1}\sum_{k=1}^{\ell-j}b_{j,k}\psi_{j,k}.
$$
For the first term from \fref{calculeignefunction}, \fref{id expansion Ti}, \fref{bd pointwise tildephi} and \fref{bd bjk bootstrap}, there holds since $c_{i,0}=1$:
\bee
&&\left| b_{\ell,0}(\psi_{\ell,0}-\psi_{0,0}) \right| = | b_{\ell,0}|\left|\sum_{j=1}^{\ell} c_{i,j}\sqrt b^{2j-\gamma}T_j\left(\frac{r}{\sqrt{b}} \right)+\tilde \phi_{\ell}-\tilde{\phi}_0 \right| \\
&\lesssim & \sqrt b^{\alpha}\left|\sum_{j=1}^{\ell} \sqrt b^{2j-\gamma}\left(1+\frac{r}{\sqrt{b}} \right)^{2j-\gamma}+\frac{b^{\frac{g}{2}}}{(\sqrt b +r)^{\gamma}} \right|\lesssim  \frac{\sqrt b^{\alpha}}{(\sqrt b +r)^{\gamma}} \left| \left(\sqrt b+r \right)^{2}+b^{\frac{g}{2}} \right|\la r\ra^c.
\eee
For the second term from \fref{poitwisepsijk} and \fref{bd bjk bootstrap}:
\bee
&&\left|\sum_{j=1}^{\ell-1} b_{j,0}\psi_{j,0}-\left(\sum_{j=1}^{\ell-1} b_{j,0}\right)\psi_{0,0}+\sum_{k=1}^{\ell-1}b_{0,k}\psi_{0,k}+\sum_{j=1}^{\ell-1}\sum_{k=1}^{\ell-j}b_{j,k}\psi_{j,k} \right| \\
&\lesssim &\sum_{\mathcal I \backslash \{(\ell,0) \}} |b_{j,k}|\frac{1}{(\sqrt b +r )^{\gamma}}\la r\ra^c\la z \ra^c \lesssim \frac{b^{\frac{\alpha}{2}+\eta}}{(\sqrt b +r)^{\gamma}} \la r\ra^c\la z \ra^c
\eee
and \fref{bd psi bootstrap} follows from the two above identities. Finally, \fref{bd V Linftyphi bootstrap} is a direct consequence of \fref{smallnorminitsobolevboot} and \fref{boundlinfty}.

\end{proof}

%%%%%%%%%%%%%%%%%%%%%%%%%%%%%%%%%%%%%%%%%%%%%%%%%%%%%%%%%%%
%%%%%%%%%%%%%%%%%%%%%%%%%%%%%%%%%%%%%%%%%%%%%%%%%%%%%%%%%%%

\subsection{Modulation equations}

%%%%%%%%%%%%%%%%%%%%%%%%%%%%%%%%%%%%%%%%%%%%%%%%%%%%%%%%%%%
%%%%%%%%%%%%%%%%%%%%%%%%%%%%%%%%%%%%%%%%%%%%%%%%%%%%%%%%%%%

Injecting the decomposition \eqref{deompe} and the identity \fref{eqautionprofile} into \eqref{selfsimilarequation} yields the $\e$ equation:
\be
\label{eqe}
\pa_\tau\e+\L_b\e=-\Psi+L(V)+\NL(V), \ \ V=\psi+\e
\ee 
with 
\bea
\label{defpsi}
&&\Psi=\Psi_1+\Psi_2, \ \ \Psi_2=\pa_\tau\psi+\L_b\psi\\
\label{defle}
&&L(V)=p(\Phi_{a,b}^{p-1}-Q_b^{p-1})V\\
\label{defnlv}
&& NL(V)=(\Phi_{a,b}+V)|\Phi_{a,b}+V|^{p-1}-\Phi_{a,b}|\Phi_{a,b}|^{p-1}-p\Phi^{p-1}_aV
\eea
We start the study of the flow on $[\tau_0,\tau^*)$ by deriving the modulation equations for the geometrical parameters $(a,b,b_{j,k})$ as a consequence of the orthogonality conditions \eqref{orthoe} and the non degeneracies \eqref{nondegenone}, \eqref{nondegenonebis}. We let 
\be
\label{defgibbg}
B=\frac12 \left(-\frac{b_\tau}{b}+1\right)-\frac{\ell}{\alpha}.
\ee 
In order to prepare the outgoing flux argument, we compute the modulation equations in the variables \eqref{renormalizedbvariables}.

\begin{lemma}[Modulation equations]
\label{lemmaeqmodulation}
Let 
\bee
\Mod & = & \sum_{(j,k)\in \mathcal I}\left|\pa_\tau \bt_{j,k}-(\ell-(k+j))\bt_{j,k}\right|+|a_\tau|+\left|\bt_\tau-\ell\bt+\sum_{j=1}^{\ell-1}\bt_{j,0}\right|\\
&&+ \sum_{(j,k)\in \mathcal I}\frac{|\pa_\tau b_{j,k}+\l_{j,k}b_{j,k}|}{(\sqrt{b})^\alpha},
\eee

then there holds for some small enough  universal constant $\delta>0$:
\be
\label{estparameters}
\Mod\lesssim  \frac{(a+(\sqrt{b})^{\tilde \eta})\|\e\|_{L^2_{\rho_Y}}}{(\sqrt{b})^\alpha}+(\sqrt{b})^\delta+|a|\left(|\bt|+\sum_{j=1}^{\ell-1}|\bt_{j,0}|\right)
\ee

\end{lemma}

\begin{proof}[Proof of Lemma \ref{lemmaeqmodulation}] 
We take the $L^2_{\rho_Y}$ scalar product of \eqref{eqe} with $\psi_{j,k}$ for $0\leq j \leq \ell$ and $0\leq k \leq j-1$, and compute the resulting identity using \eqref{orthoe}:
\be
\label{identitymodulation}
(\Psi_1,\psi_{j,k})_{L^2_{\rho_Y}}+(\Psi_2,\psi_{j,k})_{L^2_{\rho_Y}}=(\e,\pa_\tau\psi_{j,k})_{L^2_{\rho_Y}}+(L(V)+N(V),\psi_{j,k})_{L^2_{\rho_Y}}.
\ee
We now estimate all terms in the above identity.\\

\noindent{\bf step 1} Scalar products on $\Psi_1$. We first compute from \eqref{formulapsi}, \fref{defgibbg}:
\bee
&&E_{j,k}:= (\Psi_1,\psi_{j,k})_{L^2_{\rho_Y}}= \left(\left[B-\frac{1}{2\alpha}\frac{\pa_\tau\nu}{1+\nu}\right]\Lambda_r\Phi_{a,b} +\frac{\ell}{\alpha}\Lambda Q_b(r),\psi_{j,k}\right)_{L^2_{\rho_Y}}+(\tilde{\Psi}_1,\psi_{j,k})_{L^2_{\rho_Y}}.
\eee

\noindent \emph{The lower order term}. We claim the degeneracy
\be
\label{degenpsitildeun}
\|\tilde\Psi_1\|_{H^1_{\rho_Y}}\lesssim (\sqrt{b})^{\alpha} b^\delta , \ \ |(\tilde{\Psi}_1,\psi_{j,k})_{L^2_{\rho_Y}}|\lesssim (\sqrt{b})^{\alpha} b^\delta
\ee
The second estimate in the above identity is a direct consequence of the first one. To prove the first one, recalling \fref{psitildeone}, we estimate using \eqref{controlenorme} and $\mu\geq \frac 12$:
\bee
&&\|\frac{1}{(\sqrt{b})^\alpha}\frac{1}{\alpha^2}\left(\frac{\pa_z\nu}{1+\nu}\right)^2\frac{1}{\mu^{\frac{2}{p-1}}}(\Lambda^2Q_b+\alpha\Lambda Q_b)\left(\frac{r}{\mu}\right)\|^2_{H^1_{\rho_Y}}\lesssim  \sum_{k=0,1} \int b^\delta \frac{\mu^{C}}{r^{2\gamma+2k}} \rho_YdY\lesssim b^\delta,
\eee
and similarly:

\bea
\label{eecnoenoen}
\nonumber &&\frac{1}{\sqrt{b}^\alpha}\left\|\frac{1}{\mu^{\gamma}}\Lambda Q_b\left(\frac{r}{\mu}\right)-\Lambda Q_b(r)\right\|_{H^1_{\rho_r}}\\&=& \sum_{k=0,1} \frac{1}{\sqrt{b}^\alpha}\left\| \int_1^\mu\frac{1}{(\mu')^{\gamma+1}}\partial_r^k\left[-\gamma\Lambda Q_b-r\pa_r(\Lambda Q_b)\right]\left(\frac{r}{\mu'}\right)d\mu'\right\|_{L^2_{\rho_r}}\\
\non & \lesssim & \sum_{k=0,1} \frac{1}{\sqrt{b}^\alpha}\int_1^\mu\frac{1}{(\mu')^{\gamma+1}}\left\|\partial_r^k(\alpha\Lambda Q_b+\Lambda^2Q_b)\left(\frac{r}{\mu'}\right)\right\|_{L^2_{\rho_r}} d\mu'\lesssim b^\delta\int_1^\mu(\mu')^C \lesssim b^\delta \mu^C
\eea

from which using the exponential weight in $z$:
$$\frac{1}{\sqrt{b}^\alpha}\left\|\frac{1}{\mu^{\gamma}}\Lambda Q_b\left(\frac{r}{\mu}\right)-\Lambda Q_b(r)\right\|^2_{H^1_{\rho_Y}}\lesssim b^{\delta}\int \mu^C\rho_zdz\lesssim b^{\delta},$$
and from \fref{psitildeone} and the above identities, \eqref{degenpsitildeun} is proved.\\

\noindent \emph{The main order term}. By definition: $$-\frac{\pa_\tau\nu}{\alpha(1+\nu)}\Lambda_r\Phi_{a,b} =a_{\tau}\pa_a\Phi_{a,b}$$ and hence from \eqref{nondegenonebis}:
$$
\left(-\frac{1}{2\alpha}\frac{\pa_\tau\nu}{1+\nu}\Lambda_r\Phi_{a,b},\psi_{j,k}\right)_{L^2_{\rho_Y}}=-\frac{a_\tau}{2\alpha}(\sqrt{b})^{\alpha}\left[I\delta_{j0}\delta_{k\ell}+O(b^\delta+a)\right].
$$
Similarly, $$\Lambda_r\Phi_{a,b}=-2b\pa_b\Phi_{a,b}$$ and hence from \eqref{nondegenone}:
$$\left(B\Lambda_r\Phi_{a,b} ,\psi_{j,k}\right)_{L^2_{\rho_Y}}=B(\sqrt{b})^{\alpha}\left[I\delta_{j0}\delta_{k0}+O(b^\delta+a)\right].$$
Finally from \eqref{scalrapronienoevni}: $$\left(\Lambda Q_b,\psi_{j,k}\right)_{L^2_{\rho_r}}=(\sqrt{b})^\alpha\left[I\delta_{j0}\delta_{k0}+O(b^\delta)\right].$$
Using \fref{degenpsitildeun}, this yields the values:
\bea
\label{calculezerozero}
\frac{E_{0,0}}{(\sqrt{b})^{\alpha}}&=&B\left[I+O(b^\delta+a)\right]+\frac{\ell}{\alpha}\left(I+O(b^\delta)\right)+ a_\tau O\left(b^\delta +a\right),
\eea
and
\be
\label{calculezero2l}
\frac{E_{0,\ell}}{(\sqrt{b})^{\alpha}}=-\frac{a_\tau}{2\alpha}\left[I+O(b^{\delta}+a)\right]+BO(b^{\delta}+a)+O(b^\delta)
\ee
and for $(j\geq 1, k\geq 0)$ and $(j=0, 1\leq k\leq \ell-1)$:
\be
\label{estimationejk}
\frac{E_{j,k}}{(\sqrt{b})^\alpha}= BO(b^{\delta}+a)+ a_\tau O(b^{\delta}+a)+O( b^\delta).
\ee

\noindent{\bf step 2} $\Psi_2$ term. We compute from \eqref{defpsifinal}, \eqref{defpsi}:
\be
\label{expressionpsitwo}
\Psi_2 = \sum_{(j,k)\in \mathcal I}(\pa_\tau b_{j,k}+\l_{j,k}b_{j,k})\psi_{j,k}-\left(\sum_{j=1}^\ell(\pa_\tau b_{j,0}+\l_{0,0}b_{j,0})\right)\psi_{0,0}+\tilde{\Psi}_2
\ee
with
\bea
\label{expresionpsitwotilde}
\Psit_2= \frac{b_\tau}{b}\left[\sum_{j=1}^\ell b_{j,0}(b\pa_b\psi_{j,0}-b\pa_b\psi_{0,0})+\sum_{k=1}^{\ell-1}b_{0,k}b\pa_b\psi_{0,k}+\sum_{j=1}^{\ell}\sum_{k=0}^{\ell-j}b_{j,k}b\pa_b\psi_{j,k}\right].
\eea
From \eqref{bound:partialb}, \fref{bd bjk bootstrap}:
\be
\label{esterroru}
\|\Psit_2\|_{H^1_{\rho_Y}}= \frac{b_\tau}{b}O\left(b^\delta\sum_{j,k}|b_{j,k}|\right)= \frac{b_\tau}{b}O\left(b^\delta(\sqrt{b})^\alpha\right)
\ee
and hence using the $L^ 2_{\rho_Y}$ orthogonality of eigenfunctions:
\bea
\label{computationpsijktwo}
\nonumber (\Psi_2,\psi_{j,k})_{L^2_{\rho_Y}}&=&\left|\begin{array}{ll}\left(\pa_\tau b_{j,k}+\l_{j,k}b_{j,k}\right)\|\psi_{j,k}\|_{L^2_{\rho_Y}}^2\ \ \mbox{for}\ \ (j,k)\in \mathcal I\\
-\sum_{j=1}^\ell(\pa_\tau b_{j,0}+\l_{0,0}b_{j,0})\|\psi_{0,0}\|_{L^2_{\rho_Y}}^2 \ \ \mbox{for}\ \ (j,k)= (0,0) \\
0 \ \ \text{for} \ (j,k)=(0,\ell)
\end{array}\right.\\
&&+\frac{b_\tau}{b}O\left(b^{\delta}(\sqrt{b})^\alpha\right).
\eea

\noindent{\bf step 3} $L(V)$ terms. We now estimate the rhs of \eqref{identitymodulation} which are lower order and start with the $L(V)$ term given by \eqref{defle}. We claim:
\be
\label{estlv}
\left|(L(V),\psi_{j,k})_{L^2_{\rho_Y}}\right|\lesssim b^\delta(\sqrt{b})^\alpha+a\|\e\|_{L^2_{\rho_Y}}.
\ee

We compute:
\be
\label{deriveeqmu}
\frac{d}{d\mu}\Phi_{a,b}^{p-1}=-\frac{p-1}{\mu^{1+\frac{2}{p-1}}}\Phi_{a,b}^{p-2}\Lambda Q_b\left(\frac{r}{\mu\sqrt{b}}\right)
\ee
We now use the global bounds $$|Q(y)|\lesssim \frac{1}{1+|y|^{\frac{2}{p-1}}}, \ \ |\Lambda Q(y)|\lesssim \frac{1}{1+|y|^\gamma}$$ to estimate:
\bea
\label{enkn eonenekno}
\nonumber&& \left|\frac{d}{d\mu}\Phi_{a,b}^{p-1}\right|\lesssim \frac{1}{\mu}\left[\frac{1}{(r+\mu\sqrt{b})^{\frac{2}{p-1}}}\right]^{p-2}\frac{(\mu\sqrt{b})^\alpha}{(r+\mu\sqrt{b})^\gamma}\\
\label{enoneoneon}&\lesssim & \frac{(\mu\sqrt{b})^\alpha}{(r+\mu\sqrt{b})^\alpha}\frac{1}{\mu(r+\mu\sqrt{b})^2}\\
& \lesssim &  \min\left\{\frac{\mu^{\alpha-3}(\sqrt{b})^{\alpha-2}}{r^\alpha}, \frac{1}{\mu r^2},\frac{(\mu\sqrt{b})^\alpha}{\mu}\frac{1}{(r+\sqrt{b})^{\alpha+2}}\right\}
\eea
where we used $\mu\geq \frac 12$ for the last estimate,
and hence the pointwise bound by integration in $\mu$:
\be
\label{esterrorpotential}
\left|\Phi_{a,b}^{p-1}-Q_b^{p-1}\right|\lesssim \min\left\{\frac{(\mu^{\alpha-2}-1)(\sqrt{b})^{\alpha-2}}{r^\alpha},\frac{|\log \mu|}{r^2},\frac{|\mu^\alpha-1|(\sqrt{b})^\alpha}{(r+\sqrt{b})^{\alpha+2}}\right\}.
\ee
This allows us to control the $\psi$ term using the pointwise bound \fref{bd psi bootstrap} and \fref{poitwisepsijk}:
\bee
\left|(L(\psi),\psi_{j,k})_{L^2_{\rho_Y}}\right| & \lesssim & \int_{r\leq \sqrt b} \frac{|\text{log}\mu |}{r^2}\frac{b^{\eta}\la z \ra^c}{b^{\frac{1}{p+1}}}\frac{1}{b^{\frac{\gamma}{2}}}\rho_YdY+\int_{r\geq \sqrt b} \frac{\mu^{\alpha-2}\sqrt{b}^{\alpha-2}}{r^{\alpha}}\frac{\sqrt{b}^{\alpha}}{r^{\gamma}}\frac{\la r\ra^c \la z\ra^c}{r^{\gamma}}\rho_YdY\\
& \lesssim & b^{\frac{d}{2}-1-\frac{1}{p-1}-\frac{\gamma}{2}+\eta}+b^{\alpha-1}\text{max}(1,b^{\frac{d}{2}-\gamma-\frac{\alpha}{2}}) \lesssim b^{\frac{\alpha}{2}+\delta}
\eee
since $\alpha>2$ and $d-2\gamma-2>0$. We then control the $\e$ term using the second estimate in \eqref{esterrorpotential}, \eqref{poitwisepsijk}:
\bee
&&|\left((\Phi_{a,b}^{p-1}-Q_b^{p-1})\e,\psi_{j,k}\right)_{L^2_{\rho_Y}}|\lesssim  \int |\e|\frac{a}{r^{2}}\frac{\la z\ra^c+\la r\ra^c}{r^\gamma}\rho_YdY\\
&\lesssim &a\left(\int \e^2\rho_YdY\right)^{\frac 12}\left(\int \frac{\la z\ra ^c\la r\ra^c}{r^{2\gamma+4}}\rho_YdY\right)^{\frac 12}
\lesssim a\|\e\|_{L^2_{\rho_Y}}
\eee
where we used in the last step \eqref{confitiondiemsnion} which ensures: 
\be
\label{cnenceoneno}
d-(2\gamma-4)=\sqrt{\Delta}-2>0
\ee 
to ensure the convergence of the integrals at the origin.
 This concludes the proof of \eqref{estlv}.\\

\noindent{\bf step 4} $\NL(V)$ terms. We claim the bound for some universal $\delta>0$ small enough:
\be
\label{borneamontrer}
|(\NL(v),\psi_{j,k})_{L^2_{\rho_Y}}|\lesssim (\sqrt{b})^{\alpha+\delta}+(\sqrt{b})^{\tilde{\eta}}\|\e\|_{L^2_{\rho_Y}}.
\ee
First observe that \fref{bd psi bootstrap} yields the two following bounds for $C,c>0$ and $0< \delta<\delta^*$ universal small enough:
  \be
  \label{estnonlinearvnkevone}
  \int\frac{\psi^2}{r^{2+2\delta}}\rho_YdY\lesssim b^{\alpha}\int \frac{\la r \ra^C \la z \ra^C}{r^{2\gamma+2+2\delta}}\rho_YdY \lesssim b^{\alpha}
  \ee
  \be
  \label{estlinftypsi}
  \|r^{\frac 2{p-1}+\delta}\psi\|_{L^\infty(r\le 1, |z|\leq K\sqrt{|\log b|})}\lesssim (\sqrt{b})^{\eta+c\delta}|K\log b|^C\lesssim (\sqrt{b})^{c\delta}.
  \ee
and that from \fref{bd V Linftyphi bootstrap} and \fref{positivitylamdbaq}:
\be \lab{bd pointwise Phiab V}
|\Phi_{a,b}|\lesssim \frac{1}{r^{\frac{2}{p-1}}}, \ \ |V|\lesssim \frac{(\sqrt{b})^{\tilde{\eta}}}{r^{\frac{2}{p-1}}} \ \text{for} \ r\leq 2, \ \ |V|\lesssim (\sqrt{b})^{\tilde{\eta}} \ \text{for} \ r\geq 2 .
\ee

\noindent{\em Estimate for $|z|>M\sqrt{|\log b|}$}: We estimate from the above identity, using the Gaussian weight and the fact that $d-2-2/(p-1)-\gamma>d-2\gamma>0$:
  \bee
&&  \int_{|z|\geq M\sqrt{|\log b|}} |\NL(V)||\psi_{j,k}|\rho_YdY \lesssim \int_{|z|\geq M\sqrt{|\log b|}} (|V|^{p}+\Phi_{a,b}^{p-2}|V|^2)|\psi_{j,k}|\rho_YdY\\
& \lesssim &  \int_{|z|\geq M\sqrt{|\log b|}} \left(1+\frac{1}{r^{2+\frac{2}{p-1}}} \right)\frac{\la r \ra^c \la z \ra^c}{r^{\gamma}}\rho_YdY\lesssim (\sqrt{b})^{cM} \leq (\sqrt{b})^{\alpha+\delta}
 \eee
 for $M$ universal large enough.\\
 
 \noindent{\em Estimate for $|z|<M\sqrt{|\log b|}$}. First one decomposes
  \bee
  &&\int_{|z|\leq M\sqrt{|\log b|}} |\NL(V)|\psi_{j,k}\rho_YdY\lesssim\int_{|z|\leq M\sqrt{|\log b|}}\left( |V|^p+\Phi_{a,b}^{p-2}|V|^2\right)|\psi_{j,k}|\rho_YdY\\
 &\lesssim& \int_{|z|\leq M\sqrt{|\log b|}}\left( |\e|^p+|\psi|^p+ \Phi_{a,b}^{p-2}|\e|^2+ \Phi_{a,b}^{p-2}|\psi|^2\right)|\psi_{j,k}|\rho_Y dY
 \eee
 Near the origin, we estimate using \eqref{estlinftypsi} and \fref{bd pointwise Phiab V}
 \bee
&& \int_{r\le 1, |z|\leq K\sqrt{|\log b|}}(|\e|^p+\Phi_{a,b}^{p-2}|\e|^2)|\psi_{jk}|\rho_YdY\\
&\lesssim& \int_{r\leq 1}\left[\|r^{\frac{2}{p-1}}V\|_{L^\infty(r\le 1)}+\|r^{\frac{2}{p-1}}\psi\|_{L^\infty(r\le 1,|z|\leq K\sqrt{|\log b|})}\right]\frac{|\e\psi_{j,k}|}{r^2}\rho_YdY\\
 & \lesssim &((\sqrt b)^{\tilde{\eta}}+(\sqrt{b})^{\eta})\|\e\|_{L^2_{\rho_Y}} \lesssim (\sqrt{b})^{\tilde{\eta}}\|\e\|_{L^2_{\rho_Y}} 
  \eee
 where we used \eqref{cnenceoneno}, and similarly:
  \bee
 &&\int_{r\le 1,|z|\leq K\sqrt{|\log b|}}(|\psi|^p+\Phi_{a,b}^{p-2}|\psi|^2)|\psi_{j,k}|\rho_YdY\\
 & \lesssim& \|r^{\frac{2}{p-1}+\delta}\psi\|_{L^\infty(r\le 1,|z|\leq K\sqrt{|\log b|})}\int_{r\le 1,|z|\leq K\sqrt{|\log b|}}\frac{|\psi|}{r^\delta}|\psi_{j,k}|\frac{dY}{r^2}\\
 &\lesssim &\|r^{\frac{2}{p-1}+\delta}\psi\|_{L^\infty(r\le 1,|z|\leq K\sqrt{|\log b|})}\left(\int \frac{\psi^2}{r^{2+2\delta}}\rho_YdY\right)^{\frac 12}\left(\int \frac{\psi_{j,k}^2}{r^2}\rho_YdY\right)^{\frac 12} \leq (\sqrt{b})^{\alpha+c\delta}.
 \eee
 Away from the origin, we use the Gaussian weight in $r$, \fref{bd V Linftyphi bootstrap} and \fref{bd psi bootstrap} to estimate:
 \bee
 &&\int_{r\geq 1,|z|\leq K\sqrt{|\log b|}}\left( |V|^p+\Phi_{a,b}^{p-2}|V|^2\right)|\psi_{j,k}|\rho_YdY\lesssim \int_{r\ge 1,|z|\leq K\sqrt{|\log b|}}|V|^2\la r \ra^c \la z \ra^c\rho_YdY\\
 & \lesssim & \int_{r\ge 1}|\varepsilon |^2\la r \ra^c \la z \ra^c\rho_YdY+\int_{r\ge 1}|\psi |^2\la r \ra^c \la z \ra^c\rho_YdY \lesssim   \int_{r\ge 1}|\varepsilon |(|V|+|\psi|)\la r \ra^c \la z \ra^c\rho_YdY+b^{\alpha}\\
 &\lesssim & \|\varepsilon \|_{L^2_{\rho_Y}}\left(\| V \|_{L^{\infty}_{\rho_Y}(r\geq 1)}+\|\la r\ra^c\la z \ra^c\psi \|_{L^2_{\rho_Y}}\right)+b^{\alpha} \lesssim (\sqrt b)^{\tilde{\eta}} \|\varepsilon \|_{L^2_{\rho_Y}}+b^{\alpha}
 \eee
 This concludes the proof of \eqref{borneamontrer}.\\
 
\noindent{\bf step 6} Computation of the modulation equations. We estimate from \eqref{bound:partialb}, \eqref{defgibbg}:
$$|(\e,\pa_\tau\psi_{j,k})_{L^2_{\rho_Y}}|=\frac{b_\tau}{b}O(b^\delta\|\e\|_{L^2_{\rho_Y}})=BO(b^\delta\|\e\|_{L^2_{\rho_Y}})+O(b^\delta\|\e\|_{L^2_{\rho_Y}}).$$ Injecting this together with \eqref{estlv}, \eqref{borneamontrer} and \fref{smallnorminitboot} into \eqref{identitymodulation} yields:
$$(\Psi,\psi_{j,k})_{L^2_{\rho_Y}}= BO(b^\delta\|\e\|_{L^2_{\rho_Y}})+O\left((a+(\sqrt{b})^{\tilde{\eta}})\|\e\|_{L^2_{\rho_Y}}\right)+O((\sqrt{b})^{\alpha+\delta}).$$ We then combine  the estimates \eqref{computationpsijktwo}, \eqref{calculezerozero}, \eqref{calculezero2l}, \eqref{estimationejk} and obtain the following:\\

\noindent\underline{\em law for $b_{j,k}$,  $(j,k)\neq\{(0,0);(0,\ell)\}$}. We obtain: 
\bee
\label{cekencnoeco}
\non &&(\pa_\tau b_{j,k}+\l_{j,k}b_{j,k})\|\psi_{j,k}\|_{L^2_{\rho_Y}}^2= BO\left(a(\sqrt{b})^\alpha+(\sqrt b)^{\alpha+\delta}+b^{\delta}\|\e\|_{L^2_{\rho_Y}}\right)\\
&+&  a_\tau (O(a(\sqrt{b})^\alpha)+O(\sqrt b^{\alpha+\delta}))+ O\left((\sqrt{b})^{\alpha+\delta}\right)+  O\left((a+(\sqrt b)^{\tilde{\eta}})\|\e\|_{L^2_{\rho_Y}}\right).
\eee
which implies using \eqref{renormalizedbvariables}, \fref{bd bjk bootstrap} and \eqref{valrupropre}:
\bea
\label{cjbeonoenoe}
\nonumber
&&(\sqrt{b})^\alpha\left[\pa_\tau \bt_{j,k}-(\ell-(k+j))\bt_{j,k}\right]=BO\left(a(\sqrt{b})^\alpha+(\sqrt b)^{\alpha+\delta}+b^\delta\|\e\|_{L^2_{\rho_Y}}\right)\\
&+&  a_\tau (O(a(\sqrt{b})^\alpha)+O(\sqrt b^{\alpha+\delta}))+ O\left((\sqrt{b})^{\alpha+\delta}\right)+  O\left((a+(\sqrt b)^{\tilde{\eta}})\|\e\|_{L^2_{\rho_Y}}\right).
\eea

\noindent\underline{\em law for $a$}. We obtain:
\bee
\nonumber -\frac{a_\tau}{2\alpha}I(\sqrt{b})^\alpha &= &BO\left(a(\sqrt{b})^\alpha+(\sqrt b)^{\alpha+\delta}+b^\delta\|\e\|_{L^2_{\rho_Y}}\right)+ a_\tau (O(a(\sqrt{b})^\alpha)+O(\sqrt b^{\alpha+\delta}))\\
 &+& O\left((\sqrt{b})^{\alpha+\delta}\right)+  O\left((a+(\sqrt b)^{\tilde{\eta}})\|\e\|_{L^2_{\rho_Y}}\right).
\eee
which can be rewritten as, since $|a|,b\ll 1$ and $I\neq 0$:
\bea
\label{ceknenoneeeno}
\nonumber -\frac{a_\tau}{2\alpha}I(\sqrt{b})^\alpha &= &BO\left(a(\sqrt{b})^\alpha+(\sqrt b)^{\alpha+\delta}+b^\delta\|\e\|_{L^2_{\rho_Y}}\right)\\
 &+& O\left((\sqrt{b})^{\alpha+\delta}\right)+  O\left((a+(\sqrt b)^{\tilde{\eta}})\|\e\|_{L^2_{\rho_Y}}\right).
\eea

\noindent \underline{\em law for $b$}.  Finally:
\bee
&&(\sqrt{b})^\alpha BI+I\frac{\ell}{\alpha}(\sqrt{b})^\alpha- \sum_{j=1}^\ell(\pa_\tau b_{j,0}+\l_{0,0}b_{j,0})\|\psi_{0,0}\|_{L^2_{\rho_Y}}^2\\
\nonumber &= & BO\left(a(\sqrt{b})^\alpha+(\sqrt b)^{\alpha+\delta}+b^\delta\|\e\|_{L^2_{\rho_Y}}\right)+ a_\tau (O(a(\sqrt{b})^\alpha)+O(\sqrt b^{\alpha+\delta}))\\
 &+& O\left((\sqrt{b})^{\alpha+\delta}\right)+  O\left(a+(\sqrt b)^{\tilde{\eta}})\|\e\|_{L^2_{\rho_Y}}\right).
\eee
Since from \fref{calculeignefunction} and \fref{definitioni}, $\|\psi_{0,0}\|_{L^2_{\rho_Y}}^2=I+O(b^{\delta})$, this last expression can be reformulated 
\bee
&&(\sqrt{b})^\alpha B+\frac{\ell}{\alpha}(\sqrt{b})^\alpha- \sum_{j=1}^\ell(\pa_\tau b_{j,0}+\l_{0,0}b_{j,0})\\
\nonumber &= & BO\left(a(\sqrt{b})^\alpha+(\sqrt b)^{\alpha+\delta}+b^\delta\|\e\|_{L^2_{\rho_Y}}\right)+ a_\tau (O(a(\sqrt{b})^\alpha)+O(\sqrt b^{\alpha+\delta}))\\
 &+& O\left((\sqrt{b})^{\alpha+\delta}\right)+  O\left((a+(\sqrt b)^{\tilde{\eta}})\|\e\|_{L^2_{\rho_Y}}\right).
\eee
We now reformulate these estimates using the renormalized variables \eqref{renormalizedbvariables}.
 First recalling \eqref{defgibbg} and using \eqref{valrupropre}:
\be
\label{estimateB}
B=\frac12 \left[-\frac2{\alpha}\left(\frac{\pa_{\tau}b_{\ell,0}}{b_{\ell,0}}+\frac{\pa_\tau\tilde{b}}{1+\tilde{b}}\right)+1\right]-\frac{\ell}{\alpha}=-\frac 1{\alpha}\frac{\pa_\tau\tilde{b}}{1+\tilde{b}}-\frac 1\alpha\frac{\pa_{\tau}b_{\ell,0}+[\l_{\ell,0}+O(b^\delta)])b_{\ell,0}}{b_{\ell,0}}
\ee
and hence from \eqref{bd bjk bootstrap}, \eqref{cjbeonoenoe}:
\bea
\label{estimportante}
\nonumber (\sqrt{b})^\alpha\left(B+\frac 1{\alpha}\frac{\pa_\tau\tilde{b}}{1+\tilde{b}}\right)&=& BO\left(a(\sqrt{b})^\alpha+(\sqrt b)^{\alpha+\delta}+b^\delta\|\e\|_{L^2_{\rho_Y}}\right)+ a_\tau (O(a(\sqrt{b})^\alpha)+O(\sqrt b^{\alpha+\delta}))\\
 &+& O\left((\sqrt{b})^{\alpha+\delta}\right)+  O\left((a+(\sqrt b)^{\tilde\eta})\|\e\|_{L^2_{\rho_Y}}\right).
 \eea
Moreover from \eqref{cjbeonoenoe} again and \eqref{valrupropre}, \fref{bd bjk bootstrap}:
\bee
&&\frac{\ell}{\alpha}(\sqrt{b})^\alpha-\sum_{j=1}^\ell(\pa_\tau b_{j,0}+\l_{0,0}b_{j,0})= \frac{\ell}{\alpha}(\sqrt{b})^\alpha-\sum_{j=1}^\ell(\pa_\tau b_{j,0}+\l_{j,0}b_{j,0})+\sum_{j=1}^{\ell}jb_{j,0}+O(b^\delta(\sqrt{b}^\alpha))\\
&=&\frac{\ell}{\alpha}(\sqrt{b})^\alpha+\ell b_{\ell,0}+\sum_{j=1}^{\ell-1}jb_{j,0}+  BO\left(a(\sqrt{b})^\alpha+(\sqrt b)^{\alpha+\delta}+b^\delta\|\e\|_{L^2_{\rho_Y}}\right) \\
&&+a_\tau (O(a(\sqrt{b})^\alpha)+O(\sqrt b^{\alpha+\delta}))+O\left((\sqrt{b})^{\alpha+\delta}\right)+  O\left((a+(\sqrt b)^{\tilde \eta})\|\e\|_{L^2_{\rho_Y}}\right).\eee
Moreover 
$$\frac{\ell}{\alpha}(\sqrt{b})^\alpha+\ell b_{\ell,0} = \ell\left[\frac{(\sqrt{b})^\alpha}{\alpha}+ b_{\ell,0}\right]=-\ell b_{\ell,0}\tilde{b}=\frac{\ell}{\alpha}(\sqrt{b})^\alpha \frac{\tilde{b}}{1+\tilde{b}}$$ 
and hence the bound:
\bee
&&\frac{1}{\alpha}(\sqrt{b})^\alpha\left[\frac{-\pa_\tau \tilde{b}+\ell\tilde{b}}{1+\tilde{b}}\right]+\sum_{j=1}^{\ell-1}jb_{j,0}=\pa_\tau\tilde{b}O(a(\sqrt{b})^\alpha+b^{\frac{\alpha}{2}+\delta}+b^\delta\|\e\|_{L^2_{\rho_Y}}) \\
&&+a_\tau (O(a(\sqrt{b})^\alpha)+O(\sqrt b^{\alpha+\delta}))+O\left((\sqrt{b})^{\alpha+\delta}\right)+  O\left((a+(\sqrt b)^{\tilde \eta})\|\e\|_{L^2_{\rho_Y}}\right).
\eee
Now using \eqref{renormalizedbvariables}: $$b_{j,0}=\bt_{j,0}b_{\ell,0}=-\frac{(\sqrt{b})^\alpha}{\alpha}\frac{\bt_{j,0}}{1+\bt}$$
and hence the law:
\bea
\label{lqsnconoe}
\nonumber &&(\sqrt{b})^\alpha\left[\bt_\tau-\ell\bt+\sum_{j=1}^{\ell-1}j\bt_{j,0}\right]=\pa_\tau\tilde{b}O(a(\sqrt{b})^\alpha+b^{\frac{\alpha}{2}+\delta}+b^\delta\|\e\|_{L^2_{\rho_Y}})+a_\tau (O(a(\sqrt{b})^\alpha)\\
&+& O(\sqrt b^{\alpha+\delta}))+O\left((\sqrt{b})^{\alpha+\delta}\right)+  O\left((a+(\sqrt b)^{\tilde \eta})\|\e\|_{L^2_{\rho_Y}}\right).
\eea

\noindent{\bf step 7} Conclusion. The estimates \eqref{cjbeonoenoe}, \eqref{ceknenoneeeno}, \eqref{lqsnconoe} together with \eqref{estimportante} 
yield the system
\bee
&&\sum_{(j,k)\in \mathcal I}(\sqrt{b})^\alpha\left|\pa_\tau \bt_{j,k}-(\ell-(k+j))\bt_{j,k}\right|+(\sqrt{b})^\alpha|a_\tau|+(\sqrt{b})^\alpha\left|\bt_\tau-\ell\bt+\sum_{j=1}^{\ell-1}j\bt_{j,0}\right|\\
&\lesssim&O(|\pa_\tau\tilde{b}|)O(a(\sqrt{b})^\alpha+b^{\frac{\alpha}{2}+\delta}+b^\delta\|\e\|_{L^2_{\rho_Y}})+|a_\tau |(O(a(\sqrt{b})^\alpha)+O(\sqrt b^{\alpha+\delta}))\\
&&+O\left((\sqrt{b})^{\alpha+\delta}\right)+  O\left((a+(\sqrt b)^{\tilde{\eta}})\|\e\|_{L^2_{\rho_Y}}\right).
\eee
which is invertible thanks to \eqref{controlea}, \eqref{smallnorminitboot} and implies:
\bee
&&\sum_{(j,k)\in \mathcal I}\left|\pa_\tau \bt_{j,k}-(\ell-(k+j))\bt_{j,k}\right|+|a_\tau|+\left|\bt_\tau-\ell\bt+\sum_{j=1}^{\ell-1}j\bt_{j,0}\right|\\
&\lesssim & \frac{(a+(\sqrt{b})^{\tilde \eta})\|\e\|_{L^2_{\rho_Y}}}{(\sqrt{b})^\alpha}+b^\delta +|a|\left(|\bt|+\sum_{j=1}^{\ell-1}|\bt_{j,0}|\right)
\eee
this is \eqref{estparameters}. Injecting this into \eqref{cekencnoeco} with \eqref{estimportante} yields \eqref{estparameters} and concludes the proof of Lemma \ref{lemmaeqmodulation}.
\end{proof}

%%%%%%%%%%%%%%%%%%%%%%%%%%%%%%%%%%%%%

\subsection{Local $L^2_{\rho_Y}$ bound}

%%%%%%%%%%%%%%%%%%%%%%%%%%%%%%%%%%%%%

The geometrical decomposition \eqref{deompe} build on the eigenbasis constructed in Proposition \ref{diaglb} yields an elementary setting to compute the modulation equations of Lemma \ref{lemmaeqmodulation} and the underlying outgoing vector field structure. A second elementary fruit is the control of the flow in the $L^2_{\rho_Y}$ topology.

\begin{lemma}[$L^2_{\rho_Y}$ control]
\label{lemmal2loc}
There holds the pointwise bound: 
\be
\label{controlponctuelenergy}
\|\e\|_{L^2_{\rho_Y}}\lesssim \eta(a)(\sqrt{b})^{\alpha+\eta}
\ee
where $\eta(a)=o(1)$ as $a\to 0$.

\end{lemma}

\begin{proof}[Proof of Lemma \ref{lemmal2loc}] We claim that \eqref{controlponctuelenergy} follows from the differential inequality
\bea
\label{controlenormee}
\nonumber&& \frac 12\frac{d}{d\tau}\|\e\|_{L^2_{\rho_Y}}^2+\l_{\ell,0}\|\e\|_{L^2_{\rho_Y}}^2+\frac{c^*}{2}\|\e\|_{H^1_{\rho_Y}}^2\\
&\lesssim &(\sqrt{b})^{2\alpha}\left[(\sqrt{b})^{\delta}+\eta(a)(\tilde{b}^2+\sum_{j=1}^{\ell-1}|\tilde{b}_{j,0}|^2)\right]
\eea
for some universal constant $c^*>0$ and $\eta(a)=o_{a\to 0}(1)$. Indeed, assume \eqref{controlenormee}, it implies from \eqref{bd bjk bootstrap}:
\bee
&&\frac{d}{d\tau}\left(e^{2(\ell-\frac\alpha 2+\frac{c^*}{2})\tau}\|\e\|_{L^2_{\rho_Y}}^2\right)+c^*e^{2(\ell-\frac\alpha 2+\frac{c^*}{2})\tau}\|\e\|_{H^1_{\rho_Y}}^2\\
&\lesssim& \eta(a)(\sqrt{b})^{2\alpha+2\eta}e^{2(\ell-\frac\alpha 2+\frac{c^*}{2})\tau}\lesssim \eta(a)e^{c^*\tau}(\sqrt{b})^{2\eta}
\eee
whose time integration using \eqref{bd bjk bootstrap}, \eqref{smallnorminit} yields for $\eta$ universal small enough:
\bee
\|\e(\tau)\|_{L^2_{\rho_Y}}^2&\lesssim& e^{-2(\ell-\frac\alpha 2+\frac{c^*}{2})(\tau-\tau_0)}\|\e(\tau_0)\|_{L^2_{\rho_Y}}^2+\eta(a)e^{-2(\ell-\frac\alpha 2+\frac{c^*}{2})\tau}\int_{\tau_0}^{\tau}e^{c^*\tau'}(\sqrt{b})^{2\eta}d\tau'\\
& \lesssim & \left(\frac{\sqrt{b(\tau)}}{\sqrt{b(\tau_0)}}\right)^{2\alpha+c\delta}\|\e(\tau_0)\|_{L^2_{\rho_Y}}^2+\eta(a)e^{-2(\ell-\frac\alpha 2+\frac{c^*}{2})\tau}e^{c^*\tau}(\sqrt{b(\tau)})^{2\eta}\\
&\lesssim & \eta(a)(\sqrt{b(\tau)})^{2\alpha+2\eta}
\eee
and \eqref{controlponctuelenergy} is proved. We now turn to the proof of \eqref{controlenormee}.\\

\noindent{\bf step 1} Energy identity. We compute from \eqref{eqe}:
\be
\label{cnommemep}
\frac12 \frac{d}{d\tau}\|\e\|^2_{L^2_{\rho_Y}}=(\pa_\tau\e,\e)_{L^2_{\rho_Y}}=(-\mathcal L_b\e-\Psi+L(V)+\NL(V),\e)_{L^2_{\rho_Y}}.
\ee
The linear term is coercive from the spectral gap estimate \eqref{spectralgaptotal} and the choice of orthogonality conditions \eqref{orthoe}: \be \lab{spectralgap energy estimate}
(\mathcal L_b\e,\e)_{L^2_{\rho_Y}}\geq \l_{\ell,0}\|\e\|_{L^2_{\rho_Y}}^2+c^*\|\e\|_{H^1_{\rho_Y}}^2\ee for some $c^*>0$. We now estimate all remaining terms in \eqref{cnommemep}.\\

\noindent{\bf step 2} $\Psi$ terms. We claim:
\be
\label{energyestimatepsiterm}
|(\e,\Psi)_{L^2_{\rho_Y}}|\lesssim \|\e\|_{H^1_{\rho_Y}}(\sqrt{b})^\alpha\left[b^\delta+\eta(a)(|\tilde{b}|+\sum_{j=1}^{\ell-1}|\bt_{jk}|)\right]+ \left[\eta(a)+b^{c\eta}\right]\|\e\|^2_{H^1_{\rho_Y}}.
\ee

We observe from \eqref{estparameters},  \eqref{bd bjk bootstrap} the rough bounds:
\be
\label{roughestimatesmodulation}
|\pa_\tau b_{j,k}|+|b_{j,k}|+(\sqrt{b})^\alpha|a_\tau|+(\sqrt{b})^\alpha(|\tilde{b}_\tau|+|\tilde{b}|)\lesssim (\sqrt{b})^\alpha, \ \ \left|\frac{b_\tau}b\right|\lesssim 1.
\ee
We now estimate the $\Psi_2$ term. Recall \eqref{expressionpsitwo}, then from \eqref{esterroru}, \eqref{orthoe}, \eqref{roughestimatesmodulation}:
$$|(\e,\Psi_2)_{L^2_{\rho_Y}}|=|(\e,\Psit_2)_{L^2_{\rho_Y}}|\lesssim \|\e\|_{H^1_{\rho_Y}}b^\delta(\sqrt{b})^\alpha.$$ 
We now estimate $\Psi_1$ and recall \eqref{formulapsi}, \eqref{psitildeone}. From \eqref{degenpsitildeun}:
$$|(\e,\Psit_1)_{L^2_{\rho_Y}}|\lesssim \|\e\|_{H^1_{\rho_Y}}b^\delta(\sqrt{b})^\alpha$$ and from \eqref{einoenoneoeon}, \eqref{orthoe}:
$$|(\e,\Lambda Q_b)_{L^2_{\rho_Y}}|=|(\e,(\sqrt{b})^\alpha \tilde{\phi}_{0,b})_{L^2_{\rho_Y}}|\lesssim b^\delta(\sqrt{b})^\alpha\|\e\|_{H^1_{\rho_Y}}.$$ Finally, from \eqref{estimportante}, \eqref{estparameters}:
\bee
(\sqrt{b})^\alpha(|B|+|\pa_\tau a|)&\lesssim& (\sqrt{b})^\alpha\left[b^\delta+ |\tilde{b}|+\sum_{j=1}^{\ell-1}|\bt_{j,0}|\right]+  \left[a+(\sqrt b)^{\tilde \eta}\right]\|\e\|_{L^2_{\rho_Y}}
\eee
and from \eqref{controlerreur}, \eqref{einoenoneoeon}, \eqref{orthoe}:
\bee
|(\e,\Lambda \Phi_{a,b})_{L^2_{\rho_Y}}|&\lesssim &\|\e\|_{L^2_{\rho_Y}}\left[\|\Lambda \Phi_{a,b}-\Lambda Q_b\|_{L^2_{\rho_Y}}+\|\Lambda Q_b-(\sqrt{b})^\alpha\psi_{0,0}\|_{L^2_{\rho_Y}} \right]\\
&\lesssim & (\sqrt{b})^\alpha(a+b^\delta)\|\e\|_{H^1_{\rho_Y}}.
\eee
We conclude:
$$|(\e,\Psi_1)|_{L^2_{\rho_Y}}\lesssim \|\e\|_{H^1_{\rho_Y}}(\sqrt{b})^\alpha\left[b^\delta+|a|(|\tilde{b}|+\sum_{j=1}^{\ell-1}|\bt_{j,0}|)\right]+ \left[\eta(a)+b^{c\eta}\right]\|\e\|^2_{H^1_{\rho_Y}}$$
and \eqref{energyestimatepsiterm} is proved.\\

\noindent{\bf step 3} $L(V)$ term. We claim:
\be
\label{estlvenergy}
(L(V),\e)_{L^2_{\rho_Y}}\lesssim b^\delta(\sqrt{b})^\alpha\|\e\|_{H^1_{\rho_Y}}+a\|\e\|_{H^1_{\rho_Y}}^2.
\ee
Indeed, first compute:
\bee
(L(V),\e)_{L^2_{\rho_Y}}=p\left((\Phi_{a,b}^{p-1}-Q_b^{p-1})\e,\e\right)_{L^2_{\rho_Y}}+p\left((\Phi_{a,b}^{p-1}-Q_b^{p-1})\psi,\e\right)_{L^2_{\rho_Y}}.
\eee
For the linear term, we recall \eqref{enoneoneon} which implies for $r\leq \sqrt{b}$ using $\mu\ge \frac 12$:
$$\left|\frac{d}{d\mu}\Phi_{a,b}^{p-1}\right|\lesssim \frac{(\mu\sqrt{b})^\alpha}{(r+\mu\sqrt{b})^\alpha}\frac{1}{\mu(r+\mu\sqrt{b})^2}\lesssim \frac 1b.$$ Hence $$|\Phi_{a,b}^{p-1}-Q_b^{p-1}|\lesssim \frac{1+|\mu|}{b}\ \ \mbox{for}\  \ r\le \sqrt{b}$$  and 
we estimate using \eqref{poitwisepsijk}, \eqref{bd bjk bootstrap}, \eqref{hardyapoids}:
\bee
&&\int_{r\leq \sqrt{b}}|\Phi_{a,b}^{p-1}-Q_b^{p-1}||\psi||\e|\rho_YdY \lesssim\sum|b_{j,k}|\int_{r\leq \sqrt{b}}\frac{|\e|\la z\ra^c}{br^{\gamma}}\rho_YdY\\
& \lesssim & (\sqrt{b})^\alpha\|\frac{\e}{r}\|_{L^2_{\rho_Y}}\left(\int_{r\leq \sqrt{b}}\frac{r^4}{b^2r^{2\gamma+2}}r^{d-1}dr\right)^{\frac 12}\lesssim b^\delta(\sqrt{b})^\alpha\|\e\|_{H^1_{\rho_Y}}.
\eee
where we used \eqref{estfumdamental} in the last step. 
For $r\geq \sqrt{b}$, we use \eqref{enoneoneon} again which implies:
$$\left|\frac{d}{d\mu}\Phi_{a,b}^{p-1}\right|\lesssim \frac{(\mu\sqrt{b})^\alpha}{(r+\mu\sqrt{b})^\alpha}\frac{1}{\mu(r+\mu\sqrt{b})^2}\lesssim \frac{(\sqrt{b})^\alpha \mu^{\alpha-1}}{r^{\alpha+2}}.$$ Hence $$|\Phi_{a,b}^{p-1}-Q_b^{p-1}|\lesssim \frac{(\mu\sqrt{b})^\alpha}{r^{\alpha+2}}\ \ \mbox{for}\  \ r\ge \sqrt{b}$$  and 
we estimate using \eqref{poitwisepsijk}, \eqref{bd bjk bootstrap}, \eqref{hardyapoids}:
\bee
&&\int_{r\geq \sqrt{b}}|\Phi_{a,b}^{p-1}-Q_b^{p-1}||\psi||\e|\rho_YdY \lesssim\sum|b_{j,k}|\int_{r\geq \sqrt{b}}\frac{(\sqrt{b})^\alpha|\e|\la z\ra^c\la r\ra^c}{r^{2+\alpha+\gamma}}\rho_YdY\\
& \lesssim & (\sqrt{b})^\alpha\|\frac{\e}{r}\|_{L^2_{\rho_Y}}\left(\int_{r\geq \sqrt{b}}\frac{(\sqrt{b})^{2\alpha} \la r\ra^c}{r^{2+2\alpha+2\gamma}}\rho_rr^{d-1}dr\right)^{\frac 12}\\
&\lesssim & (\sqrt{b})^\alpha\|\e\|_{H^1_{\rho_Y}}\left(\int_{r\geq \sqrt{b}}\left(\frac{\sqrt{b}}{r}\right)^{2\alpha-\delta}\frac{(\sqrt{b})^\delta \la r\ra^c}{r^{2+2\gamma+\delta}}\rho_rr^{d-1}dr\right)^{\frac 12}\lesssim b^\delta(\sqrt{b})^\alpha\|\e\|_{H^1_{\rho_Y}}
\eee
where we used \eqref{estfumdamental} again in the last step.
For the quadratic term,
we recall \eqref{deriveeqmu}, \eqref{positivitylamdbaq} which imply  $$\frac{d}{d\mu}\Phi_{a,b}^{p-1}<0$$ from which $\Phi_{a,b}^{p-1}-Q_b^{p-1}\leq 0$ for $\mu\geq 1$. From \eqref{limiteuig}, \eqref{controlea}, $\nu(z)=1+aP_{2\ell}(z)>1$ for $|z|\geq z^*$ universal large enough,
and we thus conclude using the bound \eqref{esterrorpotential} and \eqref{hardyapoids}:
\bee
\left((\Phi_{a,b}^{p-1}-Q_b^{p-1})\e,\e\right)_{L^2_{\rho_Y}}&\leq&  \int_{|z|\leq z^*}|\Phi_{a,b}^{p-1}-Q_b^{p-1}|\e^2\rho_Y dY\lesssim \int_{|z|\leq z^*}\frac{|aP_{2\ell}(z)|}{r^2}\e^2\rho_Y dY\\
&\lesssim & a\|\e\|_{H^1_{\rho_Y}}^2.
\eee
 This concludes the proof of \eqref{estlvenergy}.\\
 
 \noindent{\bf step 4} Nonlinear term. We treat the nonlinear carefully and claim for some $\delta$ universal:
 \be
 \label{estnonlienar}
 |(\NL(V),\e)_{L^2_{\rho_Y}}|\leq \eta(a) \|\e\|_{H^1_{\rho_Y}}^2+(\sqrt{b})^{2\alpha+\delta}.
 \ee
 \noindent{\em Estimate for $|z|>M\sqrt{|\log b|}$}: We estimate from \eqref{bd V Linftyphi bootstrap}, using the Gaussian weight and \eqref{bd psi bootstrap}:
  \bee
  \int_{|z|\geq M\sqrt{|\log b|}} |\NL(V)|\e\rho_YdY&\lesssim& \int_{|z|\geq M\sqrt{|\log b|}} (|V|^{p}+\Phi_{a,b}^{p-2}|V|^2)(|V|+|\psi|)\rho_YdY\\
 & \lesssim & \frac{(\sqrt{b})^{cM}}{(\sqrt{b})^c}\leq (\sqrt{b})^{2\alpha+\delta}
 \eee
 for $M$ universal large enough.\\
 
 \noindent{\em Estimate for $|z|<M\sqrt{|\log b|}$}. By homogeneity:
 \bee
  \int_{|z|\leq M\sqrt{|\log b|}} |\NL(V)|\e\rho_YdY&\lesssim&\int_{|z|\leq M\sqrt{|\log b|}}\left( |V|^p+\Phi_{a,b}^{p-2}|V|^2\right)|\e|\rho_YdY\\
 &\lesssim& \int_{|z|\leq M\sqrt{|\log b|}}\left( |\e|^p+|\psi|^p+ \Phi_{a,b}^{p-2}|\e|^2+ \Phi_{a,b}^{p-2}|\psi|^2\right)|\e|\rho_Y dY
 \eee
 Near the origin, we estimate using \eqref{hardyapoids}, \eqref{bd psi bootstrap} and \eqref{bd V Linftyphi bootstrap}:
\bee
&& \int_{r\le 1, |z|\leq M\sqrt{|\log b|}}(|\e|^p+\Phi_{a,b}^{p-2}|\e|^2)|\e|\rho_YdY\\
&\lesssim& \int_{r\leq 1}\left[\|r^{\frac{2}{p-1}}V\|_{L^\infty(r\le 1)}+\|r^{\frac{2}{p-1}}\psi\|_{L^\infty(r\le 1,|z|\leq M\sqrt{|\log b|})}\right]\frac{\e^2}{r^2}\rho_YdY\\
 & \lesssim & ((\sqrt b)^{\tilde{\eta}}+\sqrt{b}^{\eta}) \|\e\|_{H^1_{\rho_Y}}^2
  \eee
  and similarly:
 \bee
 &&\int_{r\le 1,|z|\leq M\sqrt{|\log b|}}(|\psi|^p+\Phi_{a,b}^{p-2}|\psi|^2)|\e|\rho_YdY\\
 & \lesssim& \|r^{\frac{2}{p-1}+\delta}\psi\|_{L^\infty(r\le 1,|z|\leq M\sqrt{|\log b|})}\int_{r\le 1,|z|\leq M\sqrt{|\log b|}}|\psi| |\e|\frac{dY}{r^{2+\delta}}\\
 &\lesssim &\sqrt{b}^{\eta+c\delta}\left(\int \frac{\psi^2}{r^{2+2\delta}}\rho_YdY\right)^{\frac 12}\left(\int \frac{\e^2}{r^2}\rho_YdY\right)^{\frac 12} \lesssim  \sqrt{b}^{\delta}\|\e\|_{H^1_{\rho_Y}}^2+(\sqrt{b})^{2\alpha+\delta}.
 \eee

 Away from the origin, we use the Gaussian weight in $r$ and \eqref{bd psi bootstrap} to estimate:
 \bee
 &&\int_{r\ge 1,|z|\leq K\sqrt{|\log b|}}(|\psi|^p+\Phi_{a,b}^{p-2}|\psi|^2)|\e|\rho_YdY\lesssim (\sqrt{b})^{2\alpha}\int_{r\ge 1,|z|\leq K\sqrt{|\log b|}}\la r \ra^c \la z \ra^c |\e|\rho_YdY\\
 & \lesssim &\sqrt{b}^{\delta}\|\e\|_{H^1_{\rho_Y}}^2+(\sqrt{b})^{2\alpha+\delta}
 \eee
 
 This concludes the proof of \eqref{estnonlienar}.\\
 
\noindent{\bf step 5} Conclusion. The collection of estimates \fref{spectralgap energy estimate}, \fref{energyestimatepsiterm}, \fref{estlvenergy} and \fref{estnonlienar}, injected in \fref{cnommemep}, yields \fref{controlenormee}.

\end{proof}

%%%%%%%%%%%%%%%%%%%%%%%%%%%%%%%%%%%%%%%%%
%%%%%%%%%%%%%%%%%%%%%%%%%%%%%%%%%%%%%

\section{$L^{\infty}$ bound through energy estimates}
\label{linftybound}
%%%%%%%%%%%%%%%%%%%%%%%%%%%%%%%%%%%%%
%%%%%%%%%%%%%%%%%%%%%%%%%%%%%%%%%%%%%

The energy estimate \eqref{controlenormee} easily closes the control of  the modulation equations of Lemma \ref{lemmaeqmodulation} {\it provided the $L^\infty$ control of $V=\psi+\e$}. It is a classical difficulty in the study of singularity formation that the description of the solution near the singularity involves growing in space profiles like $\psi$ which are unbounded in $L^{\infty}$, and exponentially localized norms which are too weak to control the nonlinear term both at the origin and infinity in space. In the setting of the radially symmetric type II blow up, the $L^{\infty}$ bound for $r$ large is obvious and relies on a scaling argument, section \ref{faraway}, see \cite{MRR}, \cite{CRS} for related arguments. At the origin, the unconditional $L^2_{\rho_Y}$ control provides an outer estimate on the sphere $r= 1$ which can easily be propagated inside using upper and lower solutions and the maximun principle, see for example \cite{Mizo,BSeki}. We propose a more energetic proof based on $W^{1,q}$ estimates which is well suited for the cylindrical geometry and handles both the difficulties of type II and the ode type I blow up as in \cite{BK,MZduke}. This provides a pure energy method for the control of the non linear flow.

%%%%%%%%%%%%%%%%%%%%%%%%%%%%%%%%%%%%%

\subsection{Definition of $\N$ and weighted Sobolev bound}

%%%%%%%%%%%%%%%%%%%%%%%%%%%%%%%%%%%%%

We prepare the analysis for the $L^\infty$ bound by introducing two new decompositions of the flow. We let $$0<r^*,\nu\ll 1, \ \ A,q\gg 1$$  universal constants independent of $a,b$ to be chosen later.\\

\noindent{\it New decomposition away from the origin}. We extract from the decomposition $V=\e+\psi$ the leading order term and consider 
\be
\label{dezeta}
V=v+\zeta, \ \ \zeta=b_{\ell,0}(\psi_{\ell,0}-\psi_{0,0})(r).
\ee 
We estimate from \eqref{calculeignefunction}, \fref{id expansion Ti}, \fref{bd pointwise tildephi}, \fref{bd bjk bootstrap} for $r\leq A$ using the {\it essential degeneracy} $c_{i,0}=1$, $\alpha>2$ and $g\leq 2$:
\bea
\label{estzeta}
&&\nonumber |\zeta(r)|=|b_{\ell,0}|\left| \sum_{j=1}^{\ell} (\sqrt b)^{2j-\gamma}T_j\left(\frac{r}{\sqrt{b}}\right)+\tilde{\phi}_{\ell}(r)-\tilde{\phi}_0(r) \right|\\
\non &\lesssim & (\sqrt{b})^{\alpha} \left( \sum_{j=1}^{\ell} (\sqrt b)^{2j-\gamma}\left(1+\frac{r}{\sqrt{b}}\right)^{2j-\gamma}+(\sqrt{b})^g\frac{(1+r)^{2\ell+4}}{(\sqrt b+r)^{\gamma}}\right) \\
& \lesssim & (\sqrt{b})^\alpha \left|\begin{array}{lll} (\sqrt{b})^{g-\gamma}\ \ \mbox{for} \ \ r\leq \sqrt{b}\\ r^{-\gamma}(r^2+(\sqrt{b})^g)\ \ \mbox{for}\ \ \sqrt{b}\leq r\leq 1\\ r^{2\ell+2}\ \ \mbox{for}\ \ r\geq 1\end{array}\right.  \lesssim (\sqrt{b})^{g-\frac{2}{p-1}}(1+r^{2\ell+2})
\eea
which implies in particular
\be
\label{estphibrutale}
\|r^{\frac{2}{p-1}}\zeta\|_{L^{\infty}(r\leq 1)}\lesssim  (\sqrt{b})^g.
\ee

We moreover compute from \fref{renormalizedbvariables}
\bee
&&\pa_\tau\zeta+\L_b\zeta=(\pa_\tau b_{\ell,0}+\l_{\ell,0}b_{\ell,0})\psi_{\ell,0}-(\pa_\tau b_{\ell,0}+\l_{0,0}b_{\ell,0})\psi_{0,0}+  b_{\ell,0}\frac{b_\tau}{b} b\pa_b(\psi_{\ell,0}-\psi_{0,0})\\
&=&(\pa_\tau b_{\ell,0}+\l_{\ell,0}b_{\ell,0})(\psi_{\ell,0}-\psi_{0,0})+(\l_{\ell,0}-\l_{0,0})b_{\ell,0}\psi_{0,0}+ b_{\ell,0}\frac{b_\tau}{b} b\pa_b(\psi_{\ell,0}-\psi_{0,0})\\
& = & (\pa_\tau b_{\ell,0}+\l_{\ell,0}b_{\ell,0})(\psi_{\ell,0}-\psi_{0,0})-\frac{\ell+\tilde{\lambda}_{\ell,0}-\tilde{\lambda}_{0,0}}{\alpha}\frac{(\sqrt{b})^\alpha}{1+\tilde{b}} \psi_{0,0}+ b_{\ell,0}\frac{b_\tau}{b} b\pa_b(\psi_{\ell,0}-\psi_{0,0})\\
& = &  (\pa_\tau b_{\ell,0}+\l_{\ell,0}b_{\ell,0})(\psi_{\ell,0}-\psi_{0,0})+\frac{\ell}{\alpha}(\sqrt{b})^\alpha\frac{\tilde{b}}{1+\tilde{b}} \psi_{0,0}- \frac{\ell}{\alpha}(\sqrt{b})^\alpha\psi_{0,0}\\
&+&   b_{\ell,0}\frac{b_\tau}{b} b\pa_b(\psi_{\ell,0}-\psi_{0,0})-\frac{\tilde{\lambda}_{\ell,0}-\tilde{\lambda}_{0,0}}{\alpha}\frac{(\sqrt{b})^\alpha}{1+\tilde{b}} \psi_{0,0}\\
& = & (\pa_\tau b_{\ell,0}+\l_{\ell,0}b_{\ell,0})(\psi_{\ell,0}-\psi_{0,0})-\frac{(\sqrt{b})^\alpha}{\alpha(1+\tilde{b})}\left[-\ell\tilde{b}+\tilde{\lambda}_{\ell,0}-\tilde{\lambda}_{0,0}\right]\psi_{0,0}\\
&+&  b_{\ell,0}\frac{b_\tau}{b} b\pa_b(\psi_{\ell,0}-\psi_{0,0})-\frac{\ell}{\alpha}(\sqrt{b})^\alpha\left[\psi_{0,0}-\frac{1}{(\sqrt{b})^\gamma}\Lambda Q\left(\frac{r}{\sqrt{b}}\right)\right]-\frac{\ell}{\alpha}\Lambda_r Q_b(r).
\eee
This yields using Lemma \ref{lemmacancellationa} the $v$ equation:
\be
\label{eqebis}
\pa_\tau v+\L_a v =F,\ \ F=-\Psi_3+L(\zeta)+\NL(V), \ \ \L_a=-\Delta+\frac 12\Lambda -p\Phi_{a,b}^{p-1},
\ee 
with
\be
\label{defpsithree}
\Psi_3=\Psit_1+\Psit_3, \ \ L(\zeta)=p(\Phi_{a,b}^{p-1}-Q_b^{p-1})\zeta
\ee
and
\bea
\label{defpsitthree}
\Psit_3&=&\left[\frac12 \left(-\frac{b_\tau}{b}+1\right)-\frac{\ell}{\alpha}-\frac1{2\alpha}\frac{\pa_\tau\nu}{1+\nu}\right]\Lambda_r\Phi_{a,b} \\
\nonumber &+ & (\pa_\tau b_{\ell,0}+\l_{\ell,0}b_{\ell,0})(\psi_{\ell,0}-\psi_{0,0})-\frac{(\sqrt{b})^\alpha}{\alpha(1+\tilde{b})}\left[-\ell\tilde{b}+(\tilde{\lambda}_{\ell,0}-\tilde{\lambda}_{0,0})\right]\psi_{0,0}\\
\nonumber &-& \frac{\ell}{\alpha}(\sqrt{b})^\alpha\left[\psi_{0,0}-\frac{1}{(\sqrt{b})^\gamma}\Lambda Q\left(\frac{r}{\sqrt{b}}\right)\right]+   b_{\ell,0}\frac{b_\tau}{b} b\pa_b(\psi_{\ell,0}-\psi_{0,0})
\eea

\noindent{\it Change of functions at the origin}.  For the derivation of $L^\infty$ bounds near the origin $0<r\lesssim r^*\ll1$, it is more convenient to change variables and define: \be
\label{defww}
w=\frac{v}{T},\ \ T=\Lambda \Phi_{a,b}.
\ee 

\noindent{\it Definition of $\mathcal N$}. We now consider the norms
\bee
&&\| v\|^{2}_{{\rm ext loc}}= \sum_{0\leq i+j\leq 3}  \int_{r\ge \frac{r^*}{2}} \frac{|\pa_r^i(\la z\ra \pa_z)^jv|^{2}}{D^{2\alpha(1+\nu)}\la z\ra}\rho_rdY\\
&&\| V\|^{2q+2}_{{\rm ext global},q}=\int_{\sqrt{r^2+D^2}\geq A} \frac{V^{2q+2}}{\la z\ra}dY\\
&& \| w\|^{2q+2}_{{\rm int,q}}= \int_{r\leq r^*} \frac{w^{2q+2}}{\la z\ra (1+D^{2Kq})}dY
\eee
and the quantity:
\bea
\label{defnorme}
&& \mathcal N=\| v\|_{{\rm ext loc}}+ \| V\|_{{\rm ext global},q}+\|\la z\ra \pa_z V\|_{{\rm ext global},q}\\
\nonumber &+&\left(\| V\|_{{\rm ext global},q}+\|\la z\ra \pa_z V\|_{{\rm ext global},q}\right)^{1-\frac{d}{2q+2}}\|\pa_rV\|_{{\rm ext global},q}^{\frac{d}{2q+2}}\\
\nonumber& + & \| w\|_{{\rm int,q}}+\| \la z\ra w\|_{{\rm int,q}}+ \left(\| w\|_{{\rm int,q}}+\| \la z\ra w\|_{{\rm int,q}}\right)^{1-\frac{d}{2q+2}}\|\pa_rw\|_{{\rm int,q}}^{\frac{d}{2q+2}}
\eea
We claim the weighted Sobolev embedding:

\begin{lemma}[ Weighted Sobolev embedding]
\label{weightn}
Recall \eqref{defphi}, then:
\be
\label{boundlinfty}
    \|\phi V\|_{L^\infty}+\|V\|_{L^\infty(D\geq 2A)}+\left\|\frac{w}{D^{\alpha \nu}}\right\|_{L^\infty(\frac{r^*}{2}\leq r \leq 2A)}+\|w\|_{L^\infty(r\leq 2A,D\leq 2A)}\lesssim_{r^*,A,q} \N+(\sqrt{b})^g.
\ee
and
\be
\label{estpotiwiselinfty}
\left\|\pa_r^i(\la z\ra \pa_z)^j\left(\frac{v}{D^{\alpha(1+\nu)}}\right)\right\|_{L^\infty(\frac{r^*}{2}\leq r\leq 2A)}\lesssim_{r^*,A} \| v\|_{{\rm ext loc}}, \ \ 0\leq i+j\leq 1.
\ee
\end{lemma}

\begin{proof}[Proof of Lemma \ref{weightn}] Let a smooth cut off function $$\chi(x)=\left|\begin{array}{ll} 1\ \ \mbox{for}\ \ |x|\geq 2,\\ 0\ \ \mbox{for}\ \ |x|\leq 1\end{array}\right..$$
\noindent\underline{\it $r\geq 2A$ or $D\geq 2A$}. Note that $|\pa_r(\chi (r/A))|=A^{-1}|\pa_r\chi (r/A)|\lesssim1 $ since $A\geq 1$. We apply \eqref{estitmiatesobolevpoids} to $\chi (r/A)V$ and conclude for $q$ large enough:
\bea
\non &&\|V\|_{L^\infty(r\geq 2A)}^{2q+2}\lesssim \|\chi\left(\frac{r}{A} \right)V\|_{L^\infty}^{2q+2} \\
\non & \lesssim_q & \left(\int \frac{|\chi\left( \frac{r}{A}\right)V|^{2q+2}+|\la z\ra \pa_z\left(\chi\left( \frac{r}{A}\right)V\right)|^{2q+2}}{\la z\ra}dY\right)^{1-\frac{d}{2q+2}}\left(\int \frac{|\pa_r\left(\chi\left( \frac{r}{A}\right)V\right)|^{2q+2}}{\la z\ra}dY\right)^{\frac{d}{2q+2}}\\
\non & \lesssim_q & \left(\int_{r\geq A} \frac{|V|^{2q+2}+|\la z\ra \pa_z\left(V\right)|^{2q+2}}{\la z\ra}dY\right)^{1-\frac{d}{2q+2}}\left(\int_{r\geq A } \frac{|\pa_rV|^{2q+2}+|V|^{2q+2}}{\la z\ra}dY\right)^{\frac{d}{2q+2}}\\
\non &\lesssim_{q} &  \int_{r\ge A} \frac{|V|^{2q+2}+|\la z\ra \pa_zV|^{2q+2}}{\la z\ra}dY\\
\non &+&\left(\int_{r\ge A} \frac{|V|^{2q+2}+|\la z\ra \pa_zv|^{2q+2}}{\la z\ra}dY\right)^{1-\frac{d}{2q+2}}\left(\int_{r\geq A } \frac{|\pa_rV|^{2q+2}}{\la z\ra}dY\right)^{\frac{d}{2q+2}}\\
\non & \lesssim_q & \left[\| V\|_{{\rm ext global},q}+\|\la z\ra \pa_z V\|_{{\rm ext global},q}\right]^{2q+2}\\
\lab{bd V Linfty rgeq2A} & + &\left[ \left(\| V\|_{{\rm ext global},q}+\|\la z\ra \pa_z V\|_{{\rm ext global},q}\right)^{1-\frac{d}{2q+2}}\|\pa_rV\|_{{\rm ext global},q}^{\frac{d}{2q+2}}\right]^{2q+2}\lesssim_q \N^{2q+2}
\eea
since if $r\geq A$ one has $\sqrt{r^2+D^2}\geq A$. similarly, consider $\chi_A(z)=\chi\left(\frac{D}{A}\right)$, then 
$$
D\sim A \ \ \mbox{implies}\ \  z^{2\ell}\sim\frac 1a\left(\frac{A}{(\sqrt{b})}\right)^\alpha\gg 1
$$
and hence 
\be
\label{loclacsocso}\la z\ra \left|\pa_z\left(\chi \left(\frac{D}{A}\right)\right)\right|= \frac{\la z \ra|\pa_z D|\left|\chi' \left(\frac{D}{A}\right)\right|}{A}\lesssim \frac{|z|a|\pa_z P_{2\ell}|}{1+aP_{2\ell}(z)}\frac{\sqrt{b}(1+aP_{2\ell}(z))^{\frac{1}{\alpha}}}{A}\lesssim 1,
\ee
and hence applying \eqref{estitmiatesobolevpoids} to $\chi_A(z)V$ ensures:
\bea
\non \|V\|_{L^\infty(D\geq 2A)}^{2q+2}&\lesssim_q&\left(\int_{D\ge A} \frac{|V|^{2q+2}+|\la z\ra \pa_zV|^{2q+2}}{\la z\ra}dY\right)^{1-\frac{d}{2q+2}}\left(\int_{D\geq A } \frac{|\pa_rV|^{2q+2}}{\la z\ra}dY\right)^{\frac{d}{2q+2}}\\
\lab{bd V Linfty Dgeq2A} &\lesssim_q & \N^{2q+2}
\eea
since $\sqrt{r^2+A^2}\geq D\geq A$.\\

\noindent\underline{\em $r,D\leq 2A$}. For $\frac{r^*}{2}\le r\leq 2A$, first note that
$$
\la z \ra \frac{|\pa_zD|}{D}=\frac{\la z \ra a |\pa_z P_{2\ell}(z)|}{1+aP_{2\ell}(z)}\lesssim 1
$$
we then apply \eqref{estitmiatesobolevpoidsbis} to $\frac{v}{D^{\alpha(1+\nu)}}$ and conclude that:
\bea
\label{eenoeneoe}
\nonumber&&\left\|\frac{v}{D^{\alpha(1+\nu)}}\right\|_{L^\infty(\frac{r^*}{2}\leq r\leq 2A)}^{2}\lesssim_{r^*,A} \sum_{0\leq i+j\leq 2}\int_{\frac{r^*}{2}\leq r\leq 2A} \frac{|\pa_r^i(\la z\ra \pa_z)^j(\frac{v}{D^{\alpha(1+\nu)}})|^{2}}{\la z\ra}dY\\
 & \lesssim_{r^*,A}& \sum_{0\leq i+j\leq 2}\int_{\frac{r^*}{2}\leq r}  \frac{|\pa_r^i(\la z\ra \pa_z)^jv|^2}{D^{2\alpha(1+\nu)}}\frac{\rho_rdY}{\la z\ra}\lesssim_{r^*,A}\|v\|_{{\rm extloc}}^2
\eea
where we used the lower bound $\rho_r\geq e^{-A^2}$ for $r\leq 2A$. The above estimate \fref{eenoeneoe} implies from \eqref{estzeta} the rough bound: 
\be \lab{bd V Linfty rleq2ADleq2A}
\|V\|_{L^\infty(\frac{r^*}{2}\leq r\leq 2A,D\leq 2A)}\lesssim \N+\|\zeta\|_{L^\infty(\frac{r^*}{2}\leq r\leq 2A,D\leq A)}\lesssim \N +(\sqrt{b})^{\alpha}.
\ee
Near the origin $0<r\leq r^*$, we apply \eqref{estitmiatesobolevpoids} to $(1-\chi(\frac{2r}{r^*}))(1-\chi(\frac{D}{2A}))w$, using \fref{loclacsocso}, and obtain:
\bea
\non &&\|w\|_{L^\infty(r\le \frac{r^*}{2},D\leq 2A)}^{2q+2}\lesssim_{q,r^*}\int_{r\leq r^*,D\leq 4A} \frac{|w|^{2q+2}+|\la z\ra \pa_zw|^{2q+2}}{\la z\ra}dY\\
\non & + & \left(\int_{r\leq r^*,D\leq 4A} \frac{|w|^{2q+2}+|\la z\ra \pa_zw|^{2q+2}}{\la z\ra}dY\right)^{1-\frac{d}{2q+2}}\left(\int_{r\leq r^*,D\leq 4A}\frac{|\pa_rw|^{2q+2}}{\la z\ra}dY\right)^{\frac{d}{2q+2}}\\
\lab{bd w Linfty rleqr*2 Dleq2A} & \lesssim & \N^{2q+2}.
\eea
Observing that the global rough bound $|\Lambda Q(y)|\lesssim |y|^{-2/(p-1)}$ implies that $r^{2/(p-1)}T=r^{2/(p-1)}D^{-2/(p-1)}\Lambda Q(r/D)\lesssim 1$, the above bound implies for $r\leq \frac{r^*}{2},D\leq 2A$:
\be \lab{bd r2p-1v rleqr*Dleq2A}
\| r^{\frac {2}{p-1}}v(r)\|_{L^{\infty}(r\leq \frac{r^*}{2},D\leq 2A)} = \|wr^{\frac{2}{p-1}}T\|_{L^{\infty}(r\leq \frac{r^*}{2},D\leq 2A)} \lesssim \|w\|_{L^{\infty}(r\leq \frac{r^*}{2},D\leq A)}\lesssim \mathcal N .
\ee
Finally, we observe that $T\gtrsim \frac{D^\alpha}{r^\gamma}\ \ \mbox{for}\ \ r\geq D$ and hence from \fref{eenoeneoe} 
\be \lab{bd w/Dalphanu r*2leqrleq2A}
\|\frac{w}{D^{\alpha \nu}}\|_{L^\infty(\frac{r^*}{2}\leq r\leq 2A)}\lesssim_A\left\|\frac{v}{D^{\alpha(1+\nu)}}\right\|_{L^\infty(\frac{r^*}{2}\leq r\leq2A)}\lesssim_A \N .\ee

\noindent\underline{Conclusion}. We conclude from \fref{bd V Linfty rgeq2A}, \fref{bd V Linfty Dgeq2A}, \fref{bd V Linfty rleq2ADleq2A}, \fref{bd r2p-1v rleqr*Dleq2A} and \eqref{estphibrutale}:
\bee
&& \|\phi V\|_{L^\infty} \lesssim \|r^{\frac{2}{p-1}} V\|_{L^\infty(r\leq \frac{r^*}{2},D\leq 2A)}+\| V\|_{L^\infty(D\geq 2A)}+ \| V\|_{L^\infty(r\geq \frac{r^*}{2})} \\
 &\lesssim & \|r^{\frac{2}{p-1}} v\|_{L^\infty(r\leq \frac{r^*}{2},D\leq 2A)}+\|r^{\frac{2}{p-1}} \zeta \|_{L^\infty(r\leq \frac{r^*}{2},D\leq 2A)}+ \| V\|_{L^\infty(r\geq2A)} + \| V\|_{L^\infty(D\geq 2A)}\\
 &+&  \| V\|_{L^\infty(\frac{r^*}{2}\leq r\leq2A,D\leq 2A)} \lesssim  \mathcal N +\sqrt{b}^{\alpha}+\sqrt{b}^{g} \lesssim \mathcal N +\sqrt{b}^{g}
\eee
as $\alpha \geq g$. We infer from \fref{bd w Linfty rleqr*2 Dleq2A} and \fref{bd w/Dalphanu r*2leqrleq2A} that:
$$
\| w \|_{L^{\infty}(r\leq 2A, D\leq 2A)} \leq \| w \|_{L^{\infty}(r\leq \frac{r^*}{2}, D\leq 2A)} +\| \frac{w}{D^{\alpha \nu}} \|_{L^{\infty}( \frac{r^*}{2}\leq r \leq 2A, D\leq 2A)} \lesssim \mathcal N
$$
The two above inequalities, with \fref{bd V Linfty Dgeq2A} and \fref{bd w/Dalphanu r*2leqrleq2A} give \fref{boundlinfty}. The bound \fref{eenoeneoe} can be proven similarly for derivatives of $v$, this is \eqref{estpotiwiselinfty}.

\end{proof}

The rest of this section is devoted to the control of the various compenents of $\matchal N$ in \eqref{defnorme} which each require a separate analysis.

%%%%%%%%%%%%%%%%%%%%%%%%%%%%%%%%%%%%%

\subsection{$L^2$ bound away from the origin}

%%%%%%%%%%%%%%%%%%%%%%%%%%%%%%%%%%%%%

We start with the outer $L^2$ bound $\| v\|_{{\rm ext loc}}$. It is a consequence of the spectral structure of the linearized operator near $\Phi_{a,b}$ in a polynomially weighted space, and from the fact that we know from the bootstrap assumptions that the first modes below $\ell-\alpha/2$ are not excited.

\begin{lemma}[Weighted $H^1_{\rho_r}$ bound outside the origin]
\label{lemmah1poly} 
There holds the bound for all $0< \nu<\nu^*$ with $\nu^*$ depending on $\eta$ but independent on $\tilde{\eta}$:
\be
\label{boundL^2inz}
e^{c_1\nu \tau}\int \frac{v^{2}}{D^{2\alpha(1+\nu)}}\frac{\rho_rdY}{\la z\ra}+\int_{\tau_0}^\tau e^{c_1\nu \tau'}\int\frac{|\nabla v|^2+r^{-2}v^2}{D^{2\alpha(1+\nu)}}\frac{\rho_rdY}{\la z\ra}d\tau' \leq 1
\ee
for some universal constant $c_1>0$.
\end{lemma}

\begin{remark} The key feature of this lemma is the weighted $D$ gain and $\la z\ra$ which are both sharp for the analysis.
\end{remark}

\begin{proof}[Proof of Lemma \ref{lemmah1poly}] \noindent{\bf step 1} General weighted $L^{2}$ energy identity. We compute from \eqref{eqebis} for any function $\chi(\tau,r,z)$:
\bee
&&\frac{1}{2}\frac{d}{d\tau}\int v^{2}\chi \rho_rdY =  \frac{1}{2}\int v^{2}\pa_\tau \chi \rho_rdY+\int v \chi\pa_\tau v \rho_rdY\\
&=& \frac{1}{2}\int v^{2}\pa_\tau \chi \rho_rdY+\int v \left[\Delta v+p\Phi_{a,b}^{p-1}v-\frac 12\Lambda v+F\right]\chi  \rho_rdY.
\eee
We integrate by parts in $r$:
\bee
&&\int v \left[\Delta_r v-\frac 12r\pa_rv\right]\chi\rho_rdY=\int \chi v \pa_r(\rho_rr^{d-1}\pa_r v)drdz\\
& = & - \int \chi (|\pa_r v|^2\rho_rdY+\frac{1}{2}\int v^{2}\left[\Delta_r\chi-\frac 12r\pa_r\chi\right]\rho_rr^{d-1}drdz
\eee
and in $z$:
\bee
&&\int v \left[\pa_z^2v-\frac 12z\pa_z v\right]\chi  \rho_rdY= -\int \chi |\pa_z v|^2\rho_rdY+\frac{1}{2}\int v^{2}\left[\pa_z^2\chi+\frac12\pa_z(z\chi)\right]\rho_rdY
\eee
to derive the algebraic relation:
\be
\label{nkoneoneo}
\frac{1}{2}\frac{d}{d\tau}\int v^{2}\chi \rho_rdY=-\matchal Q_{q,\chi}(v,v)+\int v F\chi\rho_rdY
\ee
where we introduced the quadratic form:
\bea
\label{deqqh}
&&\mathcal Q_{\chi}(h,h)=\int \chi|\nabla h|^2\rho_rdY\\
\nonumber & - & \frac{1}{2}\int h^2\left[2p\Phi_{a,b}^{p-1}\chi+\pa_\tau\chi+\pa_z^2\chi-\frac{2}{p-1}\chi+\frac12\pa_z(z\chi)+\Delta_r\chi-\frac 12r\pa_r\chi\right]\rho_rdY
\eea

\noindent{\bf step 2} $L^2_{\rho_r}$ identity with polynomial weight. Let 
\be
\label{decchiz}
\chi(z)=\frac{1}{\la z\ra^{4\ell(1+\nu)+1}},
\ee 
then for $z$ large:
$$-\pa_z^2\chi-\frac 12\pa_z(z\chi)=\left[2\ell(1+\nu)+O\left(\frac{1}{\la z\ra}\right)\right]\chi.$$
We apply \eqref{deqqh} and conclude:
\bea
\label{expressionformequadra}
\nonumber Q_{\chi}(v,v)&=& \int \chi|\nabla v|^2\rho_rdY+\int\left(\ell(1+\nu)+\frac 1{p-1}- p\Phi_{a,b}^{p-1}\right)\chi v^2\rho_rdY\\
&+& O\left(\int_{|z|\leq z^*}v^2\rho_rdY+\nu^2 \int \chi v^2\rho_rdY\right).
\eea
for some large enough $z^*(\nu)$.
We recall from \eqref{positivitylamdbaq} and Proposition \ref{propinfty} the lower bound in terms of quadratic form on $H^1_{\rho_r}$: $$-\Delta_r+\frac 12\Lambda_r-p\Phi_{a,b}^{p-1}\geq -\Delta_r+\frac 12\Lambda_r-\frac{pc_\infty^{p-1}}{r^2}\geq -\frac{\alpha}{2}$$ which implies from \fref{hardyapoids} that for some universal constant $c>0$:
\bee
Q_{\chi}(v,v)&\geq & \int \left(\ell(1+\nu)-\frac{\alpha}{2} -c\nu^2\right)v^2\chi\rho_rdY+\nu^2\left[\int|\nabla v|^2\rho_r\chi dY+\int \frac{v^2}{r^2}\chi\rho_rdY\right]\\
&+& O\left(\int_{|z|\leq z^*}v^2\rho_rdY\right).
\eee
We estimate the well localized quadratic term using \eqref{decchiz}, \eqref{condiotnazero}, \eqref{controlea}:
\be
\label{cobeoeeeo}
\int_{|z|\leq z^*}v^2\rho_rdY\lesssim_{\nu}\int \left[\e^2+(\psi-\zeta)^2\right]\rho_YdY\lesssim (\sqrt{b})^{2\alpha+2\eta}
\ee
where we used \eqref{controlponctuelenergy} and \eqref{exitcondition} in the last step, and the fact that $\zeta$ is the leading order term of $\psi$ given by \eqref{defpsifinal}.
We have therefore obtained from \fref{nkoneoneo}:
\bea
\label{encoencoene} &&\frac{d}{d\tau}\int \chi v^{2}\rho_rdY+2\nu^2\int|\nabla v|^2\rho_r\chi dY\\
\nonumber &+& \left[(2\ell(1+\nu)-\alpha-c\nu^2)\int \chi  v^{2}+2\nu^2\int \frac{v^{2}}{r^2}\chi dY\right]\lesssim  |\int \chi Fv\rho_rdY|+(\sqrt{b})^{2\alpha+2\eta}.
\eea
\noindent{\bf step 4} Estimate for the forcing term.  We now estimate the $F$ term given by \eqref{eqebis} and claim:
\be \lab{bd forcing L2rhor}
\left|\int \chi F v \rho_rdY \right|\lesssim \sqrt{b}^{2\alpha+2\eta}+\nu^2 \int \chi v^2 \rho_r dY.
\ee
To prove it, for the $\Psi_3$, $L(\zeta)$ terms, we estimate from  \eqref{boundspsioneone}, \eqref{estlinearterm}, H\"older and Young inequality:
\bee
\left|\int \chi(\Psi_3+L(\zeta)v\rho_rdY\right|&\lesssim& \left(\int\frac{|\Psi_3|^{2}+|L(\zeta)|^2}{\la z\ra^{2(2\ell+2\nu)+1}}\rho_rdY\right)^{\frac 12}\left(\int \chi v^{2}\rho_rdY\right)^{\frac 12}\\
&\leq&  (\sqrt{b})^{2\alpha+2\eta}+\nu^2\int \chi v^{2}\rho_rdY.
\eee
For the $\NL(V)$ term, we estimate by homogeneity with \eqref{bd V Linftyphi bootstrap}, \eqref{boundlinfty}:
\bee
&&\int \chi|\NL(V)||v|\rho_rdY\lesssim \int |v|\left[|V|^p+|\Phi_{a,b}|^{p-2}V^2\right]\chi \rho_rdY\\
& \lesssim &  \int |v|\left(1+\frac{1}{r^{\frac{2}{p-1}}}\right)^{p-2}V^2\chi \rho_rdY\lesssim  \int |v|\left(1+\frac{1}{r^2}\right)\phi V^2\chi \rho_rdY
\eee
and split the integral in two parts for $M\gg 1$ large enough using the Gaussian weight in the $dz$ integrability provided by $\chi$. Indeed, using \eqref{estzeta} and \fref{bd V Linftyphi bootstrap}:
\bee
&&\int_{r\leq M|\log b|} |v|\left(1+\frac{1}{r^2}\right)\phi V^2\chi \rho_rdY\lesssim \int_{r\leq K|\log b|} |v|\left(1+\frac{1}{r^2}\right)\phi(r)(v^2+\zeta^2)\chi \rho_rdY\\
& \lesssim & (\|\phi V\|_{L^{\infty}}+\|\phi \zeta\|_{L^{\infty}(r\leq M|\log b|)})\int v^{2}\left(1+\frac{1}{r^2}\right)\chi\rho_rdY\\
& + & \nu^2\int v^{2}\left(1+\frac{1}{r^2}\right)\chi\rho_rdY+c_\nu\int_{r\leq M|\log b|}\zeta^{2}(\phi\zeta)^{2}\left(1+\frac{1}{r^2}\right)\chi\rho_rdY\\
& \leq & C\nu^2\int v^{2}\left(1+\frac{1}{r^2}\right)\chi\rho_rdY+(\sqrt{b})^{2\alpha+2g}
\eee
as $\alpha>g$ and for some $c>0$
\bee
&&\int_{r\geq M|\log b|} |v|\left(1+\frac{1}{r^2}\right)\phi V^2\chi \rho_rdY \lesssim  \int_{r\geq M|\log b|} |\zeta|\left(1+\frac{1}{r^{\frac{2}{p-1}}}\right)^{p-2}V^2\chi \rho_rdY+b^{cM}\|\phi V\|_{L^{\infty}}^{3}\\
& \lesssim & (\sqrt{b})^{cM}\leq (\sqrt{b})^{2\alpha+1}
\eee
provided $M$ has been chosen large enough. The collection of above bounds, injected in \fref{eqebis}, gives \fref{bd forcing L2rhor}.

\noindent{\bf step 5} Conclusion. From \fref{encoencoene} and \fref{bd forcing L2rhor} we infer the pointwise differential inequality:
\bea
\label{nceknoenoenvoenv}
 &&\frac{d}{d\tau}\int \chi v^{2}\rho_rdY+\left[2\ell(1+\nu)-\alpha-C\nu^2\right]\int \chi v^{2}\rho_rdY\\
\nonumber&+& \frac{\nu^2}{2}\left[\int|\nabla v|^2\rho_r\chi dY+\int \frac{v^{2}}{r^2}\chi\rho_rdY\right]\lesssim  (\sqrt{b})^{2\alpha+2\eta}.
\eea
We now compute from \eqref{defgibbg}, \eqref{estimateB}, \fref{estparameters} and \eqref{nceknoenoenvoenv}:
\bee
&&\frac{d}{d\tau}\left\{e^{c\nu \tau}\frac{\int \chi v^{2}\rho_rdY}{(\sqrt{b})^{2\alpha(1+\nu)}}\right\}\\\
&=& \frac{e^{c\nu \tau}}{(\sqrt{b})^{2\alpha(1+\nu)}}\left[\frac{d}{d\tau}\int \chi v^{2}\rho_rdY+c\nu \int \chi v^{2}\rho_rdY-\alpha(1+\nu)\frac{b_\tau}{b}\int \chi v^{2}\rho_rdY\right]\\
& = & \frac{e^{c\nu \tau}}{(\sqrt{b})^{2\alpha(1+\nu)}}\left[\frac{d}{d\tau}\int \chi v^{2}\rho_rdY+[(2\ell-\alpha)(1+\nu )+c\nu+O((\sqrt{b})^{\eta})]\int \chi v^{2}\rho_rdY\right\}\\
& \leq & \frac{e^{c\nu \tau}}{(\sqrt{b})^{2\alpha(1+\nu)}}\left\{(-(\alpha-c)\nu+C\nu^2)\int \chi v^{2}\rho_rdY-\frac{\nu^2}{2}\int\left[|\nabla v|^2+ \frac{v^{2}}{r^2}\right]\chi\rho_rdY\right\}\\
& + & e^{c\nu \tau}\frac{(\sqrt{b})^{2\alpha+2\eta}}{(\sqrt{b})^{2\alpha(1+\nu)}}\leq  -\nu^2\frac{e^{c\nu \tau}}{(\sqrt{b})^{2\alpha(1+\nu)}}\left[\int|\nabla v|^2\rho_r\chi dY+\int \frac{v^{2}}{r^2}\chi\rho_rdY\right]+O(\sqrt{b}^{\eta})
\eee
provided $\nu$ small enough and $c<\alpha$. Integration in time ensures:
$$e^{c\nu \tau}\frac{\int v^{2}\rho_r\chi dY}{(\sqrt{b})^{2\alpha(1+\nu)}}+\int_{\tau_0}^\tau\frac{e^{c\nu \tau'}}{(\sqrt{b})^{2\alpha(1+\nu)}}\left[ \int|\nabla v|^2\rho_r\chi dY+\int \frac{v^{2}}{r^2}\chi\rho_rdY\right]d\tau' \lesssim 1.$$ We now observe for $|z|\geq z^*$: $$\frac{1}{\la z\ra}\frac{1}{\mu^{2\alpha(1+\nu)}}\lesssim \frac{1}{a^C\la z\ra^{4\ell(1+\nu)+1}}= \frac{\chi}{a^C}$$ which together with \eqref{cobeoeeeo} and \fref{controlea} yields \eqref{boundL^2inz} for some $0<c=c_1\ll1$ small enough.
\end{proof}

%%%%%%%%%%%%%%%%%%%%%%%%%%%%%%%%%%%%%

\subsection{Control of derivatives outside the origin} 

%%%%%%%%%%%%%%%%%%%%%%%%%%%%%%%%%%%%%

We now propagate the $L^2$ bound \eqref{boundL^2inz} to higher derivatives. 

\begin{lemma}[Control of derivatives outside the origin]
\label{lemmah1polyderivative} 
For $c_2<c_1$ independent of $\tilde{\eta}$ ($c_1$ is defined in Lemma \ref{lemmah1poly}), there holds the bounds for $0<r^*\ll 1$ for $\nu$ small enough depending on $\eta$ but independent on $\tilde{\eta}$:
\be
\label{boundL^2inzout}
\int_{r\geq \frac{r^*}{2}} \frac{(\pa_r^i(\la z\ra\pa_z)^jv)^{2}}{[(\mu\sqrt{b})^\alpha]^{2(1+\nu)}}\frac{\rho_rdY}{\la z\ra}\leq c(r^*) e^{-c_2\nu \tau}, \ \ 0\leq i+j\leq 3.
\ee
\end{lemma}

\begin{proof} This is a standard parabolic regularity claim. We briefly sketch the proof of the main steps to take care of the $\la z\ra$ weight. We let $$v_1=\pa_rv, \ \ v_2=\pa_z v.$$ We let $\chi$ be given by \eqref{decchiz}.\\

\noindent{\bf step 1} Control of $v_2$. From \eqref{eqebis}:
$$\pa_\tau(zv_2)+\L_a (zv_2)+2\pa_zv_2 =z\pa_z F+p(z\pa_z\Phi_{a,b}^{p-1})v.$$  The same chain of estimates like for the proof of \eqref{encoencoene} leads to:
\bee
&&\frac{d}{d\tau}\int \chi (zv_2)^{2}\rho_rdY+\nu^2\int\left[|\nabla (zv_2)|^2+\frac{(zv_2)^{2}}{r^2}\right]\rho_r\chi dY\\
\nonumber &+& (2\ell(1+\nu)-\alpha-\nu^2)\int (zv_2)^{2}\rho_r\chi dY\\
&\lesssim& |\int \left[z\pa_z F+p(z\pa_z\Phi_{a,b}^{p-1})v\right](z v_2)\chi\rho_rdY|+(z^*)^C\int_{|z|\leq z^*}v_2^2\rho_rdY
\eee
for some large enough $z^*(\nu)$.
We now estimate all terms in the above identity. The crossed term is estimated using the rough bound:
$$|(z\pa_z\Phi_{a,b}^{p-1}|\lesssim \frac{z|\pa_zD|}{D}\frac{1}{D^2}\left|[2Q^{p-1}+(p-1)yQ^{p-2}Q'](\frac{r}{D})\right|\lesssim \frac{1}{D^2}\left(\frac{D}{r}\right)^2\lesssim \frac1{r^2}$$ and hence from H\"older:
\bee
\int \chi |z\pa_z\Phi_{a,b}^{p-1}v||z v_2|\rho_rdY\leq \nu^3\int \frac{(zv_2)^{2}}{r^2}\rho_rdY+c(\nu)\int \frac{v^{2}}{r^2}\rho_rdY.
\eee
The $\Psi_3$ and $L(\zeta)$ terms are estimated from \eqref{boundspsioneone}, \fref{estlinearterm}:
 \bee
&& \int \chi (|z\pa_z \Psi_3|+|z\pa_z L(\zeta)|)|z v_2|\rho_rdY\\
 &\lesssim& \left(\int\frac{|z\pa_z\Psi_3|^{2}+|z\pa_z L(\zeta)|^2}{\la z\ra^{2(2\ell+2\eta)+1}}\chi \rho_rdY\right)^{\frac 12}\left(\int \chi (zv_2)^{2}\rho_rdY\right)^{\frac 12}\\
 &\lesssim & (\sqrt{b})^{2\alpha+2\eta}+\nu^3\int \chi (zv_2)^{2}\rho_rdY.
\eee
For the nonlinear term, we compute using $\pa_z\zeta=0$:
\bee
\pa_z \NL(V)=p\pa_z\Phi_{a,b}((\Phi_{a,b}+V)^{p-1}-\Phi_{a,b}^{p-1}-(p-1)\Phi_{a,b}^{p-2}V)+\pa_z v((\Phi_{a,b}+V)^{p-1}-\Phi_{a,b}^{p-1}).
\eee
We use $$|z\pa_z\Phi_{a,b}|\lesssim \frac{|z\pa_zD|}{D}\frac{1}{D^{\frac{2}{p-1}}}\Lambda Q\left(\frac rD\right)\lesssim \Phi_{a,b}$$ to estimate by homogeneity
\bee
&&\left|z\pa_z \NL(V)\right|\lesssim \Phi_{a,b}^{p-2}V^2+\Phi_{a,b}|V|^{p-1}+|z\pa_zv |(|V|^{p-1}+|V|\Phi_{a,b}^{p-2})\\
& \lesssim & \left(1+\frac{1}{r^2}\right)\left[|\phi V|+|\phi V|^{p-2}\right]|V|+|zv_2|\left(1+\frac 1{r^2}\right)(|\phi V|^{p-1}+|\phi V|)
\eee
We therefore split the integral in $r\leq M|\log b|$, $r\ge M|\log b|$ for some $M\gg 1$ as before and estimate using \fref{bd V Linftyphi bootstrap}, \eqref{estzeta}, \fref{estphibrutale}:
\bee
\int_{r\geq M|\log b|} \chi |z\pa_z\NL(v)||z v_2|\rho_rdY\lesssim \nu^3\int (zv_2)^{2}\left(1+\frac1{r^2}\right)\rho_rdY+(\sqrt{b})^{2\alpha+4}
\eee
because of the exponential weight, and as $p\geq 3$
\bee
&&\int_{r\leq M|\log b|} \chi |z\pa_z\NL(v)||z v_2|\rho_rdY\lesssim  \int_{r\leq M|\log b|} \chi |zv_2| \left(1+\frac{1}{r^2}\right)\left[|\phi V|+|\phi V|^{p-2}\right]|V|\\
&+&  \int_{r\leq M|\log b|} \chi |zv_2|^2\left(1+\frac 1{r^2}\right)(|\phi V|^{p-1}+|\phi V|) \\
&\leq& \nu^3\int (zv_2)^{2}\left(1+\frac1{r^2}\right)\rho_rdY+c_\nu\int \left[|\phi V|+|\phi V|^{p-2}\right]^2|V|^2\left(1+\frac1{r^2}\right)\chi\rho_rdY \\
&+&(\sqrt b)^{\tilde \eta} \int_{r\leq M|\log b|} \chi |zv_2|^2\left(1+\frac 1{r^2}\right)\\
&\leq& 2\nu^3\int (zv_2)^{2}\left(1+\frac1{r^2}\right)\rho_rdY+c_\nu\int \phi ^2(\zeta^4+v^4)\left(1+\frac1{r^2}\right)\chi\rho_rdY \\
&\leq& 2\nu^3\int (zv_2)^{2}\left(1+\frac1{r^2}\right)\rho_rdY+c_\nu\int \left[ |\phi \zeta|^2\zeta^2+(|\phi V|^2+|\phi \zeta|^2)v^2\right]\left(1+\frac1{r^2}\right)\chi\rho_rdY \\
&\leq& 3\nu^3\int (zv_2)^{2}\left(1+\frac1{r^2}\right)\rho_rdY+\sqrt{b}^{2\alpha+2g} +(\sqrt b)^{2\tilde{\eta}}\int v^2\left(1+\frac{1}{r^2}\right)\chi \rho_r dY\\
\eee
The collection of above bounds yields the differential inequality:
\bee
&&\frac{d}{d\tau}\int \chi (zv_2)^{2}\rho_rdY+\frac{\nu^2}{2}\int\left[|\nabla (zv_2)|^2+\frac{(zv_2)^{2}}{r^2}\right]\rho_r\chi dY\\
\nonumber &+& (2\ell(1+\eta)-\alpha-C\nu^2)\int (zv_2)^{2}\rho_r\chi dY\\
& \lesssim & \int_{|z|\leq z^*} v_2^2\rho_r\chi dY+(\sqrt{b})^{2\alpha +2\eta} +\int v^2\left(1+\frac{1}{r^2}\right)\chi \rho_r dY.
\eee
Integrating in time using the space time bound \eqref{boundL^2inz} yields the pointwise bound
$$e^{c_2\nu \tau}\int \frac{(zv_2)^{2}}{D^{2(1+\nu)}}\frac{\rho_rdY}{\la z\ra}+\int_{\tau_0}^\tau e^{c_2\nu \tau'}\int\frac{|\nabla (zv_2)|^2}{D^{2(1+\nu)}}\frac{\rho_rdY}{\la z\ra}d\tau' \leq 1$$ by possibly taking any smaller constant $c_2<c_1$ for $\nu$ and $b$ small enough. Iterating $z\pa_z$ derivatives can be done along the same lines, using the estimates \eqref{boundspsioneone} and \fref{estlinearterm} for the source terms, and the nonlinearity being three times differentiable as $p\geq 3$, and is left to the reader.\\

\noindent{\bf step 2} Control of $v_1$. We now let one $\pa_r$ derivative go through \eqref{eqebis}, and we run the energy identity \eqref{encoencoene} with a mulitplier $\chi$ localized strictly away from the origin. Since all terms without derivatives have been estimated at the previous step,
 the control of $\pa_r$ derivatives and mixed derivatives follows. The elementary details are left to the reader.

\end{proof}
%%%%%%%%%%%%%%%%%%%%%%%%%%%%%%%%%%%%%

\subsection{$L^{2q+2}_{\rho_r}$ bound at the origin}

%%%%%%%%%%%%%%%%%%%%%%%%%%%%%%%%%%%%%

 We now study $L^\infty$ bounds near the origin $0<r\lesssim r^*\ll1$. For this, it is more convenient to use the variable $w$ given by \eqref{defww} and derive suitable $L^{2q+2}$ monotonicity estimates in $w$, which equivalently correspond to suitably weighted estimates for $v$ near the origin. The key point here is that we may derive such bounds for a uniform weight in $z$.

\begin{lemma}[$L^{2q+2}$ bound at the origin]
\label{boundltwoqueue}
There exists $0<c_3< c_2$ such that the following holds. Pick $K$ large enough, then there holds for all $q>q(d,p)$ large enough:
\be
\label{cneineoneonevnio}
e^{c_3\nu q\tau}\int_{r\leq 1} \frac{w^{2q+2}}{\la z\ra (1+D^{2Kq})}dY+\int_{\tau_0}^\tau\int_{r\leq 1} e^{c_3\nu q\tau'}\frac{w^{2q+2}}{r^2\la z\ra (1+D^{2Kq})}dYd\tau'\leq 1.
\ee
\end{lemma}

\begin{proof} Observe 
\be
\label{eqttt}
(\Delta_r+p\Phi_{a,b}^{p-1})T=0
\ee
and hence \eqref{eqebis} becomes:
\be
\label{eqwwfeo}
\pa_t w=\Delta w+2\frac{\nabla T\cdot\nabla w}{T}+V_1 w-\frac 12Y\cdot\nabla w-\frac{1}{p-1}w+\widetilde{F}, \ \ \widetilde{F}=\frac{F}{T}
\ee
with 
\be
\label{devone}
V_1=\frac 1T\left[-\pa_tT+\pa_{zz}T-\frac 12Y\cdot\nabla T\right].
\ee
\noindent{\bf step 1} $L^{2q+2}$ energy identity. 
We compute from \eqref{eqwwfeo} for a given $\chi(t,r,z)$:
\bee
&&\frac{1}{2q+2}\frac{d}{dt}\int  \chi w^{2q+2}dY=\frac{1}{2q+2}\int  \pa_t\chi w^{2q+2}dY\\
&+& \int w^{2q+1} \left\{\Delta w+2\frac{\nabla T\cdot\nabla w}{T}+V_1 w-\frac 12Y\cdot\nabla w-\frac{1}{p-1}w+\widetilde{F}\right\}\chi dY\\
& = & -Q_\chi(w^{q+1},w^{q+1})+\int w^{2q+1}\tilde{F}\chi dY
\eee
with
\bee
&&Q_\chi(h,h)=\frac{2q+1}{(q+1)^2}\int |\nabla h|^2\chi dY+\int h^2\left\{\frac{2}{2q+2}\Delta(\log T)-V_1-\frac{d+1}{2(2q+2)}-\frac{1}{p-1}\right\}\chi dY\\
& -&  \int h^2\left[\frac{1}{2q+2}\Delta \chi-\frac{2}{2q+2}\frac{\nabla T\cdot\nabla \chi}{T}+\frac{1}{2(2q+2)}Y\cdot\nabla\chi+\frac{\pa_t\chi}{2q+2}\right]dY.
\eee
We choose 
\be
\label{defjijcjoc}
\chi(r,z)=\frac{\phi_1(r)}{\la z\ra(1+D^{2Kq})}, \ \ \phi_1(r)=\left\{\begin{array}{ll} 1\ \ \mbox{for}\ \ r\leq 1,\\ 0\ \ \mbox{for}\ \ r\geq 2.\end{array}\right..
\ee

\noindent{\bf step 2} Lower bound on $Q_\chi$. We estimate:
$$
\left|\frac{\pa_t\chi}{\chi}\right|\lesssim \frac{D}{1+D^{2Kq}}\frac{|\pa_tD|}{D}\lesssim \frac{|\pa_t\mu|}{\mu}+\frac{|b_t|}{b}\lesssim 1.
$$ We then estimate in brute force using the definition of $\chi$ and the pointwise bounds \eqref{estvone}, \eqref{estlogt}:
\bee
Q_\chi(h,h)&\gtrsim& \frac{2q+1}{(q+1)^2}\int |\nabla h|^2\chi dY+\int h^2\left\{\frac{2}{2q+2}\Delta_r (\log T)\right\}\chi dY\\
& + & O\left(\int_{r\leq 2}\frac{h^2}{\la z\ra(1+D^{2Kq}) }dY\right).
\eee
We now estimate from \eqref{estpoitnwisepotential}, \eqref{defequgamma}, \eqref{eqttt}, \eqref{positivitylamdbaq}:
$$\Delta_r(\log T)=-p\Phi_{a,b}^{p-1}-\left(\frac{\pa_r T}{T}\right)^2\geq \frac{-\gamma(d-2-\gamma)-\gamma^2}{r^2}=-\frac{\gamma(d-2)}{r^2}.$$
We then use the sharp Hardy inequality 
\be
\label{sharphardy}
\int \chi|\pa_rh|^2r^{d-1}dr\geq \left(\frac{d-2}{2}\right)^2\int \frac{\chi h^2}{r^2}r^{d-1}dr+O\left(\int \frac{|\pa_r\chi|^2}{\chi}h^2r^{d-1}dr\right)
\ee

to lower bound:
\bee 
&&\frac{2q+1}{(q+1)^2}\int |\nabla h|^2\chi dY+\int h^2\left\{\frac{2}{2q+2}\Delta_r(\log T)\right\}\chi dY\\
&\geq&  \left[\frac{2q+1}{(q+1)^2}\left(\frac{d-2}{2}\right)^2-\frac{2(d-2)\gamma}{2q+2}\right]\int\frac{h^2}{r^2}\chi dY+O\left(\int_{r^*\leq r\leq 2}h^2\frac{\rho_r}{\la z\ra(1+D^{2Kq})}dY\right)
\eee
Observe
$$
\frac{2q+1}{(q+1)^2}\left(\frac{d-2}{2}\right)^2-\frac{2(d-2)\gamma}{2q+2}=\frac{d-2}{4(q+1)^2}\left[2q(d-2-2\gamma)+d-2-4\gamma\right]>0
$$
for $q$ large enough from \eqref{defgamma}, and hence for $q>q(d,p)\gg1$ and a universal $c>0$:
\bee
&&\frac{2q+1}{(q+1)^2}\int |\nabla h|^2\chi dY+\int h^2\left\{\frac{2}{2q+2}\Delta_r(\log T)\right\}\chi dY\\
& \geq & \frac c q\int \left(|\nabla h|^2+\frac{h^2}{r^2}\right)\chi dY+O\left(\int_{r^*\leq r\leq 2} h^2\frac{\rho_r}{\la z\ra(1+D^{2Kq})}dY\right).
\eee
and hence the lower bound
\bee
Q_\chi(h,h)&\geq& \frac{c}{q}\int \left(|\nabla h|^2+ \frac{h^2}{r^2}\right)\chi dY+O\left(\int_{ r^*\leq r\leq 2} h^2\frac{\rho_r}{\la z\ra(1+D^{2Kq})}dY\right)
\eee provided $0<r^*(K,q)\ll1$ universal has been chosen small enough.

\noindent{\bf step 2} $L^{2q+2}$ bound. We conclude for $q$ large enough:
\bee
&&\frac{1}{2q+2}\frac{d}{dt}\int  w^{2q+2}\chi dY+\frac{c}{q}\left[\int|\nabla (w^{q+1)}|^2\chi dY+\int\frac{w^{2q+2}}{r^2}\chi dY\right]\\
& \lesssim & \int_{r^*\leq r\leq 2} w^{2q+2}\frac{\rho_r}{\la z\ra(1+D^{2Kq})}dY+\left|\int_{r\leq 2}  w^{2q+1}\widetilde{F}\chi dY\right|
\eee
and we now estimate the rhs.\\

\noindent{\it Quadratic term}. We first observe from \eqref{estpotiwiselinfty} and \fref{boundL^2inzout} the pointwise bound:
\be
\label{linfryvouter}
\left\|\pa_r^i(\la z\ra \pa_z)^j\left(\frac{v}{D^{\alpha(1+\nu)}}\right)\right\|_{L^\infty(\frac{r^*}{2}\leq r\leq 2A)}\lesssim e^{-\frac{c_2}{2}\nu \tau}, \ \ 0\leq i,j\leq 1.
\ee
We then split the integral in two parts. For $\frac{r}{D}\geq 1$, we have $$\Lambda \Phi_{a,b}\gtrsim \frac{c(r^*)D^\alpha}{r^\gamma}$$ and hence using \eqref{linfryvouter}, \eqref{boundL^2inz}:
\bee
&&\int_{r^*\leq r\leq 2, r\geq D} w^{2q+2}\frac{\rho_r}{\la z\ra(1+D^{2Kq})}dY\lesssim \int_{r^*\leq r\leq 2, r\geq D} \left(\frac{v}{D^\alpha}\right)^{2q+2}\frac{dY}{\la z\ra(1+D^{2Kq})}\\
& \lesssim &  \left\|\frac{v}{D^{\alpha(1+\nu)}}\right\|_{L^{\infty}(r^*\leq r\leq 2,D\leq 2 )}^{2q}\int \frac{v^2}{D^{2\alpha(1+\nu)}}\frac{\rho_rdY}{\la z\ra}\lesssim e^{-c_2\nu q\tau}
\eee
with $0<c\ll 1$ independent of $q$. For $\frac rD\leq 1$, we use $\Lambda \Phi_{a,b}\gtrsim \frac{1}{D^{\frac 2{p-1}}}$ and \eqref{linfryvouter} to estimate:
\bee
\int_{r^*\leq r\leq 2, r\leq D} w^{2q+2}\frac{\rho_r}{\la z\ra(1+D^{2Kq})}dY&\lesssim &e^{-c_2\nu q\tau}\int\frac{D^{Cq}}{1+D^{2Kq}}\frac{dz}{\la z\ra} \lesssim e^{-\frac{c_2}{2}\nu q\tau}
\eee
for $K$ large enough independent of $q$.\\

\noindent{\it $\Psi_3,L(\zeta)$ term}. We estimate from H\"older, \eqref{ceibbeokboeoebv} and \eqref{ceibbeokboeoebvL}:
\bee
\left|\int_{r\leq 2}  w^{2q+1}\frac{\Psi_3+L(\zeta)}{T}\chi dY\right|&\leq& \kappa \int \frac{w^{2q+2}}{r^2}\chi dY+c_\kappa \int \left(\frac{r^{2\frac{2q+1}{2q+2}}\Psi_3}{T}\right)^{2q+2}\chi dY\\
&\leq&  \kappa \int \frac{w^{2q+2}}{r^2}\chi dY+b^{c\eta q}
\eee

\noindent{\it Nonlinear term}. We claim for a universal $c>0$:
\be
\label{ceneonoe}
\int_{r\leq 2} |\NL(V)|w^{2q+1}\chi dY\leq \kappa \int \frac{w^{2q+2}}{r^2}\chi dY+b^{c q}.
\ee
Indeed, we estimate by H\"older:
\bee
\int_{r\leq 2} \frac{|\NL(V)|}{T}w^{2q+1}\chi dY\leq \kappa \int \frac{w^{2q+2}}{r^2}\chi dY+c_\kappa \int \left(\frac{r^{\frac{2(2q+1)}{2q+2}}\NL(V)}{T}\right)^{2q+2}\chi dY.
\eee
We now estimate by homogeneity for $r\leq 2$ using \eqref{estphibrutale}, \fref{smallnorminitsobolevboot}:
\bee
\frac{|\NL(V)|}{T}&\lesssim &\frac{1}{T}\left[|v|^p+|\zeta|^p+\Phi_{a,b}^{p-2}(|v|^2+|\zeta|^2)\right]\lesssim \frac{1}{T}\|\phi v\|_{L^{\infty}(r\le 2)}\frac{|v|}{r^2}+\frac{1}{T}\|\phi \zeta\|_{L^\infty(r\le 2)}|\zeta|\\
& \lesssim & \|\phi v\|_{L^{\infty}(r\le 2)}\frac{|w|}{r^2}+(\sqrt b)^g\frac{\zeta}{T} \lesssim  (\sqrt{b})^{\tilde \eta}\frac{|w|}{r^2}+(\sqrt b)^g\frac{\zeta}{T}.
\eee
Moreover, we extract from \eqref{estzeta} the rough bound for $r\leq 2$ $$|\zeta(r)|\lesssim \frac{(\sqrt{b})^\alpha}{r^\gamma}$$  which implies from \eqref{cneoneonve} for $D\leq r\leq 2$:
$$\frac{|\zeta|}{T}\lesssim \frac{(\sqrt{b})^\alpha}{r^\gamma}\frac{D^\gamma+r^\gamma}{D^\alpha}\lesssim \frac{1}{\mu ^\alpha}.$$
On the other hand for $r\leq D$ from \eqref{estzeta}:
$$
\frac{|\zeta|}{T} \lesssim (\sqrt{b})^{g-\frac{2}{p-1}}D^{\frac{2}{p-1}}
$$ and hence the bound for $r\le 2$
\be
\label{cnekoneonvneokv}\frac{|\NL(V)|}{T}\lesssim (\sqrt b)^{\tilde \eta}\frac{|w|}{r^2}+\frac{b^{\frac g2}}{\mu^\alpha}{\bf 1}_{D\leq r\leq 2}+D^{\frac{2}{p-1}}b^{\frac g2-\frac{1}{p-1}}{\bf 1}_{r\leq \min\{2,D\}}
\ee
which implies since $\mu \geq 1/2$, $g\geq 1$ and $p\geq 3$:
$$\int \left(\frac{r^{2\frac{2q+1}{2q+2}}\NL(V)}{T}\right)^{2q+2}\chi dY\lesssim \|\phi v\|^{2q}_{L^{\infty}(r\le 2)}\int \frac{w^{2q+2}}{r^2}\chi dY+b^{c q}$$
for $K$ large enough independent of $q$ and $c$ universal. The collection of above bounds concludes the proof of \eqref{ceneonoe}.\\

\noindent{\bf step 3} Conclusion. We conclude for $r^*(q,K)$ small enough since $0<\nu\ll \eta \ll 1$:
$$\frac{1}{2q+2}\frac{d}{dt}\int  w^{2q+2}\chi dY+\frac{c}{q}\left[\int|\nabla (w^{q+1)}|^2\chi dY+\int\frac{w^{2q+2}}{r^2}\chi dY\right]\lesssim e^{-\frac{c_2}{2}\nu q \tau}$$ which implies provided $r^*(q)$ has been chosen small enough:
\bee
\frac{1}{2q+2}\frac{d}{dt}\int  w^{2q+2}\chi dY+2cq\int  w^{2q+2}\chi dY+\frac{c}{q}\left[\int|\nabla (w^{q+1)}|^2\chi dY+\int\frac{w^{2q+2}}{r^2}\chi dY\right]\lesssim e^{-\frac{c_2}{2}\nu q \tau}
\eee
whose time integration ensures that there exists $0<c_3\leq c_2$ for which:
\bee
\int  w^{2q+2}\chi dY+\int_{\tau_0}^\tau e^{c_3\nu q\sigma }\left[\int|\nabla (w^{q+1)}|^2\chi dY+\int\frac{w^{2q+2}}{r^2}\chi dY\right]d\sigma \lesssim e^{-c_3\nu q \tau}.
\eee
The definition \eqref{defjijcjoc} of $\chi$ now yields \eqref{cneineoneonevnio}.
\end{proof}

%%%%%%%%%%%%%%%%%%%%%%%%%%%%%%%%%%%%%%%

\subsection{$W^{1,2q+2}$ bound}

%%%%%%%%%%%%%%%%%%%%%%%%%%%%%%%%%%%%%%%

We now turn to the weigted control of derivatives in large $L^{2q+2}$ norms. $\pa_r$ derivatives yield singular term at the origin and we claim a lossy bound which for $q$ large is sufficient thanks to \eqref{defnorme} and the underlying weighted Sobolev estimate \eqref{estitmiatesobolevpoids}.

\begin{lemma}[$W_1^{2q+2}$ estimate]
\label{lemmaestwone}
There holds:
\be
\label{cneineoneonevniobis}
\int_{r\leq 1} \frac{(\pa_rw)^{2q+2}}{\la z\ra(1+D^{2Kq})}dY\lesssim  \frac{1}{b^{Cq}}
\ee
and:
\be
\label{cneineoneonevniobisbis}
\int_{r\leq 1} \frac{(\la z\ra \pa_zw)^{2q+2}}{\la z\ra(1+D^{2Kq})}dY\lesssim  e^{-c_4\nu q \tau}
\ee
for some universal constants $C=C(d,p)$ independent of $q$ and $c_4\lesssim c_3$.
\end{lemma}

\begin{proof} Let $$w_1=\pa_r w, \ \ w_2=\pa_z w.$$

\noindent{\bf step 1} $\pa_r$ bound. Taking $\pa_r$ of \eqref{eqwwfeo}:
\bee
\pa_tw_1&=&\left[\Delta w_1+2\frac{\nabla T\cdot\nabla w_1}{T}+V_1w_1-\frac 12Y\cdot\nabla w_1\right]\\
&-& \left[\frac{d-1}{r^2}-2\pa_{rr}\log T\right]w_1+2\pa_{rz}(\log T)w_2+(\pa_rV_1)w-\frac 12 w_1\\
& + & \pa_r \widetilde{F}.
\eee
We let $\chi$ be given by \eqref{defjijcjoc} and hence arguing as above, we have for $q$ large enough:
\bee
&&\frac{1}{2q+2}\frac d{dt}\int w_1^{2q+2}\chi dY\\
&\leq&  -\frac{c}{q}\int\left[|\nabla (w_1^{q+1})|^2+\int\frac{w_1^{2q+2}}{r^2}\right]\chi dY-\int \left[\frac{d-1}{r^2}-2\pa_{rr}\log T\right]w_1^{2q+2}\chi dY\\
& + & C(r^*)\int_{r^*\leq r\leq 2}w_1^{2q+2}\frac{dY}{\la z\ra(1+D^{2Kq})}+ \int \left[2\pa_{rz}(\log T)w_2+(\pa_rV_1)w+\pa_r \widetilde{F}\right]w_1^{2q+1}\chi dY.
\eee
We now observe from \eqref{estpoitnwisepotentialbis}:
$$ \frac{d-1}{r^2}-2\pa_{rr}\log T\geq \frac{d-1-2\gamma}{r^2}>0$$ and hence using \eqref{estvone}, \eqref{estlogt}:
\bee
&&\frac{1}{2q+2}\frac d{dt}\int w_1^{2q+2}\chi dY+\frac{c}{q}\int\left[|\nabla (w_1^{q+1})|^2+\int\frac{w_1^{2q+2}}{r^2}\right]\chi dY\\
& \lesssim & C(r^*)\int_{r^*\leq r\leq 2}w_1^{2q+2}\frac{dY}{\la z\ra(1+D^{2Kq})}+ \int \left(\frac{|w_2|}{r\la z\ra }+\frac{|w|}{r}+|\pa_r \widetilde{F}|\right)|w_1|^{2q+1}\chi dY.
\eee
We split the crossed term:
\bee
&&\int_{r\leq 2} (|w_2|+|w|)|w_1|^{2q+1}\frac{dY}{r\la z\ra(1+D^{2Kq})}=\int_{r\leq 2} r (|w_2|+|w|)|w_1|^{2q+1}\frac{dY}{r^2\la z\ra(1+D^{2Kq})}\\
&\lesssim & r^*\int_{r\leq 2}[w_2^{2q+2}+w^{2q+2}+w_1^{2q+2}]\frac{dY}{r^2\la z\ra(1+D^{2Kq})}\\
&+& c(r^*)\int_{r^*\leq r\leq 2}\left[w_1^{2q+2}+w^{2q+2}+w_2^{2q+2}\right]\frac{dY}{\la z\ra(1+D^{2Kq})}
\eee
and therefore obtain:
\bee
&&\frac{1}{2q+2}\frac d{dt}\int w_1^{2q+2}\chi dY+\frac{c}{q}\int\left[|\nabla (w_1^{q+1})|^2+\int\frac{w_1^{2q+2}}{r^2}\right]\chi dY\\
& \lesssim & r^*\int_{r\leq 2}[w_2^{2q+2}+w^{2q+2}+w_1^{2q+2}]\frac{dY}{r^2\la z\ra(1+D^{2Kq})}+\int|\pa_r \widetilde{F}||w_1|^{2q+1}\chi dY\\
&+& c(r^*)\int_{r^*\leq r\leq 2}\left[w_1^{2q+2}+w^{2q+2}+w_2^{2q+2}\right]\frac{dY}{\la z\ra(1+D^{2Kq})}.
\eee
We now estimate all terms in the above identity.\\

\noindent{\it Exterior term}. From \eqref{linfryvouter}, for $K$ large enough
$$
\int_{r^*\leq r\leq 2}\left[w_1^{2q+2}+w^{2q+2}+w_2^{2q+2}\right]\frac{dY}{\la z\ra(1+D^{2Kq})}.\lesssim e^{-c_2\nu q\tau}\int_{r\leq 2} \frac{1+D^{Cq}}{\la z\ra(1+D^{2Kq})}dY\lesssim e^{-\frac{c_2}{2}\nu q\tau}.
$$
\noindent{\it $\Psi_3,L(\zeta)$ term}. We estimate from H\"older, \eqref{ceibbeokboeoebvbis}, \eqref{ceibbeokboeoebvbisL}:
\bee
\int_{r\leq 2}  |w_1|^{2q+1}\left|\pa_r\left(\frac{\Psi_3}{T}+\frac{L(\zeta)}{T}\right)\right|\chi dY&\leq&\kappa \int \frac{w_1^{2q+2}}{r^2}\chi dY+c_\kappa \int \left[r^{\frac{2(2q+1)}{2q+2}}\pa_r\left(\frac{\Psi_3}{T}\right)\right]^{2q+2}\chi dY\\
&\leq & \kappa \int \frac{w_1^{2q+2}}{r^2}\chi dY+\frac{1}{b^{c q}}
\eee

\noindent{\it Nonlinear term}. The control of the nonlinear term demands a loss due to the singularity at the origin which will however be manageable.
We integrate by parts to estimate using \eqref{linfryvouter}
$$
\left|\int w_1^{2q+1}\pa_r\left(\frac{\NL}{T}\right)\chi dY\right|\lesssim 1+\int_{r\leq 2} \frac{|\NL|}{|T|}\left[|\pa_r(w_1^{2q+1})|+\frac{|w_1|^{2q+1}}r\right]\chi dY.
$$
We now claim from \eqref{boundlinfty} the rough bound: 
\be
\label{cneneneoeno}
\|v\|_{L^\infty(r\leq 2)}\lesssim \frac{\N}{(\sqrt{b})^{\frac{2}{p-1}}}+(\sqrt b)^{\frac 12}.
\ee
Indeed, let $r\le 2$, then for $D\geq A$, since $g\geq 3/2$ and $p\geq 3$:
\be
\label{estlinfty}
\|v\|_{L^\infty(r\leq 2,D\geq A)}\lesssim \|V\|_{L^\infty(r\leq 2,D\geq A)}+\|\zeta\|_{L^\infty(r\leq 2)}\lesssim \N+\sqrt{b}^{g-\frac{2}{p-1}}\lesssim \N+\sqrt{b}^{\frac 12},
\ee and for $D\leq A$ we use $|T|\lesssim \frac{1}{D^{\frac{2}{p-1}}}\lesssim \frac{1}{(\sqrt{b})^{\frac 2{p-1}}}$ to estimate from \fref{boundlinfty}:
$$\|v\|_{L^\infty(r\leq 2,D\leq A)}\lesssim \|v\|_{L^\infty(r^*\leq r\leq 2,D\leq A)}+\frac{1}{(\sqrt{b})^{\frac 2{p-1}}}\|w\|_{L^\infty( r\leq 2,D\leq A)}\lesssim \frac{\N}{(\sqrt{b})^{\frac{2}{p-1}}}+(\sqrt b)^{\frac 12},$$ and \eqref{estlinfty} is proved. This implies the rough bound $$\|v\|_{L^\infty(r\leq 2)}\lesssim \frac{1}{b^{\frac{1}{p-1}}}.$$ We now use $$T\gtrsim \left|\begin{array}{ll}\frac{D^\alpha}{r^\gamma}\ \ \mbox{for}\  \ r\ge D,\\  \frac{1}{D^{\frac{2}{p-1}}}\ \ \mbox{for}\  \ r\le D\end{array}\right.$$ to derive the rough bound for some universal $C>0$: $$\frac{|\NL|}{T}\lesssim \frac{|v|}{T}\left[|v|^{p-1}+\frac{|v|}{D^{\frac{2(p-2)}{p-1}}}\right]\lesssim \frac{1+D^C} {b^C} $$ which yields the lossy bound:
\bee
\int_{r\leq 2} \frac{|\NL|}{|T|}\left[|\pa_r(w_1^{2q+1})|+|\frac{|w_1|^{2q+1}}r\right]\chi dY\lesssim \frac 1{b^C}\int \left[|\pa_r(w_1^{2q+1})|+|\frac{|w_1|^{2q+1}}r\right](1+D^C)\chi dY.
\eee
Now from H\"older:
\bee
\frac{1}{b^C}\int_{r\leq 2}\frac{|w_1|^{2q+1}}r(1+D^C)\chi dY&\lesssim &\frac{1}{b^C}\left(\int_{r\leq 2}\frac{w_1^{2q+2}}{r^2}\chi dY\right)^{\frac{2q+1}{2q+2}}\left(\int_{r\leq 2} \frac{r^{2q}(1+D^{Cq})}{\la z\ra (1+D^{2Kq})} dY\right)^{2q+2}\\
&\leq&  \kappa \int \frac{w_1^{2q+2}}{r^2}\chi dY+\frac{1}{b^{Cq}}
\eee
for some universal constant $C=c(d,p)$ independent of q.
similarly:
\bee
&&\frac{1}{b^C}\int_{r\leq 2}|\pa_r(w_1^{2q+1})|(1+D^C)\chi dY\lesssim \frac{1}{b^C}\left(\int_{r\leq 2} |\pa_r(w_1^{q+1})|^2\chi dY\right)^{\frac 12}\left(\int_{r\leq 2} w_1^{2q}(1+D^C)^2\chi dY\right)^{\frac 12}\\
& \leq & \kappa \int (|\pa_r(w_1^{q+1})|^2+w_1^{2q+2})\chi dY+\frac{1}{b^{Cq}}\int_{r\le 2} \frac{1+D^{Cq}}{1+D^{Kq}}dY\leq  \kappa \int (|\pa_r(w_1^{q+1})|^2+w_1^{2q+2})\chi dY+\frac{1}{b^{Cq}}
\eee
The collection of above bounds yields the pointwise differential inequation:
\bea
\label{firnsovneno4}
&&\frac{1}{2q+2}\frac d{dt}\int w_1^{2q+2}\chi dY+\frac{c}{q}\int\left[|\nabla (w_1^{q+1})|^2+\int\frac{w_1^{2q+2}}{r^2}\right]\chi dY\\
\nonumber & \leq & \frac1{b^{Cq}}+r^*\int_{r\leq 2}\frac{w_2^{2q+2}+w^{2q+2}}{r^2}\frac{dY}{\la z\ra(1+D^{2Kq})}.
\eea

\noindent{\bf step 2} $\pa_z$ bound. Taking $\pa_z$ of \eqref{eqwwfeo}:

\bee
\pa_tw_2&=&\left[\Delta w_2+2\frac{\nabla T\cdot\nabla w_2}{T}+V_1w_2-\frac 12Y\cdot\nabla w_2\right]\\
&+& \left[2\pa_{zz}\log T\right]w_2+2\pa_{rz}(\log T)w_1+(\pa_zV_1)w-\frac 12 w_2+  \pa_z \widetilde{F}.
\eee
We therefore let  $$\tilde{\chi}=\la z\ra^{2q+2}\chi,$$ and obtain arguing as above:
\bee
&&\frac{1}{2q+2}\frac d{dt}\int w_2^{2q+2}\tilde{\chi} dY\\
&\leq&  -\frac{c}{q}\int\left[|\nabla (w_2^{q+1})|^2+\int\frac{w_2^{2q+2}}{r^2}\right]\tilde{\chi} dY+C(r^*)\int_{r^*\leq r\leq 2}(\la z\ra w_2)^{2q+2}\frac{dY}{\la z\ra(1+D^{2Kq})}\\
&+&  \int \left[2\pa_{zz}(\log T)w_2+2\pa_{rz}\log Tw_1+(\pa_zV_1)w+\pa_z \widetilde{F}\right]w_2^{2q+1}\tilde{\chi} dY.
\eee
We estimate all terms in this identity.\\

\noindent{\it Exterior term}. From \eqref{linfryvouter}:
\bee
\int_{r^*\leq r\leq 2}(\la z\ra w_2)^{2q+2}\frac{dY}{\la z\ra(1+D^{2Kq})}\lesssim \|\la z\ra \pa_z w\|_{L^{\infty}(r^*\leq r\leq 2)}^{2q}\int_{r^*\leq r\leq 2}\frac{(\la z\ra \pa_z w)^2}{\la z\ra(1+D^{2Kq})}dY\lesssim e^{-c_2q\nu \tau}.
\eee

\noindent{\it Crossed term}.
From \eqref{estlogt}, \eqref{linfryvouter}:
$$
\left| \int\pa_{zz}(\log T)w_2\tilde{\chi}w_2^{2q+1} dY\right|\lesssim \int_{r\leq 2}  w_2^{2q+2}\tilde{\chi} dY\leq r^*\int_{r\leq 2}  \frac{w_2^{2q+2}}{r^2}\tilde{\chi} dY+e^{-c_2\nu q \tau}.$$
The crossed term in integrated by parts in $r$ using \eqref{estpr;ogt} and the $L^\infty$ bound \eqref{linfryvouter}:
\bee
&&\left|\int_{r\leq 2}\pa_{rz}\log Tw_1w_2^{2q+1}\tilde{\chi} dY\right|\lesssim e^{-c_2\nu q\tau }\\
&+& \int |w|\left[|\pa_{rz}\log T|\pa_r(w_2^{2q+1})|+\frac{|\pa_{rz}\log T||w_2^{2q+1}|}{r}+|\pa_{rrz}\log T||w_2^{2q+1}|\right]\tilde{\chi} dY\\
& \lesssim & e^{-c_2\nu q \tau}+\int |w|\left[\frac{|w^q_2||\pa_r(w_2^{q+1})|}{r\la z\ra}+\frac{|w_2|^{2q+1}}{r^2\la z\ra}\right]\tilde{\chi}dY
\eee
and then from H\"older:
\bee
 \int |w|\frac{|w^q_2||\pa_r(w_2^{q+1})|}{r\la z\ra}\tilde{\chi}dY&\leq&\kappa \int |\pa_r(w_2^{q+1})|^2\tilde{\chi}dY+\frac{1}{\kappa}\int \frac{w^2w_2^{2q}}{r^2\la z\ra^2}\tilde{\chi}dY\\
 & \leq & \kappa \int \left[|\pa_r(w_2^{q+1})|^2+\frac{w_2^{2q+2}}{r^2}\right]\tilde{\chi}dY+c_\kappa\int \frac{w^{2q+2}}{r^2\la z\ra^{2q+2}}\tilde{\chi}dY\\
 & \leq & \kappa \int \left[|\pa_r(w_2^{q+1})|^2+\frac{w_2^{2q+2}}{r^2}\right]\tilde{\chi}dY+c_\kappa \int \frac{w^{2q+2}}{r^2}\chi dY.
 \eee
 similarly for the second term:
 \bee
  \int |w|\frac{|w_2|^{2q+1}}{r^2\la z\ra}\tilde{\chi}dY&\leq&  \kappa \int\frac{w_2^{2q+2}}{r^2}\tilde{\chi}dY+c_\kappa \int \frac{w^{2q+2}}{r^2\la z\ra^{2q+2}}\tilde{\chi}dY\\
  & \leq &  \kappa \int \left[|\pa_r(w_2^{q+1})|^2+\frac{w_2^{2q+2}}{r^2}\right]\tilde{\chi}dY+c_\kappa \int \frac{w^{2q+2}}{r^2}\chi dY.
  \eee
Next from \eqref{estvone}, \eqref{linfryvouter}:
\bee
&&\left|\int(\pa_zV_1)ww_2^{2q+1}\tilde{\chi} dY\right|\lesssim \int\frac{|w|}{\la z\ra} |w_2^{2q+1}|\tilde{\chi} dY\lesssim \int_{r\leq 2}|w||\la z\ra w_2|^{2q+1}\chi dY\\
& \leq &  r^*\int_{r \leq 2}[w^{2q+2}+(\la z\ra w_2)^{2q+2}]\frac{ \chi dY}{r^2}+c(r^*)\int_{r^*\leq r\leq 2}[w^{2q+2}+(\la z\ra w_2)^{2q+2}]\chi dY\\
& \leq & r^*\int_{r \leq 2}\frac{w^{2q+2}+(\la z\ra w_2)^{2q+2}}{r^2}\chi dY+e^{-c_2\nu q \tau}
\eee

\noindent{\it $\Psi_3,L(\zeta)$ term}. From H\"older and \eqref{ceibbeokboeoebvbisbis}, \eqref{ceibbeokboeoebvbisbisL}:
\bee
\left|\int_{r\leq 2}\pa_z \left(\frac{\psi_3+L(\zeta)}{T}\right)w_1^{2q+1}\tilde{\chi} dY\right|&\leq& \kappa \int \frac{w_2^{2q+2}}{r^2}\tilde{\chi} dY+c_\kappa\int \left[r^{2\frac{2q+1}{2q+2}}\pa_z\left(\frac{\Psi_3+L(\zeta)}{T}\right)\right]^{2q+2}\tilde{\chi} dY\\
&\leq & \kappa \int \frac{w_2^{2q+2}}{r^2}\chi dY+b^{c\eta q}
\eee

\noindent{\it Nonlinear term}. We estimate by homogeneity using $\pa_z\zeta=0$
\bee
|\pa_z\NL|\lesssim |\pa_z\Phi_{a,b}|(|V|^{p-1}+|V|^2|\Phi_{a,b}|^{p-3})+|\pa_zv|(|V|^{p-1}+|\Phi_{a,b}|^{p-2}|V|).
\eee

Moreoever using \eqref{estpr;ogt}: $$|\pa_zv|\lesssim |\pa_zT||w|+|T\pa_zw|\lesssim T\left(\frac{|w|}{\la z\ra}+|w_2|\right)$$ and $$|\pa_z\Phi_{a,b}|\lesssim \frac{T}{\la z\ra}, \ \ |T|\lesssim \Phi_{a,b}\lesssim \frac{1}{r^{\frac{2}{p-1}}}.$$
\underline{$D\leq A$}. This  implies for $r\leq 2$, $D\le A$ using \eqref{estphibrutale}:
\bee
\frac{|\pa_z\NL|}T&\lesssim & \frac{1}{\la z\ra}\left[|\zeta|^{p-1}+\zeta^2\Phi_{a,b}^{p-3}\right]\\
&+ & \frac{1}{\la z\ra}\left[T^{p-1}|w|^{p-1}+w^2T^2\Phi_{a,b}^{p-3}\right]+\left[T^{p-1}|w|^{p-1}+wT\Phi_{a,b}^{p-2}\right]\left(\frac{|w|}{\la z\ra}+|w_2|\right)\\
& \lesssim & \|w\|_{L^\infty(r\le 2,D\leq A)}\left(\frac{|w|}{\la z\ra}+|w_2|\right)\frac{1}{r^2}+\frac{b^g}{\la z \ra r^2}
\eee
similarly $$\left|\NL\frac{\pa_zT}{T^2}\right|\lesssim \frac{|\NL|}{\la z\ra T}\lesssim \|w\|_{L^\infty(r\le 2,,D\leq A)}\frac{|w|}{r^2 \la z\ra}+\frac{b^g}{\la z \ra r^2}$$ and hence the bound using H\"older:
\bee
&&\left|\int_{D\leq A}\pa_z\left(\frac{\NL}{T}\right)w_2^{2q+1}\tilde{\chi}dY\right|\lesssim\int \left[\|w\|_{L^\infty(r\le 2,D\leq A)}\left(\frac{|w|}{\la z\ra}+|w_2|\right)+b^{g}\right]\frac{w_2^{2q+1}}{r^2}\tilde{\chi}dY\\
& \leq & \kappa \left[\int \frac{w_2^{2q+2}}{r^2}\tilde{\chi}dY+\int\frac{w^{2q+2}}{r^2}\chi dY\right]+b^{gq}
\eee
where we used \eqref{boundlinfty} in the last step.\\
\underline{$D\geq A$}. Since $r\leq 2$ and $D\geq A\gg r$, we have $$T\sim \Phi_{a,b}\sim \frac{1}{D^{\frac{2}{p-1}}}\leq 1$$ and hence for $r\leq 2$:
\bee
\frac{|\pa_z\NL|}{T}&\lesssim &\frac1{\la z\ra}\left[|V|^{p-1}+\Phi_{a,b}^{p-3}V^2\right]+(|V|^{p-1}+|\Phi_{a,b}|^{p-2}|V|)\left(\frac{|w|}{\la z\ra}+|w_2|\right)\\
& \lesssim & \frac{1}{\la z\ra}\frac{\|\phi V\|_{L^\infty}}{r^2}\left(\frac{|w|}{\la z\ra}+|w_2|+(\sqrt b)^{\frac 12}\right)
\eee
and the same chain of estimates as above yields for $c$ universal:
$$
\left|\int_{D\geq A}\pa_z\left(\frac{\NL}{T}\right)w_2^{2q+1}\tilde{\chi}dY\right|\leq  \kappa \left[\int \frac{w_2^{2q+2}}{r^2}\tilde{\chi}dY+\int\frac{w^{2q+2}}{r^2}\chi dY\right]+b^{cq}.
$$
The collection of above bounds yields the pointwise differential inequation since $0<\nu \ll \eta \ll 1$:
\be
\label{neicnnoeeno}
\frac{1}{2q+2}\frac d{dt}\int w_2^{2q+2}\tilde{\chi} dY+cq\int\frac{w_2^{2q+2}}{r^2}\tilde{\chi} dY\lesssim \int \frac{w^{2q+2}}{r^2}\chi dY+e^{-c_2\nu q \tau}
\ee

\noindent{\bf step 3} Conclusion. The time integration of \eqref{neicnnoeeno} using the space time bound \eqref{cneineoneonevnio} yields that there exists $0<c_4\lesssim c_3$ such that 
$$
e^{c_4\nu \tau q}\int_{r\leq 1} \frac{(\la z\ra \pa_zw)^{2q+2}}{\la z\ra(1+D^{2Kq})}dY+\int_{\tau_0}^\tau e^{c_4\nu \tau'q}\int_{r\leq 1} \frac{(\la z\ra \pa_zw)^{2q+2}}{\la z\ra(1+D^{2Kq})}dYd\tau'\leq 1
$$
which implies \eqref{cneineoneonevniobisbis}. Injecting this bound into \eqref{firnsovneno4} and integrating in time yields the lossy bound \eqref{cneineoneonevniobis}.
\end{proof}

%%%%%%%%%%%%%%%%%%%%%%%%%%%%%%%%%%%%%%%%%%%

\subsection{Far away scaling bound}
\label{faraway}
%%%%%%%%%%%%%%%%%%%%%%%%%%%%%%%%%%%%%%%%%%%

Recall \eqref{eqebis}. We now claim the far away $W^{1,2q+2}$ bound which is a simple consequence of the fact that this norm is always above scaling for $q>q(d,p)$ large enough.
 We let $q\gg1$ and
\be
\label{dfzstr}
(z^*)^{2\ell}=\frac{1}{a(\sqrt{b})^\alpha}\gg 1.
\ee
Given a large enough constant $A$, we let a cut off function $$\chi_{A,b}(r,z)=\chi\left(\frac{r}A,\frac{z}{Az^*}\right), \ \ \chi(Y)=\left|\begin{array}{ll} 0\ \ \mbox{for} \ \ |Y|\leq 1,\\ 1\ \ \mbox{for} \ \ |Y|\geq 2.\end{array}\right.$$ Note that for $az^{2\ell}\gg1$ :
$$D=\mu\sqrt{b}=\left[az^{2\ell}\left(1+O\left(\frac{1}{z^{2}}\right)\right)\right]^{\frac 1\alpha}\sqrt{b}= \left(\frac z{z^*}\right)^{\frac{2\ell}{\alpha}}\left[1+O\left(\frac{1}{z^{2}}\right)\right]$$ and hence 
\be
\label{propertysupport}
Y\in \rm{Supp}\chi_{A,b} \ \ \mbox{implies} \ \ r\gtrsim  A\ \ \mbox{or}\ \ D\gtrsim A.
\ee

\begin{lemma}[Far away scaling bound]
There holds for a constant $c_5$ depending on $\nu$, $c_2$ and $c_3$:
\be
\label{sobolevboundoutside}
\int \chi_{A,b}\left[(\la z\ra\pa_z)^jV\right]^{2q+2}dY\lesssim b^{c_5 q}, \ \ j=0,1
\ee
and for a universal constant $C$ independent of $q$:
\be
\label{sobolevboundoutsidebis}
\int \chi_{A,b}\left[\pa_rV\right]^{2q+2}dY\lesssim \frac{1}{b^{Cq}}.
\ee

\end{lemma}

\begin{proof} 
We compute the $V$ equation: $$\pa_\tau V-\Delta V+\frac 12\Lambda V-R(V)=0, \ \ R(V)=(\Phi_{a,b}+V)|\Phi_{a,b}+V|^{p-1}-\Phi_{a,b}^p.$$

\noindent{\bf step 1} $L^{2q+2}$ estimate. We compute:
\bee
&&\frac{1}{2q+2}\frac{d}{dt}\int \chi_{A,b}V^{2q+2}dY\\
&=&\int\left(\Delta V-\frac 12\Lambda V+R(V)\right)\chi_{A,b}V^{2q+1}dY+\frac{1}{2q+2}\int \pa_t\chi_{A,b}|V|^{2q+2}dY\\
&=&-(2q+1)\int V^{2q}|\nabla V|^2\chi_{A,b}dY-\frac 12\left(\frac2{p-1}-\frac{d}{2q+2}\right) \int |V|^{2q+2}\chi_{A,b}dY\\
&+&\int R(V)V^{2q+1}\chi_{A,b}dY+ \frac{1}{2q+2}\int V^{2q+2}\left[\pa_t\chi_{A,b}-\Delta \chi_{A,b}+\frac 12Y\cdot\nabla\chi_{A,b}\right]
\eee
From \eqref{propertysupport}: $$|\Phi_{a,b}|
\lesssim \frac1{A^{\frac{2}{p-1}}} \ \ \mbox{on}\ \ \rm{Supp}\chi_{A,b}$$ and hence using \eqref{boundlinfty}, \eqref{propertysupport}: $$|V|\lesssim \N, \ \ |R(V)|\lesssim \frac1{A^c}|V|\ \ \mbox{on}\ \ {\rm Supp}(\chi_{A,b}).$$ 
Hence the bound:
$$\int |R(V)V^{2q+1}|\chi_{A,b}dY\leq  (\sqrt b)^{\tilde{\eta}}\int |V|^{2q+2}\chi_{A,b}dY.$$
We then estimate by definition of $\chi_{A,b}$ and \fref{estparameters}: $$|\pa_t\chi_{A,b}|+|\Delta \chi_{A,b}|+|Y\cdot\nabla\chi_{A,b}|\lesssim {\bf 1}_{A\le r\le 2A,D\leq 2A}+{\bf 1}_{r\le 2A,cA\leq D\leq CA}.$$ 
From \eqref{estzeta}, \eqref{linfryvouter}:
\bee
&&\int_{A\le r\le 2A,D\leq 2A}|V|^{2q+2}dY \lesssim \int_{A\le r\le 2A,D\le 2A}|\zeta|^{2q+2}dY+\int_{A\le r\le 2A,D\le 2A}|v|^{2q+2}dY\\
&\lesssim & (\sqrt{b})^{\alpha(2q+2)}\left|\left\{z, \ D\leq 2A \right\} \right|+e^{-c_2\nu \tau(2q+2)}\left|\left\{z, \ D\leq 2A \right\} \right|\\
&\lesssim & (\sqrt{b})^{\alpha(2q+2)}(\sqrt b)^{-\frac{\alpha}{2\ell}}+(\sqrt b)^{\frac{c_2\alpha}{2\ell-\alpha}\nu (2q+2)}(\sqrt b)^{-\frac{\alpha}{2\ell}}\lesssim  b^{c q}
\eee
for $q$ large enough and $c$ depending on $\nu$ and $c_2$. Next from \eqref{estzeta}, \eqref{linfryvouter}, as $g>3/2$ and $p\geq 3$:
\bee
&&\int_{r\leq A}|\zeta|^{2q+2}r^{d-1}dr\lesssim (\sqrt{b})^{\alpha (2q+2)}+\int_{r\le \sqrt{b}}(\sqrt{b})^{(2q+2)(g-\frac{2}{p-1})}r^{d-1}dr\\
&+& \int_{\sqrt{b}\le r\leq 1}\left(\frac{(\sqrt{b})^\alpha(r^2+(\sqrt{b}^g))}{r^{\gamma}}\right)^{2q+2}r^{d-1}dr\lesssim (\sqrt{b})^{q}.
\eee
We then use $$|v|=|Tw|\lesssim \frac{|w|}{D^{\frac{2}{p-1}}}$$ to estimate using \eqref{cneineoneonevnio}:
$$\int_{D\sim A,r\leq 2A}|V|^{2q+2}dY\lesssim \int_{D\sim A,r\leq 2A}\frac{|w|^{2q+2}}{\la z\ra(1+D^{2qK})}dY+\frac{1}{b^c}\int_{r\leq A}|\zeta|^{2q+2}r^{d-1}dr\lesssim b^{c q}$$
with $c$ depending on $\nu$ and $c_3$. The collection of above bounds yields the pointwise differential inequation:
$$ \frac{d}{dt}\int \chi_{A,b}V^{2q+2}dY+cq\int \chi_{A,b}V^{2q+2}dY\lesssim b^{C q}$$
with $C$ depending on $\nu$, $c_2$ and $c_3$, whose time integration yields $$\int \chi_{A,b}V^{2q+2}dY\lesssim b^{C q}.$$

{\bf step 2} Derivative estimate. The derivative bounds follow the exact same path using in particular $\pa_z\zeta=0$ and the inner bounds \eqref{cneineoneonevniobis}, \eqref{cneineoneonevniobisbis}, the details are left to the reader.

\end{proof}

%%%%%%%%%%%%%%%%%%%%%%%%%%%%%%%%%%%%%
%%%%%%%%%%%%%%%%%%%%%%%%%%%%%%%%%%%%%

\section{Closing the bootstrap}

%%%%%%%%%%%%%%%%%%%%%%%%%%%%%%%%%%%%%
%%%%%%%%%%%%%%%%%%%%%%%%%%%%%%%%%%%%%

We are now in position to close the bootstrap and conclude the proof of Theorem \ref{thmmain}.

%%%%%%%%%%%%%%%%%%%%%%%%%%%%%%%%%%%%%

\subsection{Proof of Proposition \ref{propbootstrap}} 

%%%%%%%%%%%%%%%%%%%%%%%%%%%%%%%%%%%%%

We argue by contradiction and conclude using a classical topological argument.\\

\noindent{\bf step 1} The (Exit) condition is saturated. From \eqref{controlponctuelenergy}, $$\|\e\|_{L^2_{\rho_Y}}\leq \eta(a) (\sqrt{b})^{\alpha+\eta}$$ and hence \eqref{smallnorminitboot} is improved. Moreoever, we inject the bounds \eqref{boundL^2inz}, \eqref{cneineoneonevnio}, \eqref{cneineoneonevniobis}, \eqref{cneineoneonevniobisbis}, \eqref{sobolevboundoutside}, \eqref{sobolevboundoutsidebis} and conclude that there exists $\delta >0$ depending on $\nu,c_1,c_2,c_3,c_4,c_5$ but independent of $q$ and $\tilde{\eta}$, and a universal $C>0$ such that:
$$\matchal N^q\lesssim e^{-\delta q\tau}+\left(e^{-\delta q\tau}\right)^{1-\frac{d}{2q+2}}\left(\frac{1}{e^{Cq\tau}}\right)^{\frac{2}{2q+2}}\lesssim e^{-\delta q\tau}$$
provided $q$ has been chosen universal large enough since $\delta,C$ are independent of q, and hence choosing $\tilde{\eta}\ll \delta$, \eqref{smallnorminitsobolevboot} is improved: $\mathcal N\ll (\sqrt b)^{\tilde{\eta}}$. We now oberve that \eqref{estparameters} implies:
\be
\label{cnekneoneon}
|a_\tau|+\left|\frac{d}{d\tau}\left(b_{\ell,0}e^{\left(\ell-\frac \alpha 2\right)\tau}\right)\right|\lesssim (\sqrt{b})^\eta.
\ee
 whose time integration ensures:
$$|a-a(\tau_0)|+\left|b_{\ell,0}(\tau)e^{\left(\ell-\frac \alpha 2\right)\tau}-b_{\ell,0}(\tau_0)e^{\left(\ell-\frac \alpha 2\right)\tau_0}\right|\leq b_0^{c\eta}$$
which using \eqref{initmode}, \eqref{condiotnazero} improves \eqref{initmodeboot}, \eqref{controlea}.
We conclude from a standard continuity argument that the (Exit) condition is saturated, meaning that a solution exits the trapped regime at time $\tau^*$ if and only if the instable modes have grown too big: 
\be
\label{esnokeneon}
(\tilde b,\tilde b_{j,k})\in \frac{(\sqrt{b}(\tau^*))^{\eta}}{(\sqrt{b(\tau_0)})^{\eta}}\mathcal S.
\ee

\noindent{\bf step 2} The topological argument. We now reformulate the (Exit) condition in diagonal form as required from \eqref{estparameters}. Indeed, we may diagonalize the associated matrix which has positive eigenvalues $\mu_{j,k}\geq 0$ such that letting $$(\tilde{b},(\tilde{b}_{jk})_{(j,k)\neq(\ell,0),j+k\leq \ell})=P\tilde{B}_{jk}$$ for some suitable universal invertible matrix,  \eqref{estparameters} implies:
\be
\label{estparametersbis}
\sum_{j,k}\left|(\tilde{B}_{jk})_\tau+\mu_{j,k}\tilde{B}_{jk}\right|\lesssim \eta(a)(\sqrt{b})^\eta.
\ee
We now choose a small enough universal constant $\kappa$
and consider initial values satisfying $$\sum_{j,k} \left|\frac{\tilde{B}_{j,k}(\tau_0)}{(\sqrt{b})^{\eta}(\tau_0)}\right|^2\leq \kappa^2$$ then from \eqref{esnokeneon}, \eqref{controlponctuelenergy}, we may let $\tau^{**}<\tau^*$ the first time such that $$\sum_{j,k} \left|\frac{\tilde{B}_{j,k}}{(\sqrt{b})^{\eta}}(\tau^{**})\right|^2=\kappa^2,$$

we claim that the corresponding vector field is outgoing:
\be
\label{outgoingcondition}
\frac{d}{d\tau}\left\{\sum_{j,k} \left|\frac{\tilde{B}_{j,k}}{(\sqrt{b})^{\eta}}\right|^2\right\}(\tau^{**})>0.
\ee
Assume \eqref{outgoingcondition} and that all solutions leave the trapped regime, then from standard argument, the map 
$$\frac{1}{\kappa}\left(\frac{\tilde{B}_{j,k}}{(\sqrt{b})^{\eta}}(\tau_0)\right)\mapsto \frac{1}{\kappa}\left(\frac{\tilde{B}_{j,k}}{(\sqrt{b})^{\eta}}(\tau^{**})\right)
$$ is continuous. It moreover sends by definition the unit sphere onto its boundary and is the identity when restricted to the boundary, a contradiction to Brouwer's theorem.\\
\noindent{\it Proof of \eqref{outgoingcondition}}: We compute:
\bee
&&I=\frac{d}{d\tau}\left\{\sum_{j,k} \left|\frac{\tilde{B}_{j,k}}{(\sqrt{b})^{\eta}}\right|^2\right\}=  \frac{1}{(\sqrt{b})^{2\eta}}\left[\sum_{j,k}2\tilde{B}_{j,k}\left((\tilde{B}_{j,k})_\tau-\frac\eta 2\frac{b_\tau}{b}\tilde{B}_{j,k}\right)\right].
\eee
We now observe from \eqref{defgibbg}, \eqref{estimateB}, \eqref{estparameters}: $$\frac{b_\tau}{b}=1-\frac{2\ell}{\alpha}+\eta(a)O\left((\sqrt{b})^\eta\right)$$ and hence from \eqref{estparametersbis}:
$$(\tilde{B}_{j,k})_\tau-\frac\eta 2\frac{b_\tau}{b}\tilde{B}_{j,k}=\left[\mu_{j,k}+\frac{\eta}{2}\left(\frac{2\ell}{\alpha}-1\right)\right]\tilde{B}_{j,k}+\eta(a) O\left((\sqrt{b})^\eta\right).
$$
Recalling $\mu_{j,k}\geq 0$ 

and the pointwise differential inequation \eqref{controlenormee} with $0<\eta\ll c_*$, we obtain:
\bee
I &\geq& \frac{c\eta}{(\sqrt{b})^{2\eta}}\left[\sum_{j,k} \left|\frac{\tilde{B}_{j,k}}{(\sqrt{b})^{\eta}}(\tau^{**})\right|^2\right]+  \frac{(\sqrt{b})^\eta}{(\sqrt{b})^{2\eta}}\eta(a)O\left((\sqrt{b})^\eta\right)\\
& \geq & c\kappa -\eta(a)>0,
\eee
and \eqref{outgoingcondition} is proved.

%%%%%%%%%%%%%%%%%%%%%%%%%%%%%%%%%%%%%

\subsection{Proof of Theorem \ref{thmmain}}

%%%%%%%%%%%%%%%%%%%%%%%%%%%%%%%%%%%%%

 Let $0<T\ll1$ small enough and an initial data as in the assumptions of Proposition \ref{propbootstrap}, then the corresponding solution to \eqref{heatnonlin} generates a global in time to the self similar solution \eqref{selfsimilarequation} which admits for all $t\in [0,T)$ a decomposition $$u(t,x)=\frac{1}{(T-t)^{\frac{1}{p-1}}}U\left(t,\frac{x}{\sqrt{T-t}}\right),\ \ U(t,r,z)=\Phi_{a,b}+V.$$  The law \eqref{loidea} follows from \eqref{cnekneoneon}, \eqref{smallnorminitsobolevboot} which yield the time integrability $$\int_{\tau_0}^{+\infty}|a_\tau|d\tau<b_0^{c\eta}<\frac{a_0}{10},$$ and similarly for \eqref{loideb}.
  It remains to prove the asymptotic stability \eqref{linftydecay}. Recall \eqref{boundlinfty} which first implies $$\lim_{t\to T}\left(\|V\|_{L^\infty(D\geq A)}+\|V\|_{L^\infty(r\geq 1)}\right)= 0.$$ We then consider $D\leq A,r\le 1$ and estimate in brute force:
\bee
(\sqrt{b})^{\frac{2}{p-1}}\|V\|_{L^\infty(D\leq A,r\leq 1)}&\lesssim &(\sqrt{b})^{\frac{2}{p-1}}\|\zeta\|_{L^\infty(D\leq A,r\leq 1)}+\|(\sqrt{b})^{\frac{2}{p-1}}Tw\|_{L^\infty(D\leq A,r\leq 1)}\\
&\lesssim& b^{\delta}+\|w\|_{L^\infty(D\leq A,r\leq 1)}\to  0\ \ \mbox{as}\ \ t\to T
\eee
where we used $D=\mu\sqrt{b}\geq\frac 12 \sqrt{b}$ and \eqref{estzeta}, \eqref{boundlinfty}. This concludes the proof of Theorem \ref{thmmain}.

%%%%%%%%%%%%%%%%%%%%%%%%%%%%%%%%%%%%%
%%%%%%%%%%%%%%%%%%%%%%%%%%%%%%%%%%%%%

\begin{appendix}

%%%%%%%%%%%%%%%%%%%%%%%%%%%%%%%%%%%%%
%%%%%%%%%%%%%%%%%%%%%%%%%%%%%%%%%%%%%

\section{Estimates on the soliton}
\label{appendsoliton}
%%%%%%%%%%%%%%%%%%%%%%%%%%%%%%%%%%%%%
%%%%%%%%%%%%%%%%%%%%%%%%%%%%%%%%%%%%%

This appendix is devoted to the derivation of sharp estimates on the soliton profile.

\begin{lemma}[Global control of the tails]
Let $V=\log \Lambda Q$, then $\forall r>0$,
\bea
\label{estpoitnwisepotential}
&&-\frac{\gamma}{r}<\pa_rV\leq 0,\\
\label{estpoitnwisepotentialbis}&& \pa_{rr}V\leq \frac{\gamma}{r^2}.
\eea
\end{lemma}

\begin{proof}  The lemma follows from the $Q$ equation in the regime $p>p_{JL}(d)$ and a Sturm Liouville oscillation argument.\\

\noindent{\bf step 1} Proof of \eqref{estpoitnwisepotential}. Let from \eqref{positivitylamdbaq} 
\be
\label{cnknbenon}
H=-\Delta_r-pQ^{p-1}>H_0=-\Delta -\frac{\gamma(d-2-\gamma)}{r^2}.
\ee 
Let $\psi_1=\Lambda Q$,  $\psi_2=\frac{1}{r^\gamma}$, then
\be
\label{eqlamadnjis}\psi_1>0, \ \ H_0\psi_1<H\psi_1=0, \ \ H_0\psi_2=0.
\ee 
This first implies $$\frac{1}{r^{d-1}}\pa_r(r^{d-1} \pa_r\psi_1)=-pQ^{p-1}\psi_1<0$$ from which with $\pa_r\psi_1(0)=0$: $$r^{d-1}\pa_r\psi_1\leq 0, \ \ \pa_r\psi_1\leq 0 \ \ \mbox{and hence} \ \ \pa_rV=\frac{\pa_r\psi_1}{\psi_1}\leq 0.$$ Next from \eqref{eqlamadnjis}, the  Wronskian $W=\psi_1'\psi_2-\psi_1\psi_2'$ satisfies $$\frac{1}{r^{d-1}}\pa_r(r^{d-1}W)=(-H_0\psi_1)\psi_2>0.$$ At the origin, $$r^{d-1}W(r)=O(r^{d-1-(\gamma+1)})\to 0\ \ \mbox{as}\ \ r\to 0$$ and hence $W(r)>0$, $ \psi_1(r)>0$ implies $$-\frac{\psi_1'}{\psi_1}<-\frac{\psi_2'}{\psi_2}=\frac{\gamma}{r}$$ and \eqref{estpoitnwisepotential} is proved.\\

\noindent{\bf step 2} Proof of \eqref{estpoitnwisepotentialbis}. We first compute from \eqref{eqlamadnjis}: 
\be
\label{equationV}
\Delta V=\nabla \cdot\left(\frac{\pa_r\Lamdba Q}{\Lambda Q}\right)=\frac{\Delta \Lambda Q}{\Lambda Q}-\left(\frac{\pa_r\Lambda Q}{\Lambda Q}\right)^2=-pQ^{p-1}-(\pa_r V)^2.
\ee
This implies:
\be
\label{ebceibebiebe}
\pa_{rr}V+\frac{\pa_rV}{r}=\Delta V-\frac{d-2}{r}\pa_r V=-pQ^{p-1}-\pa_rV\left[\frac{d-2}{r}+\pa_rV\right].
\ee
Let $\Phi=-r\pa_rV$, then from \eqref{estpoitnwisepotential}: $0<\Phi<\gamma$ and $\Phi$ satisfies the order one nonlinear equation:
\be
\label{cneneonoe}
r\pa_r\Phi=pQ^{p-1}r^2-\Phi(d-2-\Phi).
\ee 
We claim
\be
\label{moineonvn}
\pa_r\Phi> 0.
\ee
Indeed, at the origin, $\pa_rV(0)=0$ and hence from \eqref{equationV}: $$\Delta V(0)=d\pa_{rr}V(0)=-pQ^{p-1}(0)$$ from which near the origin:
\bee
\frac{\pa_r\Phi}{r}&=& pQ^{p-1}+\frac{\pa_r V}{r}(d-2-\Phi)=pQ^{p-1}(0)+(d-2)\pa_r^2V(0)+o_{r\to 0}(1)\\
& =& pQ^{p-1}(0)\left(1-\frac{d-2}{d}\right)+o_{r\to 0}(1)>0.
\eee
By contradiction, let $r_0>0$ be the first point where $\Phi'(r_0)=0$, then $\Phi''(r_0)\leq 0$ but deriving \eqref{cneneonoe} at $r=r_0$ yields:
$$r_0\Phi''(r_0)=p(p-1)rQ^{p-2}\Lambda Q(r_0)>0$$ and a contradiction follows which concludes the proof of \eqref{moineonvn}. We have therefore proved $$0\leq \Phi\leq \gamma, \ \ \pa_r\Phi>0.$$ Now $$\pa_r\Phi=-r\pa_{rr}V-\pa_rV=-r\pa_{rrV}+\frac{\Phi}{r}, \ \ \pa_{rr}V=\frac{\Phi}{r^2}-\frac{\pa_r\Phi}{r}\leq \frac{\gamma}{r^2},$$ this is \eqref{estpoitnwisepotentialbis}. 
\end{proof}

\begin{lemma}[Estimates on $\log T$ and $V_1$]
Let $T,V_1$ be given by \eqref{defww}, \eqref{devone}, then for all $r\leq 2$:
\be
\label{estvone}
 |V_1|+|r\pa_rV_1|+|\la z\ra \pa_zV_1|\lesssim 1, \ \
 \ee

 \be
 \label{estlogt}
|\pa_z^{k}\pa_r^j(\log T)|\lesssim \frac{1}{r^j\la z\ra^k}, \ \ k,j\geq 0, \ \ k+j\ge 1
 \ee

\end{lemma}

\begin{proof} Recall $$T=\frac{1}{D^\gamma}\Lambda Q\left(\frac{r}{D}\right), \ \ D=\mu(z)\sqrt{b}=\sqrt{b}\left(1+aP_{2\ell}(z)\right)^{\frac 1\alpha}.$$ Hence $$\frac{\pa_zT}{T}=-\frac{\pa_z D}{D}\zeta\left(\frac{r}{D}\right)\ \ \mbox{where} \ \ \zeta(y)=\gamma+\frac{y\pa_y\Lambda Q}{\Lambda Q}$$ satifies $$\pa_y^k\zeta=O\left(\frac{1}{\la y\ra^{k}}\right), \ \ k\in \Bbb N.$$ Moreover, $$\frac{\pa_z D}{D}=\frac{\pa_z\mu}{\mu}=-\frac{1}{\alpha}\frac{aP'_{2\ell}}{1+aP_{\ell}}$$ so that $$\pa_z^k\left(\frac{\pa_z D}{D}\right)=O\left(\frac{1}{\la z\ra^{k+1}}\right).$$ Similarly, $$|\pa_r\log T|\lesssim \frac 1D\frac{\pa_y\Lambda Q}{\Lambda Q}\left(\frac rD\right)\lesssim \frac 1r,$$ and \eqref{estlogt} follows by induction. 
We now compute:
$$\pa_\tau \log T=-\frac{\pa_t D}{D}\zeta\left(\frac{r}{D}\right)=\left[-\frac{b_\tau}{2b}+\frac{1}\alpha\frac{a_\tau P_{2\ell}}{1+aP_{2\ell}}\right]\zeta\left(\frac{r}{D}\right)$$ and since $$\left|\frac{b_\tau}{b}\right|+\left|\frac{a_\tau}{a}\right|\lesssim 1,$$ we obtain 
\be
\label{estpr;ogt}
\left|\pa_z^k\pa_r^j(\pa_\tau \log T)\right|=O\left(\frac{1}{r^j\la z\ra^k}\right), \ \ k+j\ge 1.
\ee
\end{proof}

%%%%%%%%%%%%%%%%%%%%%%%%%%%%%%%%%%%%%
%%%%%%%%%%%%%%%%%%%%%%%%%%%%%%%%%%%%%

\section{Coercivity estimates}
\label{appendcoerc}
%%%%%%%%%%%%%%%%%%%%%%%%%%%%%%%%%%%%%
%%%%%%%%%%%%%%%%%%%%%%%%%%%%%%%%%%%%%

This Appendix is devoted to the proof of Hardy and Sobolev like inequalities that are used all along the proof. The proofs are standard and recalled for the convenience of the reader.

\begin{lemma}[Hardy inequality with Gaussian weight]
There holds for $\e$ with cylindrical symmetry:
\be
\label{hardyapoids}
\int \left(|Y|^2+\frac{1}{r^2}\right)\e^2e^{-\frac{|Y|^2}{4}}dY\lesssim \int (|\nabla \e|^2+\e^2)e^{-\frac{|Y|^2}{4}}dY
\ee
\end{lemma}

\begin{proof} The proof follows a classical integration by parts and is left to the reader.
\end{proof}

\begin{lemma}[Weighted Hardy inequalities]
There holds for all $\beta>-\frac{d-2}{2}$:
\be
\label{hardyinqeualityweighted}
\int r^{2\beta}|\pa_ru|^2\rho_rdY\geq \left(\frac{d-2+2\beta}{2}\right)^2\int \frac{r^{2\beta}}{r^2}u^2\rho_rdY-\frac{d-2+2\beta}{4}\int r^{2\beta}u^2\rho_rdY.
\ee
\end{lemma}

\begin{proof}
We integrate by parts and use Young inequality to compute:  
\bee
&&\int \frac{u^2}{r^2}\chi \rho_rdY=\frac{1}{d-2}\int \chi u^2 \nabla\cdot \left( \frac {e_r}{r}\right)\rho_rdY=-\frac{1}{d-2}\int \frac1r\left[2\chi u\pa_ru+\chi' u^2-\frac{r}{2}\chi u^2\right]\rho_rdY\\
& \leq & \frac{1}{d-2}\int \frac{u^2}{r^2}\left(-r\chi'+2\beta\chi+\frac{r^2}{2}\chi\right)\rho_rdY-\frac{2\beta}{d-2}\int \frac{u^2}{r^2}\rho_rdY
\\
&+& \frac{1}{d-2}\left[\frac{1}{A}\int \chi|\pa_ru|^2\rho_rdY+A\int \frac{u^2}{r^2}\chi\rho_rdY\right]
\eee
and hence:
\bee
\int \chi|\pa_ru|^2\rho_rdY\geq (d-2)A\int \frac{\chi}{r^2}u^2\left(1+\frac{2\beta-A}{d-2}\right)\rho_rdY+A\int \frac{u^2}{r^2}\left(r\chi'-2\beta\chi-\frac{r^2}{2}\chi\right)\rho_rdY.
\eee
For $\beta>-(d-2)$, the optimal choice $A=\frac{d-2+2\beta}{2}>0$ ensures:
$$\int \chi|\pa_ru|^2\rho_rdY\geq\left(\frac{d-2+2\beta}{2}\right)^2\int \frac{\chi}{r^2}u^2\rho_rdY+\frac{d-2+2\beta}{2}\int \frac{u^2}{r^2}\left(r\chi'-2\beta\chi-\frac{r^2}{2}\chi\right)\rho_rdY.$$ Letting $\chi=r^{2\beta}$ yields the claim.
\end{proof}

\begin{lemma}[Weighted Sobolev bound]
Let $2q+2>d+1$ and $v(r,z)$ with cylindrical symmetry, then:  
\be
\label{estitmiatesobolevpoids}
\|v\|_{L^\infty}^{2q+2}\lesssim_{q}\left(\int \frac{|v|^{2q+2}+|\la z\ra \pa_zv|^{2q+2}}{\la z\ra}dY\right)^{1-\frac{d}{2q+2}}\left(\int \frac{|\pa_rv|^{2q+2}}{\la z\ra}dY\right)^{\frac{d}{2q+2}}.
\ee
\end{lemma}

\begin{proof} This follows from Sobolev and a scaling argument. We first claim:
\be
\label{cneiocneon}
\|v\|^{2q+2}_{L^{\infty}}\lesssim \int \frac{|v|^{2q+2}+|\pa_r v|^{2q+2}+|\la z \ra\pa_zv|^{2q+2}}{\la z\ra}dY.
\ee
Indeed, we have from Sobolev for $2q+2>d+1$:
\bee
\|v\|^{2q+2}_{L^{\infty}(|z|\leq 1)}&\lesssim & \int_{|z|\leq 1} \left(|v|^{2q+2}+|\pa_r v|^{2q+2}+|\pa_zv|^{2q+2}\right)dY\\
&\lesssim & \int \frac{|v|^{2q+2}+|\pa_r v|^{2q+2}+|\la z \ra\pa_zv|^{2q+2}}{\la z\ra}dY.
\eee

Let now $A\geq 1$ and the cylinder $\mathcal C_A=\{r\geq 0, \ \ \frac A2\leq |z|\leq A\}$ Let $V(r,z)=v(r,Az),$ then $2q+2>d+1$ and Sobolev in the cylinder $\mathcal C_1$  ensure: \bee
\|v\|^{2q+2}_{L^\infty(\matchal C_A)}&=&\|V\|^{2q+2}_{L^\infty(\mathcal C_1)}\lesssim  \|V\|_{L^{2q+2}(\mathcal C_1)}^{2q+2}+ \|\nabla V\|_{L^{2q+2}(\mathcal C_1)}^{2q+2}\\
& \lesssim & \frac{1}{A}\left[\|v\|_{L^{2q+2}(\mathcal C_A)}^{2q+2}+ \|\pa_r v\|_{L^{2q+2}(\mathcal C_A)}^{2q+2}+\|A\pa_z v\|_{L^{2q+2}(\mathcal C_A)}^{2q+2}\right]\\
& \lesssim  & \int \frac{|v|^{2q+2}+|\pa_r v|^{2q+2}+|\la z \ra\pa_zv|^{2q+2}}{\la z\ra}dY
\eee
and since the bound is independent of $A\geq 1$, \eqref{cneiocneon} is proved. We now apply this estimate to $$v_\l(r,z)=v\left(\frac{r}{\lambda},z\right)$$ which yields:
\bee
\|v\|^{2q+2}_{L^{\infty}}=\|v_\l\|^{2q+2}_{L^{\infty}}\lesssim \l^d\int \left[\frac{|v|^{2q+2}+|\la z \ra\pa_zv|^{2q+2}}{\la z\ra}+\frac{1}{\l^{2q+2}}\frac{|\pa_r v|^{2q+2}}{\la z\ra}\right]dY
\eee
and optimizing in $\l$ yields \eqref{estitmiatesobolevpoids}.
\end{proof}

\begin{lemma}[Weighted outer Sobolev bound]
Let $v(r,z)$ with cylindrical symmetry, then for all $0<\delta<R$:  
\be
\label{estitmiatesobolevpoidsbis}
\|v\|_{L^\infty(\delta\leq r\leq R)}^{2}\lesssim_{\delta,R} \sum_{0\leq i+j\leq 2}\int_{\delta\leq r\leq R} \frac{|\pa_r^i(\la z\ra \pa_z)^jv|^{2}}{\la z\ra}dY.
\ee
\end{lemma}

\begin{proof} Indeed, we have from the two dimensional Sobolev $H^2\subset L^{\infty}$:
\bee
\|v\|^{2q+2}_{L^{\infty}(\delta\leq r\leq R,|z|\leq 1)}&\lesssim_{q,\delta,R} &
\sum_{0\leq i+j\leq 2} \int_{\delta\leq r\leq R,|z|\leq 1} |\pa_r^i\pa_z^jv|^{2}dY\lesssim  \sum_{0\leq i+j\leq 2}  \int_{\delta\leq r\leq R} \frac{|\pa_r^i(\la z\ra \pa_z)^jv|^{2}}{\la z\ra}dY
\eee
Let now the cube $$\mathcal C_A=\{\delta\leq r\leq R, \ \ \frac A2\leq |z|\leq A\}, \ \  A\geq 1.$$ Let $V(r,z)=v(r,Az),$ then the two dimensional Sobolev in the cylinder $\mathcal C_1$  ensures: 
\bee
\|v\|^{2q+2}_{L^\infty(\matchal C_A)}&=&\|V\|^{2q+2}_{L^\infty(\mathcal C_1)}\lesssim_{\delta,R}
\sum_{0\leq i+j\leq 2} \int_{\delta\leq r\leq R,\frac 12\leq |z|\leq 1} |\pa_r^i\pa_z^jV|^{2}dY\\
& \lesssim_{\delta,R} & \frac{1}{A} \sum_{0\leq i+j\leq 2} \int_{\delta\leq r\leq R,\frac A2\leq |z|\leq A} |\pa_r^iA^j\pa_z^jv|^{2}dY\\
& \lesssim_{\delta,R} & \sum_{0\leq i+j\leq 2}  \int_{\delta\leq r\leq R} \frac{|\pa_r^i(\la z\ra \pa_z)^jv|^{2}}{\la z\ra}dY
\eee
and since the bound is independent of $A\geq 1$, \eqref{estitmiatesobolevpoidsbis} is proved.\end{proof}

%%%%%%%%%%%%%%%%%%%%%%%%%%%%%%%%%%%%%
%%%%%%%%%%%%%%%%%%%%%%%%%%%%%%%%%%%%%

\section{Nonlinear estimates on $\Psi_3$}
\label{appendnonlin}
%%%%%%%%%%%%%%%%%%%%%%%%%%%%%%%%%%%%%
%%%%%%%%%%%%%%%%%%%%%%%%%%%%%%%%%%%%%

This appendix is devoted to the control of various norms of the leading order driving term $\Psi_3$ and the error term $L(\zeta)$ given by \eqref{defpsithree}. We start with $L^2_{\rho_r}$ bounds with 
sharp weight in $z$.

\begin{lemma}[$L^2_{\rho_r}$ bounds on $\Psi_3$ and $L(\zeta)$]
We claim for $\nu>0$, $|b|<b^*(\nu)$:
\be \lab{boundspsioneone}
\int \frac{|(\la z\ra \pa_z)^j\Psi_3|^{2}}{1+z^{4(\ell+\nu)+1}}\rho_rdY\lesssim (\sqrt{b})^{2\alpha+2\eta}, \ \ j=0,1,2,3,
\ee
and for any $r_0>0$, for $i\geq 1$, $i+j\leq 3$, if $|b|<b^*(\nu,r_0)$:
\be
\label{boundspsioneone par}
\int_{r\geq r_0} \frac{|\pa^i_r(\la z\ra \pa_z)^j\Psi_3|^{2}}{1+z^{4(\ell+\nu)+1}}\rho_rdY\lesssim (\sqrt{b})^{2\alpha+2\eta}, \ \ i\geq 1, \ i+j\leq 3,
\ee
and
\be
\label{estlinearterm}
\int \frac{|(\la z\ra \pa_z)^jL(\zeta)|^2}{1+z^{4(\ell+\nu)+1}}\rho_rdY\lesssim (\sqrt{b})^{2\alpha+2g}, \ \ j=0,1,2,3.
\ee
\be
\label{estlinearterm2}
\int_{r\geq r_0} \frac{|\partial_r^i(\la z \ra \pa_z)^jL(\zeta)|^{2}}{1+z^{4(\ell+\nu)+1}}\rho_rdY\lesssim (\sqrt{b})^{4\alpha}, \ \ i\geq 1, \ i+j\leq 3.
\ee

\end{lemma}

\begin{proof} 

\noindent{\bf step 1} Control of $\Psit_1$. Recall \eqref{psitildeone}. Let $\delta<\alpha/\ell$. As $\delta<g$ one computes that for $i\in \mathbb N$:
\bee
&&\pa_r^i\left[\frac{1}{\mu^{\gamma}}\Lambda Q_b\left( \frac{r}{\mu}\right)-\Lambda Q_b(r)\right]\\
&=& \int_1^{\mu}\frac{d\tilde{\mu}}{\tilde{\mu}}\frac{1}{(\sqrt b)^{\frac{2}{p-1}}\tilde \mu^{\gamma}(\sqrt b \tilde \mu)^i}\left[\pa_r^i\left(-\Lambda^2 Q-\alpha\Lambda Q\right)\right]\left(\frac{r}{\sqrt b\mu}\right) \\
&= & \int_1^{\mu} \frac{d\tilde{\mu}}{\tilde{\mu}}\frac{1}{(\sqrt b)^{\frac{2}{p-1}}\tilde \mu^{\gamma}(\sqrt b \tilde \mu)^i}O\left(\left(1+\frac{r}{\sqrt b \mu}\right)^{-\gamma-\delta-i}\right)= O\left( \mu^{\delta} \frac{(\sqrt b)^{\alpha+\delta}}{r^{\gamma+\delta+i}}\right) ,
\eee
\bea
\non &&\pa_r^i\la z \ra \pa_z\left[\frac{1}{\mu^{\gamma}}\Lambda Q_b\left( \frac{r}{\mu}\right)-\Lambda Q_b(r)\right]\\
\non &=&\frac{\la z \ra \pa_z\mu}{\mu}\frac{1}{\mu^{\gamma}(\sqrt b)^{\frac{2}{p-1}}(\mu \sqrt b)^i}\left[\partial_r^i\left(-\Lambda^2 Q-\alpha \Lambda Q\right) \right]\left(\frac{r}{\sqrt b \mu}\right) \\
\lab{bd tildepsi11 paz} &=& O\left( \mu^{\delta} \frac{(\sqrt b)^{\alpha+\delta}}{(\mu \sqrt{b})^{\gamma+\delta+i}+r^{\gamma+\delta+i}}\right) =O\left( \mu^{\delta} \frac{(\sqrt b)^{\alpha+\delta}}{r^{\gamma+\delta+i}}\right) ,
\eea
which can be easily generalized to show that for $j\in \mathbb N$:
$$
\left| \pa_r^i (\la z \ra \pa_z)^j\left[\frac{1}{\mu^{\gamma}}\Lambda Q_b\left( \frac{r}{\mu}\right)-\Lambda Q_b(r)\right]\right| \lesssim  \mu^{\delta} \frac{(\sqrt b)^{\alpha+\delta}}{r^{\gamma+\delta+i}}\lesssim  \mu^{\alpha} \frac{(\sqrt b)^{\alpha+\delta}}{r^{\gamma+\delta+i}}
$$
as $g\leq \alpha$ and $\mu \geq 1/2$. Similarly, as for $j\in \mathbb N$
$$
(\la z \ra \pa_z)^j \left( \frac{\partial_z \nu}{1+\nu}\right)^2=O(\la z \ra^{-2})
$$
there holds the analogue estimate since $\mu^{\delta}\lesssim \la z \ra^{2\ell\delta/\alpha}\lesssim \la z \ra^2$:
\bea
\non &&\left| \pa_r^i (\la z \ra \pa_z)^j\left[ \left( \frac{\partial_z \nu}{1+\nu}\right)^2 \frac{1}{\mu^{\frac{2}{p-1}}}\left(\Lambda^2 Q_b+\alpha\Lambda Q_b\right)\left( \frac{r}{\mu}\right)\right]\right|\\
\non  & \lesssim & \frac{1}{(\sqrt b)^{\frac{2}{p-1}} \la z \ra ^2\mu^{\frac{2}{p-1}}(\sqrt b \mu)^i}O\left(\left(1+\frac{r}{\sqrt b \mu}\right)^{-\gamma-g-i}\right)\\
\lab{bd tildepsi12 paz}  &\lesssim & \frac{(\sqrt b)^{\alpha+\delta}\mu^{\alpha+\delta}}{\la z \ra ^2}\frac{1}{(\sqrt b\mu)^{\gamma+\delta+i}+r^{\gamma+\delta+i}}\lesssim \frac{\mu^{\alpha+\delta}}{\la z \ra^2} \frac{(\sqrt b)^{\alpha+\delta}}{r^{-\gamma-\delta-i}}\lesssim \mu^{\alpha} \frac{(\sqrt b)^{\alpha+\delta}}{r^{-\gamma-\delta-i}}.
\eea
From \eqref{psitildeone} and the above bounds one infers that for $j\in \mathbb N$:
\bea 
\non \int \frac{|(\la z\ra \pa_z)^j\Psit_1|^{2}}{1+z^{4(\ell+\nu)+1}}\rho_rdY &\lesssim & (\sqrt{b})^{2\alpha+2\delta}\int \frac{\mu ^{2\alpha}}{1+z^{4(\ell+\nu)+1}}dz\int \frac{1}{r^{-2\gamma-2\delta}}r^{d-1}\rho_rdr  \\
\lab{bd tildepsi1z} &\lesssim & (\sqrt{b})^{2\alpha+2\delta}\int \frac{(1+aP_{2\ell}(z))^2}{1+z^{4(\ell+\nu)+1}}dz
\lesssim (\sqrt{b})^{2\alpha+2\delta}
\eea
and for $i\in \mathbb N$ for $b$ small enough
\bea 
\non &&\int_{r\geq r_0} \frac{|\pa_r^i(\la z\ra \pa_z)^j\Psit_1|^{2}}{1+z^{4(\ell+\nu)+1}}\rho_rdY \\
\lab{bd tildepsi1r} &\lesssim  &(\sqrt{b})^{2\alpha+2\delta}\int \frac{\mu ^{2\alpha}}{1+z^{4(\ell+\nu)+1}}dz\int \frac{1}{r^{-2\gamma-2\delta+2i}}r^{d-1}\rho_rdr \lesssim (\sqrt{b})^{2\alpha+2\delta}.
\eea

\noindent{\bf step 2} Control of $\Psit_3$. One first computes for the first term in \fref{defpsithree} that for $i\in \mathbb N$,
\bee
\partial_r^i \Lambda_r \Phi_{a,b}&=&\frac{1}{(\sqrt b\mu)^{\frac{2}{p-1}+i}}\left(\partial_r^i\Lambda Q\right) \left( \frac{r}{\sqrt{b}\mu}  \right)=\frac{1}{(\sqrt b\mu)^{\frac{2}{p-1}+i}}O\left( \left(1+ \left( \frac{r}{\sqrt{b}\mu}  \right)\right)^{-\gamma-i}\right)\\
&=& O\left(\frac{(\sqrt b)^{\alpha}\mu^{\alpha}}{r^{\gamma+i}} \right),
\eee
\bee
\partial_r^i \la z \ra \pa_z \Lambda_r \Phi_{a,b}&=&-\frac{\la z \ra \mu_z}{\mu}\frac{1}{(\sqrt b\mu)^{\frac{2}{p-1}+i}}\left(\partial_r^i\Lambda^2 Q\right) \left( \frac{r}{\sqrt{b}\mu}  \right)\\
&=&\frac{1}{(\sqrt b\mu)^{\frac{2}{p-1}+i}}O\left( \left(1+ \left( \frac{r}{\sqrt{b}\mu}  \right)\right)^{-\gamma-i}\right)= O\left(\frac{(\sqrt b)^{\alpha}\mu^{\alpha}}{r^{\gamma+i}} \right),
\eee
which can easily be generalised to produce for $j\in \mathbb N$
$$
\left| \partial_r^i (\la z \ra \pa_z)^j \Lambda_r \Phi_{a,b} \right| \lesssim \frac{(\sqrt b)^{\alpha}\mu^{\alpha}}{r^{\gamma+i}}.
$$
Also, from \fref{estparameters}, \eqref{lemmaeqmodulation}, \eqref{controlponctuelenergy} and \fref{bd bjk bootstrap} one infers that for $j\in \mathbb N$
\be \lab{bd paz pataunu}
(\la z \ra \pa_z)^j \left(\frac{\pa_{\tau}\nu}{1+\nu}\right) = a_{\tau}(\la z \ra \pa_z)^j \left( \frac{P_{2\ell}(z)}{1+aP_{2\ell}(z)}\right)=O(|a_\tau|)=O((\sqrt b)^{\eta}).
\ee
and from \eqref{defgibbg}, \eqref{defmunu}:
\be
\label{cnbeneonoen}
\left|\frac12\left(\frac{-b_\tau}{b}+1\right)-\frac{\ell}{\alpha}\right|=|B|\lesssim|\pa_\tau\tilde{b}|+\frac{|\pa_{\tau}b_{\ell,0}+\l_{\ell,0}b_{\ell,0}|}{b_{\ell,0}}\lesssim   (\sqrt{b})^{\eta}.
\ee
We therefore infer from \fref{defpsitthree} that for $j\geq 0$:
\bee
&&\int \frac{1}{1+z^{4(\ell+\nu)+1}}\left|(\la z\ra \pa_z)^j\left[\left[\frac{1}{2}\left(-\frac{b_{\tau}}{b}+1\right)-\frac{\ell}{\alpha}+\frac{1}{2\alpha}\frac{\pa_\tau \nu}{1+\nu} \right]\Lambda_r \Phi_{a,b}\right]\right|^2\rho_rdY\\
&& \lesssim (\sqrt{b})^{2\alpha+2\eta} \int \frac{\mu^{2\alpha}}{1+z^{4(\ell+\nu)+1}}dz\int r^{-2\gamma}r^{d-1}\rho_r dr \lesssim (\sqrt{b})^{2\alpha+2\eta}
\eee
and for $i\geq 1$:
\bee
&&\int_{r\geq r_0} \frac{1}{1+z^{4(\ell+\nu)+1}}\left|\pa_r^i(\la z\ra \pa_z)^j\left[\left[\frac{1}{2}\left(-\frac{b_{\tau}}{b}+1\right)-\frac{\ell}{\alpha}+\frac{1}{2\alpha}\frac{\pa_\tau \nu}{1+\nu} \right]\Lambda_r \Phi_{a,b}\right]\right|^2\rho_rdY\\
&& \lesssim (\sqrt{b})^{2\alpha+2\eta} \int \frac{\mu^{2\alpha}}{1+z^{4(\ell+\nu)+1}}dz\int_{r\geq r_0} r^{-2\gamma-2i}r^{d-1}\rho_r dr \lesssim (\sqrt{b})^{2\alpha+2\eta}.
\eee
If $j\geq 1$ as this is the only term depending on $z$ in $\Psit_3$ from \fref{defpsithree} we conclude that:
\bea
\non && \int \frac{|(\la z\ra \pa_z)^j\Psit_3|^{2}}{1+z^{4(\ell+\nu)+1}}\rho_rdY\\
\lab{bd tildepsi3z 1}&=&\int \frac{\left|(\la z\ra \pa_z)^j\left[\left[\frac{1}{2}\left(-\frac{b_{\tau}}{b}+1\right)-\frac{\ell}{\alpha}+\frac{1}{2\alpha}\frac{\pa_\tau \nu}{1+\nu} \right]\Lambda_r \Phi_{a,b}\right]\right|^2}{1+z^{4(\ell+\nu)+1}}\rho_rdY \lesssim (\sqrt{b})^{2\alpha+2\eta}
\eea
and similarly for $i \geq 1$:
\be \lab{bd tildepsi3r 1} 
\int_{r\geq r_0} \frac{|\pa_r^i(\la z\ra \pa_z)^j\Psit_3|^{2}}{1+z^{4(\ell+\nu)+1}}\rho_rdY \lesssim (\sqrt{b})^{2\alpha+2\eta}.
\ee
We turn to the other terms in \fref{defpsithree}. For $r\geq r_0$ and $b$ small enough, $\partial_r^i\psi_{k,0}=O(r^{2k+4-\gamma})$ for $i=0,1,2$, and since $H_b\psi_{i,0}=\lambda_i \psi_{i,0}$ and $V_b=O(1)$ one deduces
\be \lab{bd rgeqr0 psii}
\partial_r^3 \psi_{i,0}=\partial_r\left[-\frac{d-1}{r}\partial_r+\frac{1}{p-1}+\frac 12 r\partial_r+V_b-\lambda_i\right]\psi_{i,0}=O(r^{2i+5-\gamma}), \ \ r\geq r_0.
\ee
Therefore, using \fref{poitwisepsijk} and \fref{bd pointwise tildephi} one obtains that for $r\in [0,+\infty)$:
$$
|\psi_{\ell,0}|+|\psi_{0,0}|\lesssim \frac{1+r^{2\ell+5}}{r^{\gamma}}, \ \ \left|\psi_{0,0}-\frac{1}{(\sqrt b)^{\gamma}}\Lambda Q \left(\frac{r}{\sqrt b}\right)\right| \lesssim \frac{(\sqrt b)^g}{r^{\gamma}}(1+r^5)
$$
and for $r\in [r_0,+\infty)$ and $i=1,2,3$:
$$
|\partial_r^i\psi_{\ell,0}|+|\partial_r^i\psi_{0,0}|\leq r^{2\ell+5-\gamma}, \ \ \left|\partial_r^i\left(\psi_{0,0}-\frac{1}{(\sqrt b)^{\gamma}}\Lambda Q \left(\frac{r}{\sqrt b}\right)\right)\right| \leq \sqrt{b}^gr^{5-\gamma}.
$$
From \fref{estparameters}, \eqref{lemmaeqmodulation}, \eqref{controlponctuelenergy}, \fref{id lambdaib} and \fref{bd bjk bootstrap} there holds
$$
|\pa_\tau b_{\ell,0}+\l_{\ell,0}b_{\ell,0}|\lesssim (\sqrt b)^{\alpha+\eta}, \ \ \ |\tilde b|+|\tilde \l_\ell|+|\tilde \l_0|\lesssim (\sqrt b)^{\eta}, \ \ \left|\frac{b_\tau}{b} \right|\lesssim 1
$$
since $\eta \ll g$. We then conclude from the three identities above, \fref{poitwisepsijk} and \fref{bound:partialb} that for
\bee
\bar{\Psi}_3&:=& (\pa_\tau b_{\ell,0}+\l_{\ell,0}b_{\ell,0})(\psi_{\ell,0}-\psi_{0,0})-\frac{(\sqrt b)^{\alpha}}{\alpha(1+\tilde b)}(-\ell \tilde b+\tilde{\l}_{\ell}-\tilde{\l}_0)\psi_{0,0}\\
&&-\frac{\ell}{\alpha}(\sqrt b)^{\alpha}\left[\psi_{0,0}-\frac{1}{(\sqrt b)^{\gamma}}\Lambda Q\left(\frac{r}{\sqrt b}\right) \right]+b_{\ell,0}\frac{b_{\tau}}{b}b\partial_b(\psi_{\ell,0}-\psi_{0,0})
\eee
there holds
\bee
&&\int \frac{|\bar \Psi_3|^{2}}{1+z^{4(\ell+\nu)+1}}\rho_rdY \\
& \lesssim & \sqrt{b}^{2\alpha} \int \frac{1}{1+z^{4(\ell+\nu)+1}}\left((\sqrt b)^{2\eta}\frac{1+r^{4\ell+10}}{r^{2\gamma}}+|b\pa_b(\psi_{\ell,0})|^2+|b\pa_b(\psi_{0,0})|^2\right)\rho_rdY\\
& \lesssim & \sqrt{b}^{2\alpha+2\eta}+ \sqrt{b}^{2\alpha+2g} \lesssim  \sqrt{b}^{2\alpha+2\eta}
\eee
and similarly for $i=1,2,3$:
\bee
\int_{r\geq r_0} \frac{|\partial_r^i\bar \Psi_3|^{2}}{1+z^{4(\ell+\nu)+1}}\rho_rdY\lesssim \sqrt{b}^{2\alpha} \int \frac{ (\sqrt{b}^{2\eta}r^{4\ell+10-2\gamma}+\sqrt{b}^{2g}r^{4\ell+10-2\gamma}) }{1+z^{4(\ell+\nu)+1}}\rho_rdY \lesssim  \sqrt{b}^{2\alpha+2\eta}.
\eee

\noindent{\bf step 3} Control of $\Psi_3$. From \fref{defpsithree}, the bounds \fref{bd tildepsi1z}, \fref{bd tildepsi1r}, \fref{bd tildepsi3z 1}, \fref{bd tildepsi3r 1} and the two bounds above imply the desired bounds \fref{boundspsioneone} and \fref{boundspsioneone par}.\\

\noindent{\bf step 4} Proof of \eqref{estlinearterm}. Recall that $\zeta=b_{\ell}(\psi_{\ell,0}-\psi_{0,0})$ and $L(\zeta)=p(\Phi_{a,b}^{p-1}-Q_b^{p-1})\zeta$. We thus infer from \fref{bd bjk bootstrap} and \fref{bd rgeqr0 psii}:
$$
\partial_z \zeta=0, \ \ \text{and} \ \ |\partial_r^i \zeta|\lesssim (\sqrt b)^{\alpha}r^{2\ell+5} \ \  \ \text{for} \ r\geq r_0 \ \ \text{and} \ i=0,...,3.
$$
Also, one computes for $i\in \mathbb N$
\bee
&&\partial_r^i(\Phi_{a,b}^{p-1}-Q_b^{p-1})=(p-1)\int_1^{\mu} \frac{d\tilde \mu}{\tilde \mu} \frac{1}{(\tilde \mu \sqrt b)^{2+i}} [\partial_r^i(Q^{p-2}\Lambda Q)]\left(\frac{r}{\tilde \mu \sqrt b}\right) \\
&=&\int_1^{\mu} \frac{d\tilde \mu}{\tilde \mu}O\left(\frac{\tilde \mu ^{\alpha}(\sqrt b)^{\alpha}}{(\mu\sqrt{b})^{2+\alpha+i}+r^{2+\alpha+i}} \right)=O\left(\frac{\mu^{\alpha}\sqrt{b}^{\alpha}}{(\sqrt{b})^{2+\alpha+i}+r^{2+\alpha+i}}\right)=O\left(\frac{\mu^{\alpha}\sqrt{b}^{\alpha}}{r^{2+\alpha+i}}\right).
\eee
and
$$
|\partial_r^i \la z \ra\partial_z\Phi_{a,b}| = \la z\ra\frac{|\mu_z|}{\mu}\frac{1}{(\sqrt{b}\mu)^{\frac{2}{p-1}+i}}(\partial_r^i\Lambda Q)\left(\frac{r}{\sqrt b \mu}\right) \lesssim \frac{\mu^\alpha(\sqrt{b})^\alpha}{(\sqrt b \mu)^{\gamma+i}+r^{\gamma+i}}
$$
which can easily be generalized to prove that for $j \geq 1 $:
$$
|\partial_r^i (\la z \ra \partial_z)^j \Phi_{a,b}| \lesssim  \frac{\mu^\alpha(\sqrt{b})^\alpha}{(\sqrt b \mu)^{\gamma+i}+r^{\gamma+i}}\lesssim \text{min}\left(\frac{\sqrt{b}^{\alpha}\mu^{\alpha}}{r^{\gamma+i}},\frac{1}{(\sqrt b\mu)^{\frac{2}{p-1}+i}+r^{\frac{2}{p-1}+i}}\right)
$$
implying that
\be \lab{bd paz Phiap-1-Qbp-1}
|\partial_r^i (\la z \ra \partial_z)^j (\Phi_{a,b}^{p-1}-Q_b^{p-1})|=|\partial_r^i (\la z \ra \partial_z)^j (\Phi_{a,b}^{p-1})| \lesssim \frac{\mu^{\alpha}\sqrt{b}^{\alpha}}{(\sqrt{b})^{2+\alpha+i}+r^{2+\alpha+i}}\lesssim \frac{\sqrt{b}^{\alpha}\mu^{\alpha}}{r^{2+\alpha+i}}.
\ee
From the above estimates we infer that for $i\geq 1$ and $j\in \mathbb N$:
\bee
\int_{r\geq r_0} \frac{|\partial_r^i(\la z \ra \pa_z)^jL(\zeta)|^{2}}{1+z^{4(\ell+\nu)+1}}\rho_rdY\lesssim \int_{r\geq r_0} (\sqrt b)^{\alpha}\frac{\mu^{\alpha}}{1+z^{4(\ell+\nu)+1}}r^{2\ell+5}\rho_rdY\lesssim (\sqrt b)^{4\alpha}
\eee
which proves \fref{estlinearterm2}. We also infer that for $j\in \mathbb N$:
\be
\label{cnkecneoneo}
|(\la z\ra \pa_z)^j L(\zeta)|\lesssim \frac{(\mu\sqrt{b})^\alpha}{r^{\alpha+2}}|\zeta|.
\ee 
Hence from \eqref{estzeta}, \eqref{estfumdamental}, \fref{confitiondiemsnion} as $g\leq 2$:
\bee
&&\int \frac{|(\la z\ra \pa_z)^j L(\zeta)|^2}{(1+z^{4\ell(1+\nu)+1})}dY\lesssim \int \frac{dz}{\la z\ra^{1+4\nu}}\int \frac{(\sqrt b)^{2\alpha}\zeta^2}{r^{2(\alpha+2)}}r^{d-1}\rho_rdr \\
&\lesssim& (\sqrt{b})^{2\alpha} \int_{r\leq \sqrt{b}} \frac{(\sqrt{b})^{2g}}{r^{\frac{4}{p-1}}}\frac{1}{r^{(2\alpha+4)}}r^{d-1}dr\\
&+&(\sqrt b)^{2\alpha} \int_{\sqrt{b}\leq r}(\sqrt{b})^{2\alpha}\frac{r^4+(\sqrt b)^{2g}}{r^{2\gamma+2\alpha+4}}(1+r)^{2\gamma+4\ell+4}\rho_rr^{d-1}dr \lesssim  (\sqrt b)^{2\alpha+2g}
\eee
and \eqref{estlinearterm} is proved. 
\end{proof}

We now turn to $W^{1,2q+2}$ near the origin.

\begin{lemma}[$W^{1,2q+2}$ bounds on $\Psi_3$] There hold the following estimates for some universal $c,C>0$:
\bea
\label{ceibbeokboeoebv}
\int_{r\leq 2}\left(r^{2\frac{2q+1}{2q+2}}\frac{\Psi_3}{T}\right)^{2q+2}\frac{dY}{\la z\ra(1+D^{2Kq})}\lesssim (\sqrt{b})^{c\eta q}\\
\label{ceibbeokboeoebvbis}
\int_{r\leq 2} \left[r^{2\frac{2q+1}{2q+2}}\pa_r\left(\frac{\Psi_3}{T}\right)\right]^{2q+2}\frac{dY}{\la z\ra(1+D^{2Kq})} \lesssim \frac{1}{(\sqrt{b})^{C q}}\\
\label{ceibbeokboeoebvbisbis}
\int_{r\leq 2} \left[r^{2\frac{2q+1}{2q+2}}\la z\ra \pa_z\left(\frac{\Psi_3}{T}\right)\right]^{2q+2}\frac{dY}{\la z\ra(1+D^{2Kq})} \lesssim (\sqrt{b})^{c\eta q}
\eea
and similarly:
\bea
\label{ceibbeokboeoebvL}
\int_{r\leq 2}\left(r^{2\frac{2q+1}{2q+2}}\frac{L(\zeta)}{T}\right)^{2q+2}\frac{dY}{\la z\ra(1+D^{2Kq})}\lesssim (\sqrt{b})^{c q}\\
\label{ceibbeokboeoebvbisL}
\int_{r\leq 2} \left[r^{2\frac{2q+1}{2q+2}}\pa_r\left(\frac{L(\zeta)}{T}\right)\right]^{2q+2}\frac{dY}{\la z\ra(1+D^{2Kq})} \lesssim \frac{1}{(\sqrt{b})^{Cq}}\\
\label{ceibbeokboeoebvbisbisL}
\int_{r\leq 2} \left[r^{2\frac{2q+1}{2q+2}}\la z\ra \pa_z\left(\frac{L(\zeta)}{T}\right)\right]^{2q+2}\frac{dY}{\la z\ra(1+D^{2Kq})} \lesssim (\sqrt{b})^{cq}
\eea

\end{lemma}

\begin{proof} 

{\bf step 1} Pointwise bound. We claim the pointwise bound for any $0\leq \delta \leq \eta$ and $r\leq 2$:
\be
\label{tobeprovedepsithree}
\left|\frac{\Psi_3}{Tr}\right|+\left|\pa_r\left(\frac{\Psi_3}{T}\right)\right|\lesssim  \frac{(\sqrt{b})^{\delta}}{r^{1+\delta}}+\frac 1r\left|\begin{array}{ll}\frac{(\sqrt{b})^\delta}{r^\delta\mu^{\alpha-\delta}}\ \ \mbox{for}\ \ r\geq D\\ D^{\frac 2{p-1}}(1+D^{\delta})\left[\frac{(\sqrt{b})^{\alpha+\delta}}{r^{\gamma}+(\sqrt{b})^\gamma}+ \frac{(\sqrt{b})^{\alpha}}{\sqrt{b}^{\gamma+\delta}+r^{\gamma+\delta}}\right]\ \ \mbox{for}\ \ r\leq D
\end{array}\right. 
\ee
\noindent{\it $\Psit_1$ term}. Recall \fref{defpsithree}. We first prove the above bound \fref{tobeprovedepsithree} for $\Psit_1$. We decompose from \eqref{psitildeone}:
$$\frac{\Psit_1}T=G_1+G_2$$ with 
\bea
\label{defgone}
&&G_1=\frac{\ell}{\alpha T} \left[\frac{1}{\mu^{\gamma}}\Lambda Q_b\left(\frac{r}{\mu}\right)-\Lambda Q_b(r)\right],\\
&&\label{defgtwo}  G_2=-\frac{1}{\alpha^2}\left(\frac{\pa_z\nu}{1+\nu}\right)^2\left(\alpha+\frac{\Lambda^2Q_b}{\Lambda Q_b}\right)\left(\frac{r}{D}\right).
\eea

\noindent{\it $G_2$ term}. We estimate in brute force using the asymptotics of $Q$:
$$\left|\alpha+\frac{\Lambda^2Q}{\Lambda Q}(y)\right|\lesssim \frac{1}{1+|y|^g}, \ \ \left|\pa_y\left(\frac{\Lambda^2Q}{\Lambda Q}(y)\right)\right|\lesssim \frac{1}{1+|y|^{1+g}}$$ from which:
\bee
&&|G_2|\lesssim \frac{1}{\la z\ra^2}\frac{D^\delta}{D^\delta+r^\delta}\lesssim (\sqrt{b})^{\delta}\frac{1}{r^\delta \la z\ra}\\
&& |\pa_rG_2|\lesssim  \frac{1}{\la z\ra^2}\frac{1}{D}\frac{1}{1+\left(\frac{r}{D}\right)^{\delta+1}}\lesssim  \frac{1}{a^c\la z\ra^2}\frac{D^\delta}{r^{\delta+1}}\lesssim \frac{1}r(\sqrt{b})^{\delta}\frac{1}{r^\delta \la z\ra}.
\eee
\noindent{\it $G_1$ term}. 
Next we estimate using $\mu\geq \frac 12$:
\bea
\label{cneneneovnoe}
\nonumber &&\left|\frac{1}{\mu^\gamma}\Lambda Q_b\left(\frac{r}{\mu}\right)-\Lambda Q_b(r)\right|\lesssim\int_1^\mu\frac{d\sigma}{\sigma^{\gamma+1}}\left|(\gamma+r\pa_r)\Lambda Q_b\right|\left(\frac{r}{\sigma}\right)\\
\nonumber & \lesssim & \int_1^\mu\frac{d\sigma}{\sigma^{\gamma+1}}\frac{1}{(\sqrt{b})^{\frac 2{p-1}}}\frac{1}{1+\left(\frac{r}{\sigma\sqrt{b}}\right)^{\gamma+\delta}}\lesssim (\sqrt{b})^{\alpha+\delta}\int_1^\mu\frac{d\sigma}{\sigma^{1-\delta}}\frac{1}{(\sigma\sqrt{b})^{\gamma+\delta}+r^{\gamma+\delta}}\\
& \lesssim &  \frac{(\sqrt{b})^{\alpha+\delta}}{\sqrt{b}^{\gamma+\delta}+r^{\gamma+\delta}}\int_1^\mu\frac{d\sigma}{\sigma^{1-\delta}}\lesssim \frac{(\sqrt{b})^{\alpha+\delta}}{\sqrt{b}^{\gamma+\delta}+r^{\gamma+\delta}}\mu^\delta.
\eea
Moreoever, 
\be
\label{cneoneonve}
T=\Lambda \Phi_{a,b}\gtrsim \frac{D^\alpha}{D^\gamma+r^\gamma}.
\ee and hence the pointwise bound for $r\geq D$:
$$\frac{1}{T}\left|\frac{1}{\mu^\gamma}\Lambda Q_b\left(\frac{r}{\mu}\right)-\Lambda Q_b(r)\right|\lesssim \frac{(\sqrt{b})^{\alpha+\delta}}{r^{\gamma+\delta}}\mu^\delta\frac{D^\gamma+r^\gamma}{D^\alpha}\lesssim \frac{(\sqrt{b})^{\alpha+\delta}}{r^{\gamma+\delta}}\mu^\delta\frac{r^\gamma}{D^\alpha} \lesssim  \frac{(\sqrt{b})^\delta}{r^\delta\mu^{\alpha-\delta}}
$$
and for $r\leq D$:
$$\frac{1}{T}\left|\frac{1}{\mu^\gamma}\Lambda Q_b\left(\frac{r}{\mu}\right)-\Lambda Q_b(r)\right|\lesssim \frac{(\sqrt{b})^{\alpha+\delta}}{\sqrt{b}^{\gamma+\delta}+r^{\gamma+\delta}}\mu^\delta D^{\frac{2}{p-1}}\lesssim \frac{(\sqrt{b})^{\alpha}}{\sqrt{b}^{\gamma+\delta}+r^{\gamma+\delta}}D^{\frac{2}{p-1}+\delta}.$$
We now estimate the $\pa_r$ derivative. First:
\bee
&&\left|\pa_r\left[\frac{1}{\mu^\gamma}\Lambda Q_b\left(\frac{r}{\mu}\right)-\Lambda Q_b(r)\right]\right|=\left|\frac{1}{\mu^{\gamma+1}}(\pa_r\Lambda Q_b)\left(\frac{r}{\mu}\right)-\pa_r\Lambda Q_b(r)\right|\\
&\lesssim& \int_1^\mu\frac{d\sigma}{\sigma^{\gamma+2}}\left|(\gamma+1+r\pa_r)(\pa_r\Lambda Q_b)\right|\left(\frac{r}{\sigma}\right)\\
& \lesssim & \int_1^\mu\frac{d\sigma}{\sigma^{\gamma+2}}\frac{1}{(\sqrt{b})^{\frac 2{p-1}+1}}\frac{1}{1+\left(\frac{r}{\sigma\sqrt{b}}\right)^{\gamma+1+\delta}}\lesssim (\sqrt{b})^{\alpha+\delta}\int_1^\mu\frac{d\sigma}{\sigma^{1-\delta}}\frac{1}{(\sigma\sqrt{b})^{\gamma+1+\delta}+r^{\gamma+1+\delta}}\\
& \lesssim &  \frac{(\sqrt{b})^{\alpha+\delta}}{\sqrt{b}^{\gamma+1+\delta}+r^{\gamma+1+\delta}}\int_1^\mu\frac{d\sigma}{\sigma^{1-\delta}}\lesssim \frac{(\sqrt{b})^{\alpha+\delta}}{\sqrt{b}^{\gamma+1+\delta}+r^{\gamma+1+\delta}}\mu^\delta\lesssim \frac 1r \frac{(\sqrt{b})^{\alpha+\delta}}{\sqrt{b}^{\gamma+\delta}+r^{\gamma+\delta}}\mu^\delta
\eee
Moreover, 
\bea
\non &&\frac{\pa_r\Lambda \Phi_{a,b}}{(\Lambda \Phi_{a,b})^2}\lesssim \left(\frac{D^\gamma+r^\gamma}{D^\alpha}\right)^2\frac{1}{DD^{\frac{2}{p-1}}}\pa_r(\Lambda Q_b)\left(\frac{r}{D}\right)\lesssim \left(\frac{D^\gamma+r^\gamma}{D^\alpha}\right)^2\frac{1}{DD^{\frac{2}{p-1}}}\frac{1}{1+\left(\frac{r}{D}\right)^{\gamma+1}}\\
\lab{bd partialrT-1}& \lesssim &  \left(\frac{D^\gamma+r^\gamma}{D^\alpha}\right)^2\frac{D^\alpha}{D^{\gamma+1}+r^{\gamma+1}}\lesssim \frac 1r\frac{D^\gamma+r^{\gamma}}{D^\alpha}
\eea
and hence:
\bee
&&\left|\pa_r\left(\frac{1}{T}\left[\frac{1}{\mu^\gamma}\Lambda Q_b\left(\frac{r}{\mu}\right)-\Lambda Q_b(r)\right]\right)\right|\lesssim\frac 1r \frac{(\sqrt{b})^{\alpha+\delta}}{\sqrt{b}^{\gamma+\delta}+r^{\gamma+\delta}}\mu^\delta\frac{D^\gamma+r^{\gamma}}{D^\alpha}\\
& + & \frac{(\sqrt{b})^{\alpha+\delta}}{\sqrt{b}^{\gamma+\delta}+r^{\gamma+\delta}}\mu^\delta\frac{D^{\gamma}+r^{\gamma}}{rD^\alpha}.
\eee
For $r\ge D$, this yields:
\bee
&&\left|\pa_r\left(\frac{1}{T}\left[\frac{1}{\mu^\gamma}\Lambda Q_b\left(\frac{r}{\mu}\right)-\Lambda Q_b(r)\right]\right)\right|\lesssim\frac 1r \frac{(\sqrt{b})^\delta}{r^\delta\mu^{\alpha-\delta}}
\eee
and for $r\le D$:
\bee
&&\left|\pa_r\left(\frac{1}{T}\left[\frac{1}{\mu^\gamma}\Lambda Q_b\left(\frac{r}{\mu}\right)-\Lambda Q_b(r)\right]\right)\right|\lesssim \frac{1}{r}\frac{D^{\frac{2}{p-1}+\delta}(\sqrt{b})^\alpha}{\sqrt{b}^{\gamma+\delta}+r^{\gamma+\delta}}
\eee

\noindent{\it Eigenvectors terms}. Recall now \fref{defpsithree}, \eqref{defpsitthree}. We now prove the bound \fref{tobeprovedepsithree} for $\Psit_3$. We have from \fref{estparameters}, \eqref{lemmaeqmodulation}, \eqref{controlponctuelenergy} and \fref{bd bjk bootstrap}:
$$
\left|\frac12 \left(-\frac{b_\tau}{b}+1\right)-\frac{\ell}{\alpha}-\frac1{2\alpha}\frac{\pa_\tau\nu}{1+\nu}\right|\lesssim (\sqrt{b})^{\eta}, \ \ |\pa_\tau b_{\ell,0}+\l_{\ell,0}b_{\ell,0}|\lesssim (\sqrt{b})^{\alpha+\eta}, \ \ |\tilde b|+|\tilde \l_{\ell}|+|\tilde \l_0|\lesssim (\sqrt b)^{\eta}
$$
from what we infer using \fref{poitwisepsijk}, \eqref{cneoneonve} and \fref{bd partialrT-1}:
\bee
&&\left| \frac{1}{T}(\pa_\tau b_{\ell,0}+\l_{\ell,0}b_{\ell,0})(\psi_{\ell,0}-\psi_{0,0})-\frac{(\sqrt{b})^\alpha [-\ell \tilde b+\tilde \l_\ell-\tilde \l _0] }{\alpha (1+\tilde b)}\psi_{0,0}\right| \\
&&+r\left| \partial_r\left[\frac{1}{T}\left((\pa_\tau b_{\ell,0}+\l_{\ell,0}b_{\ell,0})(\psi_{\ell,0}-\psi_{0,0})-\frac{(\sqrt{b})^\alpha [-\ell \tilde b+\tilde \l_\ell-\tilde \l _0]}{\alpha (1+\tilde b)} \psi_{0,0}\right)\right]\right| \\
&\lesssim &  \frac{D^\gamma+r^\gamma}{D^\alpha}\frac{\sqrt{b}^{\alpha+\eta}}{(\sqrt{b})^\gamma+r^\gamma} \lesssim \left|\begin{array}{ll}\frac{(\sqrt{b})^\eta}{\mu^\alpha}\ \ \mbox{for}\ \ r\geq D\\ \frac{D^{\frac 2{p-1}}(\sqrt{b})^{\alpha+\eta}}{r^{\gamma}+(\sqrt{b})^\gamma}\ \ \mbox{for}\ \ r\leq D.\end{array}\right.
\eee
Since Using \fref{bd pointwise tildephi}, \eqref{cneoneonve} and \fref{bd partialrT-1} one has:
\bee
&& \frac{(\sqrt{b})^{\alpha}}T\left| \psi_{0,0}-\frac{1}{\sqrt{b}^\gamma}\Lambda Q\left(\frac{r}{\sqrt{b}}\right)\right|+r\left|\pa_r\left(\frac{(\sqrt b)^{\alpha}}{T}\left( \psi_{0,0}-\frac{1}{\sqrt{b}^\gamma}\Lambda Q\left(\frac{r}{\sqrt{b}}\right)\right)\right)\right|  \\
&=& \frac{(\sqrt{b})^{\alpha}}T\left| \tilde \phi_0\right|+r\left|\pa_r\left(\frac{(\sqrt b)^{\alpha}}{T}\left( \tilde \phi_0\right)\right)\right|  \\
&\lesssim &\frac{D^\gamma+r^\gamma}{D^\alpha}\frac{\sqrt{b}^{\alpha+g}}{r^\gamma+(\sqrt{b})^\gamma} \lesssim \left|\begin{array}{ll}\frac{(\sqrt{b})^g}{\mu^\alpha}\ \ \mbox{for}\ \ r\geq D,\\ \frac{D^{\frac 2{p-1}}(\sqrt{b})^{\alpha+g}}{r^{\gamma}+(\sqrt{b})^\gamma}\ \ \mbox{for}\ \ r\leq D.\end{array}\right.
\eee
Finally, from \fref{poitwisepsijk}, \eqref{cneoneonve} and \fref{bd partialrT-1}, since $(\sqrt b)^{\delta}\lesssim D^{\delta}$:
\bee
&& \frac{(\sqrt{b})^{\alpha}}T\left| b_{\ell,0}\frac{b_\tau}{b}b\pa_b(\psi_{\ell,0}-\psi_{0,0})\right|+r\left|\pa_r\left(\frac{1}{T}\left(  b_{\ell,0}\frac{b_\tau}{b}b\pa_b(\psi_{\ell,0}-\psi_{0,0})\right)\right)\right|  \\
&\lesssim &\frac{D^\gamma+r^\gamma}{D^\alpha}\frac{\sqrt{b}^{\alpha+g}}{r^{\gamma+g}+(\sqrt{b})^{\gamma+g}} \lesssim \left|\begin{array}{ll}\frac{(\sqrt{b})^{\delta}}{r^{\delta}\mu^\alpha}\ \ \mbox{for}\ \ r\geq D,\\ \frac{D^{\frac 2{p-1}+\delta}(\sqrt{b})^{\alpha}}{r^{\gamma+\delta}+(\sqrt{b})^{\gamma+\delta}}\ \ \mbox{for}\ \ r\leq D.\end{array}\right.
\eee
\noindent{\bf step 2} Proof of \eqref{ceibbeokboeoebv}, \eqref{ceibbeokboeoebvbis}. We use the bound \eqref{tobeprovedepsithree} to prove \eqref{ceibbeokboeoebv}, noticing that $|\psi_3|/T\lesssim |\psi_3|/(Tr)$ as $r\leq 2$. Indeed, for any $0< \delta\leq \eta$ for $q$ large enough:
\bee
\int_{r\le2}\left(r^{2\frac{2q+1}{2q+2}}\frac{(\sqrt{b})^{\delta}}{r^{1+\delta}}\right)^{2q+2}\frac{dY}{\la z\ra(1+D^{2Kq})}\lesssim b^{cq\delta}.
\eee
For the second term, we split the integral. First:
\bee
\int_{D\leq r\leq 2}\left(r^{2\frac{2q+1}{2q+2}}\frac{(\sqrt{b})^\delta}{\mu^{\alpha-\delta}r^{1+\delta}}\right)^{2q+2}\frac{dY}{\la z\ra(1+D^{2Kq})}\lesssim  b^{cq\delta}
\eee
and then for the second term, we estimate for $K$ universal large enough, since $r\leq 2$:
\bee
&&\int_{r\leq \min\{2,D\}}\left(r^{2\frac{2q+1}{2q+2}}\frac{D^{\frac 2{p-1}}(1+D^{\delta})(\sqrt{b})^{\alpha+\delta}}{r(r^{\gamma}+(\sqrt{b})^\gamma)}\right)^{2q+2}\frac{dY}{\la z\ra(1+D^{2Kq})}\\
& \lesssim & \int_{r\leq \sqrt{b}}\left((\sqrt{b})^{2\frac{2q+1}{2q+2}-1-\frac{2}{p-1}+\delta}\right)^{2q+2}\frac{dY}{\la z\ra(1+D^{Kq})}\\
&+& \int_{\sqrt{b}\leq r\leq 2}\left(r^{2\frac{2q+1}{2q+2}-1-\frac{2}{p-1}}(\sqrt{b})^\delta\right)^{2q+2}\frac{dY}{\la z\ra(1+D^{Kq})}\\
& \lesssim & \int_{r\leq \sqrt{b}}(\sqrt{b})^{\delta(2q+2)-2}r^{d-1}\frac{dzdr}{\la z\ra(1+D^{Kq})}\\
&+& \int_{\sqrt{b}\leq r\leq 2} r^{-2}(\sqrt{b})^{\delta(2q+2)}r^{d-1}\frac{dzdr}{\la z\ra(1+D^{Kq})}\lesssim b^{c\delta q},
\eee
where we used $p\geq 3$ and $q\gg1$.
Similarly for the last term in \eqref{tobeprovedepsithree}, for $\psi_3/T$:
\bee
&&\int_{r\leq \min\{2,D\}}\left(r^{2\frac{2q+1}{2q+2}}\frac{(\sqrt{b})^{\alpha}}{[\sqrt{b}^{\gamma+\delta}+r^{\gamma+\delta}]}D^{\frac{2}{p-1}}(1+D^\delta)\right)^{2q+2}\frac{dY}{\la z\ra(1+D^{2Kq})}\\
& \lesssim &  \int_{r\leq \sqrt{b}}\left((\sqrt{b})^{2\frac{2q+1}{2q+2}-\frac{2}{p-1}-\delta}\right)^{2q+2}\frac{dY}{\la z\ra(1+D^{qK})}\\
&+ & \int_{\sqrt{b}\leq r\leq 2}\left(r^{2\frac{2q+1}{2q+2}-\frac{2}{p-1}-2\delta}(\sqrt{b})^\delta\right)^{2q+2}\frac{dY}{\la z\ra(1+D^{qK})}\ \lesssim b^{c\delta q}, 
\eee
and for $\pa_r(\psi_3/T)$:
\bee
&&\int_{r\leq \min\{2,D\}}\left(r^{2\frac{2q+1}{2q+2}}\frac{(\sqrt{b})^{\alpha}}{r[\sqrt{b}^{\gamma+\delta}+r^{\gamma+\delta}]}D^{\frac{2}{p-1}}(1+D^\delta)\right)^{2q+2}\frac{dY}{\la z\ra(1+D^{2Kq})}\\
& \lesssim &  \int_{r\leq \sqrt{b}}\left((\sqrt{b})^{2\frac{2q+1}{2q+2}-1-\frac{2}{p-1}-\delta}\right)^{2q+2}\frac{dY}{\la z\ra(1+D^{qK})}\\
&+ & \int_{\sqrt{b}\leq r\leq 2}\left(r^{2\frac{2q+1}{2q+2}-\frac{2}{p-1}-1-2\delta}(\sqrt{b})^\delta\right)^{2q+2}\frac{dY}{\la z\ra(1+D^{qK})}\ \lesssim \frac{1}{b^{C q}}, 
\eee
This concludes the proof of \eqref{ceibbeokboeoebv}, \eqref{ceibbeokboeoebvbis}.\\

\noindent{\bf step 3} $\pa_z$ derivative. We turn to the proof of \eqref{ceibbeokboeoebvbisbis}. 
We first claim the pointwise bound for any $0\leq \delta \leq \eta$:
\be
\label{pointwisepaz}
|\la z\ra\pa_z\Psi_3|\lesssim  \frac{(\sqrt{b})^{\delta}}{r^{\delta}}+\left|\begin{array}{ll}\frac{(\sqrt{b})^\delta}{r^\delta\mu^{\alpha-\delta}}\ \ \mbox{for}\ \ r\geq D\\ D^{\frac 2{p-1}}(1+D^{\delta})\left[\frac{(\sqrt{b})^{\alpha+\delta}}{r^{\gamma}+(\sqrt{b})^\gamma}+ \frac{(\sqrt{b})^{\alpha}}{\sqrt{b}^{\gamma+\delta}+r^{\gamma+\delta}}\right]\ \ \mbox{for}\ \ r\leq D
\end{array}\right. 
\ee
which implies \eqref{ceibbeokboeoebvbisbis} as above.

\noindent{$G_1$ term}. We first estimate:
\bee
|\la z \ra \pa_zT|= \left|\frac{\la z \ra \pa_z\mu}{\mu}\frac{1}{D^{\frac2{p-1}}}\Lambda^2Q\left(\frac{r}{D}\right)\right|\lesssim \frac{1}{ D^{\frac 2{p-1}}}\frac{1}{1+\left(\frac{r}{D}\right)^\gamma}\lesssim \frac{1}{ D^{\frac 2{p-1}}} \frac{D^\gamma}{D^\gamma+r^\gamma}
\eee
and hence using \eqref{cneoneonve}:
\be
\label{estdzderivaive}
\left|\la z\ra \frac{\pa_zT}{T^2}\right|\lesssim \frac{1}{D^{\frac 2{p-1}}}\frac{D^\gamma}{D^\gamma+r^\gamma}\left(\frac{D^\gamma+r^\gamma}{D^\alpha}\right)^2\lesssim \frac{D^\gamma+r^\gamma}{D^\alpha}
\ee
from what we infer using \fref{bd tildepsi11 paz} and $\mu \geq 1/2$
\bee
&&\left|\la z\ra G_1\right|\lesssim \frac{D^\gamma+r^\gamma}{D^\alpha}\frac{(\sqrt{b})^{\alpha+\delta}}{(\mu \sqrt{b})^{\gamma+\delta}+r^{\gamma+\delta}}\mu^\delta\\
& \lesssim & \left|\begin{array}{ll} \frac{(\sqrt{b})^\delta}{r^\delta\mu^{\alpha-\delta}}\ \ \mbox{for}\ \ r\geq D\\\ \ D^{\frac{2}{p-1}}\frac{(\sqrt{b})^{\alpha+\delta}}{\sqrt{b}^{\gamma+\delta}+r^{\gamma+\delta}}\mu^\delta\lesssim \frac{(\sqrt{b})^\alpha}{(\sqrt{b})^{\gamma+\delta}+r^{\gamma+\delta}}D^{\frac2{p-1}+\delta} \ \mbox{for}\ \ r\leq D.\end{array}\right. 
\eee

\noindent{\it $G_2$ term}. We estimate from \fref{bd tildepsi12 paz}, \eqref{cneoneonve} and \fref{estdzderivaive}
\bee
|\la z \ra \pa_z G_2|\lesssim\frac{D^\alpha+r^\gamma}{D^\alpha}\frac{(\sqrt b)^{\alpha+\delta}\mu^{\alpha+\delta}}{\la z \ra^2}\frac{1}{(\sqrt b \mu)^{\gamma+\delta}+r^{\gamma+\delta}}\lesssim \frac{(\sqrt b)^{\delta}}{r^{\delta}}
\eee

\noindent{\it Eigenvectors term}. Finally from \fref{bd paz pataunu}, \fref{cnbeneonoen}:
$$
\left|\la z\ra \pa_z\left(\frac{\Psit_3}{T}\right)\right|\lesssim (\sqrt b)^{\eta}$$ as the other terms do not depend on $z$. This concludes the proof of \eqref{pointwisepaz}.\\

\noindent{\bf step 4} Control of $L(\zeta)$ terms. From the rough bound
$$
|Q_b^{p-1}|+|\Phi_{a,b}|^{p-1}\lesssim \text{min}\left(\frac{1}{r^2},\frac{1}{b}\right),
$$
the bounds \fref{bd paz Phiap-1-Qbp-1}, \eqref{estzeta}, \fref{cneoneonve}, \fref{estdzderivaive} we infer for $j=0,1$ and $r\leq 2$:
\bee
\left|(\la z \ra \pa_z)^j\left(\frac{L(\zeta)}{T}\right) \right| & \lesssim & \frac{D^\gamma+r^\gamma}{D^\alpha}\text{min}\left(\frac{D^{\alpha}}{(\sqrt b)^{2+\alpha}+r^{2+\alpha}},\frac{1}{r^2},\frac{1}{b}\right)\frac{(\sqrt{b})^{\alpha}(r^2+(\sqrt{b})^g)}{(\sqrt b)^{\gamma} +r^{\gamma}} \\
&\lesssim&\left|\begin{array}{lll} D^{\frac{2}{p-1}}(\sqrt b)^{g-2-\frac{2}{p-1}}\ \ \mbox{for}\ \ r\leq \sqrt b\\ D^{\frac{2}{p-1}}\left(\frac{(\sqrt{b})^{\alpha}}{r^{\gamma}}+\frac{(\sqrt{b})^{\alpha+g}}{r^{\gamma+2}}\right)\ \ \mbox{for}\ \ \sqrt{b}\leq r\leq D\\  \frac{(\sqrt b)^{\alpha}}{r^\alpha}+\frac{\sqrt{b}^{\alpha+g}}{r^{2+\alpha}}\ \ \mbox{for}\ \ r\ge D\end{array}\right.\\
\eee
One then computes since $p\geq 3$ and $g\geq 3/2$ that
\bee
&&\int_{r\leq \sqrt b}\left(r^{2\frac{2q+1}{2q+2}}D^{\frac{2}{p-1}}(\sqrt b)^{g-2-\frac{2}{p-1}}\right)^{2q+2}\frac{dY}{\la z\ra(1+D^{2Kq})}\\
& \lesssim &  (\sqrt{b})^{2(2q+1)+d+(g-2-\frac{2}{p-1})(2q+2)-1}\lesssim (\sqrt b)^{q-2},
\eee
\bee
&&\int_{\sqrt b\leq r\leq 2}\left(r^{2\frac{2q+1}{2q+2}}D^{\frac{2}{p-1}}\left(\frac{(\sqrt{b})^{\alpha}}{r^{\gamma}}+\frac{(\sqrt{b})^{\alpha+g}}{r^{\gamma+2}}\right)\right)^{2q+2}\frac{dY}{\la z\ra(1+D^{2Kq})}\\
& \lesssim &  (\sqrt{b})^{2(2q+1)-\gamma (2q+2)+d+\alpha (2q+2)-1}+(\sqrt{b})^{2(2q+1)-(\gamma+2) (2q+2)+d+(\alpha+g) (2q+2)-1}\\
& \lesssim &  (\sqrt{b})^{2(2q+1)-\frac{2}{p-1} (2q+2)+d-1}+(\sqrt{b})^{2(2q+1)-(\frac{2}{p-1}+2-g) (2q+2)+d-1}\lesssim (\sqrt b)^{q-2}
\eee
and that 
\bee
&&\int_{D\leq r\leq 2}\left(r^{2\frac{2q+1}{2q+2}}\left(\frac{(\sqrt b)^{\alpha}}{r^\alpha}+\frac{\sqrt{b}^{\alpha+g}}{r^{2+\alpha}}\right)\right)^{2q+2}\frac{dY}{\la z\ra(1+D^{2Kq})}\\
& \lesssim &  (\sqrt{b})^{2(2q+1)-\alpha (2q+2)+d+\alpha (2q+2)-1}+(\sqrt{b})^{2(2q+1)-(\alpha+2) (2q+2)+d+(\alpha+g) (2q+2)-1}\lesssim (\sqrt b)^{q}
\eee
yielding \eqref{ceibbeokboeoebvL}, \eqref{ceibbeokboeoebvbisbisL}. The estimate \eqref{ceibbeokboeoebvbisL} can be proven with the same arguments and is left to the reader. 
\end{proof}

%%%%%%%%%%%%%%%%%%%%%%%%%%%%%%%%%%%%%%%%%%%%%%%%%%%%%%%%%%%%%%%%%%%%%%

\end{appendix}

\end{document}